\documentclass[JournalSubmission]{CustomTemplate}
\RequirePackage{amsthm,amsmath,amsfonts,amssymb}
\RequirePackage[numbers,sort]{natbib}
\usepackage{bbm}
\usepackage{bm}
\usepackage{stmaryrd}
\usepackage{mathtools}
\numberwithin{equation}{section}
\usepackage{mathrsfs}

\usepackage{pgfplots}
\pgfplotsset{compat=1.18}
\usepackage[hidelinks]{hyperref}
\usepackage{breakurl}
\usepackage[stretch=10]{microtype}
\usepackage{float}
\RequirePackage{graphicx}
\usepackage{graphicx}
\usepackage{url}
\usepackage{scalerel}
\usepackage{enumitem}
\usepackage{subcaption}
\usepackage{todonotes}
\usepackage{cleveref}
\crefformat{equation}{(#2#1#3)}
\usepackage[normalem]{ulem}

\startlocaldefs
\theoremstyle{plain}
\newtheorem{theorem}{Theorem}[section]
\newtheorem{corollary}[theorem]{Corollary}
\newtheorem{lemma}[theorem]{Lemma}
\newtheorem{proposition}[theorem]{Proposition}

\theoremstyle{remark}
\newtheorem{remark}[theorem]{Remark}
\newtheorem*{remark*}{Remark}

\newcommand{\eg}{{e.g.,}\ }
\newcommand{\ie}{{i.e.,}\ }
\newcommand{\nref}[2]{\hyperref[#1]{#2}} 

\newcommand{\negsp}{\hspace{-0.25em}}

\newcommand*{\doii}[1]{\url{http://doi.org/#1}}

\newcommand{\euler}{{\rm{e}}}
\renewcommand{\i}{{\rm{i}}}
\renewcommand{\Im}{\operatorname{Im}}
\renewcommand{\Re}{\operatorname{Re}}
\newcommand{\de}{\vcentcolon=}
\newcommand{\ed}{=\vcentcolon}
\renewcommand{\inf}{\mathop{\mathrm{inf}\vphantom{\mathrm{sup}}}}

\newcommand{\intd}{{\rm d}}
\newcommand{\dd}[1]{\frac{\rm{d}}{{\rm{d}}#1}}

\newcommand{\TV}{{\rm{TV}}}

\newcommand{\tmix}{t_{\rm{mix}}}
\newcommand{\Unif}[1]{\operatorname{Unif}#1}
\newcommand{\Cov}{\operatorname{Cov}}
\newcommand{\Var}{\operatorname{Var}}

\newcommand{\Linf}{L^\infty}

\newcommand{\sort}{\operatorname{sort}}

\newcommand{\Csa}{\mathbb{C}_{\rm{sa}}}
\newcommand{\tr}{\operatorname{tr}}
\newcommand{\Tr}{\operatorname{Tr}}
\newcommand{\transpose}{{\rm T}}

\newcommand{\op}{}

\newcommand{\defbold} [1]{\expandafter\newcommand\csname b#1\endcsname{\mathbf{#1}}}
\newcommand{\defcal} [1]{\expandafter\newcommand\csname c#1\endcsname{\mathcal{#1}}}
\newcommand{\defbb} [1]{\expandafter\newcommand\csname bb#1\endcsname{\mathbb{#1}}}
\newcommand{\defscr} [1]{\expandafter\newcommand\csname s#1\endcsname{\mathscr{#1}}}
\newcommand{\deffrak} [1]{\expandafter\newcommand\csname f#1\endcsname{\mathfrak{#1}}}
\usepackage{forloop}
\newcounter{ct}
\forLoop{1}{26}{ct}{
    \edef\letter{\Alph{ct}}
    \expandafter\defbold\letter
    \expandafter\defcal\letter
    \expandafter\defbb\letter
    \expandafter\defscr\letter
    \expandafter\deffrak\letter
}
\def\b1{\mathbf{1}}
\def\bb1{\mathbbm{1}}

\newcommand{\sym}{\operatorname{sym}}
\newcommand{\free}{\operatorname{free}}

\newcommand{\brackets}[1]{{\text{\fontsize{5}{5}\selectfont{$($}{\hspace*{-0.3pt}}\fontsize{6}{4}\selectfont$#1$\hspace*{-0.3pt}\fontsize{5}{5}\selectfont{$)$}}}}
\newcommand{\defLetterci} [1]{\expandafter\newcommand\csname #1ci\endcsname{#1^{\brackets{\rho}}}}
\newcommand{\defbLetterci} [1]{\expandafter\newcommand\csname b#1ci\endcsname{\mathbf{#1}^{\brackets{\rho}}}}
\newcommand{\defbbLetterci} [1]{\expandafter\newcommand\csname bb#1ci\endcsname{\mathbb{#1}^{\brackets{\rho}}}}

\usepackage{forloop}
\newcounter{ct2}
\forLoop{1}{26}{ct2}{
    \edef\letter{\Alph{ct2}}
      \expandafter\defLetterci\letter
      \expandafter\defbLetterci\letter
      \expandafter\defbbLetterci\letter
}
\newcommand{\Delci}{\Delta^{\hspace*{-1.5pt}\brackets{\rho}}}
\newcommand{\Delc}[1]{\Delta^{\brackets{#1}}}

\newcommand{\Wc}[1]{W^{\brackets{#1}}}
\newcommand{\bbQc}[1]{\bbQ^{\brackets{#1}}}

\endlocaldefs

\begin{document}

\begin{frontmatter}
	\title{Sharp concentration for sums of matrices with Markovian dependence through universality}

	\begin{aug}
		\author[A,B]{\fnms{Alexander}~\snm{Van Werde}\ead[label=e1]{a.van.werde@uni-muenster.de} \orcid{0000-0002-4670-1836}},
		\author[A]{\fnms{Jaron}~\snm{Sanders}\ead[label=e2]{jaron.sanders@tue.nl} \orcid{0000-0003-0187-2065}}
		\address[A]{Eindhoven University of Technology, Department of Mathematics and Computer Science\vspace{-1em}}
        \address[B]{University of M\"unster, Institut f\"ur Mathematische Stochastik\\ 
        \printead{e1,e2}}
	\end{aug}

	\begin{abstract}
		We prove that a sum of random matrices generated by a $\psi$-mixing Markov chain has similar spectral properties to a Gaussian matrix with the same mean and covariance structure.
		This nonasymptotic universality principle enables sharp concentration inequalities when combined with recent advances in the Gaussian literature.
		We illustrate the theory with examples, showing how it enables polynomial dimensional improvements relative to previous Markovian matrix concentration results when applied to Wigner-type matrices, and how one can recover sharp limiting values for a model used to study spectral clustering techniques.

		A key challenge in the proof is that techniques based only on classical cumulants, which can be used when summands are independent, are not sufficient on their own for efficient estimates in a Markovian setting.
		Our approach exploits Boolean cumulants and a change--of--measure argument.

	\end{abstract}



\end{frontmatter}
\section{Introduction}\label{sec: Introduction}
The study of random matrices is a rich topic in probability theory with numerous deep results and connections to other fields.
Such connections are aided by the phenomenon of \emph{universality} which states that the asymptotic properties of random matrices are remarkably robust to the entries' distribution.
For instance, a classical result is that if $\bW$ is a \emph{Wigner matrix}---a symmetric $d\times d$ matrix with independent and identically distributed entries of mean zero and unit variance---then the operator norm satisfies $\Vert \bW \Vert/\sqrt{d} \to 2$ as $d\to \infty$ given only mild moment assumptions, no matter the specific law of the entries \cite{bai1988necessary,furedi1981eigenvalues}.

Classical random matrix theory is however mostly limited to asymptotic results for homogeneous matrices, such as those with identically distributed entries.
For settings that do not admit an asymptotic formulation or matrices with inhomogeneous structure, the more recent theory of \emph{matrix concentration
	inequalities} has achieved notable success \cite{tropp2015introduction}.
For example, suppose that we are given self-adjoint $d\times d$ matrices $\bX_1,\ldots,\bX_n$ that are bounded $\Vert \bX_i \Vert \leq R$ and centered $\bbE[\bX_i] = 0$.
If these matrices are independent, then the matrix Bernstein inequality of Tropp and Oliveira \cite{tropp2015introduction,oliveira2009concentration} yields a constant $c>0$ such that
\begin{equation}
	\bbE\Biggl[\biggl\Vert \sum_{i=1}^n \bX_i \biggr\Vert \biggr] \leq c \sqrt{\ln(d) \biggl\Vert \sum_{i=1}^n \bbE[\bX_i^2] \biggr\Vert} + c \ln(d) R.\label{eq:YoungRag}
\end{equation}
Being applicable to any matrix that can be decomposed into independent bounded summands, this result is quite flexible.
It covers classical models as a special case, but also permits settings where entries are neither independent nor identically distributed.

Moreover, results with dependent summands have also been developed, further enhancing the theory's flexibility.
For example,  Neeman, Shi, and Ward \cite{neeman2023concentration} recently achieved a variant of \eqref{eq:YoungRag} with Markovian dependence.
Assume that there exists a stationary Markov chain $Z_1,\ldots,Z_n$ with values in some state space $\cZ$ as well as matrix valued functions $\bF_i:\cZ\to \bbC^{d\times d}$ such that the self-adjoint matrices can be expressed as $\bX_i = \bF_i(Z_i)$ for every $i\geq 1$.
Then, the Markovian matrix Bernstein inequality in \cite{neeman2023concentration} implies that there exists an absolute constant $c>0$ with
\begin{equation}
	\bbE\biggl[\biggl\Vert \sum_{i=1}^n \bX_i \biggr\Vert \biggr] \leq c  \sqrt{\frac{\ln(d)}{1-\lambda} \sum_{i=1}^n \biggl\Vert \bbE[\bX_i^2] \biggr\Vert}  + \frac{c}{1-\lambda} \ln(d) R.\label{eq:IcyCrow}
\end{equation}
Here, $1-\lambda$ is the \emph{absolute spectral gap} of the Markov chain $Z$ and quantifies its decay of dependence; see \cite[Definition 3.1]{neeman2023concentration}.

Such results with dependent summands have significant practical relevance as these arise in numerous applications.
For instance, sums of dependent random matrices generated by a stochastic process are crucial in the analysis of time series \cite{banna2016bernstein,han2020moment}, randomness-efficient sampling methods \cite{garg2018matrix}, random walk based graph embedding methods \cite{qiu2020matrix}, convergence rates of Hessian matrices in stochastic gradient descent \cite{neeman2023concentration}, and state space reduction methods in reinforcement learning \cite{sanders2020clustering,jedra2023nearly}.
Matrix concentration inequalities allow analyzing such settings with only minimal model-specific computations, while the dependencies involved would otherwise complicate the analysis.

The flexibility however comes at a cost: most previous results are not sharp in typical applications.
For example, the bound that results if one applies \eqref{eq:YoungRag} to a Wigner matrix is loose by a logarithmic dimensional factor.
Moreover, note that the variance proxy $\Vert \sum_{i} \bbE[\bX_i^2] \Vert$ in \eqref{eq:YoungRag} is replaced by $\frac{1}{1-\lambda}\sum_{i} \Vert \bbE[\bX_i^2] \Vert$ in \eqref{eq:IcyCrow}.
Moving the sum outside the operator norm is insubstantial if summands are identically distributed, but it can be highly inefficient if summands encode different matrix structure as is necessary in some applications.
For instance, if \eqref{eq:IcyCrow} is applied to a Wigner-type matrix, then the resulting bound is loose by a factor of order $\sqrt{\ln(d+1)}\sqrt{d}$; see also Section \ref{sec: ApplicationsMarkov}.
Not only does this give a logarithmic dimensional factor, the suboptimality is here even polynomial in the dimensionality.

\subsection{Universality-based concentration}
With the goal to achieve sharp concentration estimates in dependent settings, we here pursue a relatively new approach.
Our main result, Theorem \ref{thm: MainTracial}, establishes that a sum of matrices generated by a Markov chain has similar spectral properties to a Gaussian matrix with matching covariance structure.
By combining this universality principle with results from the Gaussian literature \cite{bandeira2024matrix,bandeira2021matrix,lust1986inegalites}, we then achieve practical concentration estimates in Corollaries \ref{cor: MarkovMarkovBound} and \ref{cor: Khintchine}.
Crucially, this is done in a flexible nonasymptotic setting allowing nonhomogeneous covariance and dependent summands.

A key advantage of a universality-based approach is that it enables a leading-order term that only depends on the covariance structure of the summed matrix.
Direct dependence on the summands or on the specific structure of the Markov chain only occurs through lower-order terms coming from the approximation error in the universality statement.
In particular, this yields better variance proxies than are used in \cite{garg2018matrix,qiu2020matrix,neeman2023concentration,banna2016bernstein,han2020moment}.
Moreover, by relying on recent advances in the Gaussian literature \cite{bandeira2021matrix,bandeira2024matrix}, we can achieve a sharp estimate of the form
\begin{equation}
	\bbE\biggl[\biggl\Vert \sum_{i=1}^n \bX_i \biggr\Vert \biggr]  \leq  C_{\textnormal{free}} \sqrt{\biggl\Vert  \bbE\biggl[\biggl(\sum_{i=1}^n\bX_i\biggr)^2 \biggr] \biggr\Vert} + \varepsilon \ln(d). \label{eq:WittyArm}
\end{equation}
Here, $1 \leq C_{\textnormal{free}} \leq 2$ is a sharp free-probabilistic constant with explicit dependence on the covariance structure of the summed matrix, and the error term $\varepsilon$ has explicit dependence on the model parameters and is small in typical applications; see Corollary \ref{cor: MarkovMarkovBound}.

Universality hence enables sharp results that are inaccessible though previous results with dependencies.
We achieve an optimal constant on the main term that, in particular, removes the often-suboptimal logarithmic factor if $\varepsilon$ is sufficiently small; note that $C_{\free}\leq 2$.
Moreover, even if one does not care about constants or logarithms, the main term in \eqref{eq:WittyArm} has a natural variance proxy that can yield significant improvements relative to results like \eqref{eq:IcyCrow} in at least two ways.
First, by having the sum inside the operator norm, the variance proxy $\Vert \bbE[(\sum_{i}\bX_i)^2] \Vert$ can enable polynomial dimensional improvements for applications where the summands have different matrix structure.
Second, moving the sum inside the square optimally\footnote{Jensen's inequality yields that $\bbE[\Vert \sum_i \bX_i \Vert^2] \geq \Vert \bbE[ (\sum_i\bX_i)^2] \Vert$.
	This shows that the leading term in \eqref{eq:WittyArm} is optimal up to the constant $C_{\textnormal{free}}\leq 2$, at least if the $L^1$ norm is replaced by the $L^2$ norm.
	Moreover, Corollary \ref{cor: MarkovMarkovBound} gives two-sided bounds on $L^p$ norms that show that the constant $C_{\textnormal{free}}$ is also optimal when $p\approx \ln(d)$.
	One can make it rigorous that the $L^1$ norm may be replaced by an $L^p$ norm conditional on a concentration--of--measure ingredient.
	We presently however only have that ingredient in special cases; see the discussion after \eqref{eq: MarkovMarkovBound}.
} incorporates the dependence between summands in the way that it naturally appears in the summed matrix.
In particular, this is more efficient than worst-case variance proxies of the form $D\times \Vert \bbE[\sum_{i}\bX_i^2] \Vert$ with $D$ a dependence coefficient for the Markov chain, like one based on a spectral gap; see also the discussion after Lemma \ref{lem: QualitativeParamEstimates}.

Examples illustrating the aforementioned features are given in Section \ref{sec: Applications}.
We briefly consider matrices that are filled entry-wise by a Markov chain in Section \ref{sec: ApplicationsMarkov}, demonstrating that our results recover the optimal asymptotic order for Wigner-type matrices.
Next, Section \ref{sec: ApplicationsBMC} considers an application to \emph{block Markov chains}.
The latter are a model for Markov chains with a cluster structure in the state space, and the properties of associated random matrices are crucial in the analysis of spectral clustering algorithms that recover the clusters based on an observed sample path \cite{sanders2020clustering,sanders2023spectral,zhang2020spectral,jedra2023nearly}.
We use our general-purpose nonasymptotic theory to achieve improvements of results that had been derived using model-specific asymptotic analysis in \cite{vanwerde2023singular,sanders2021spectral}; see Theorem \ref{thm: BMC_Noise_Free} and Proposition \ref{prop: SingValN}.

\subsection{Related work}\label{sec: Related}
Universality in classical random matrix theory has attracted significant attention, including efforts to relax the independence and homogeneity assumptions \cite{erdHos2019random,ajanki2019stability,sodin2010spectral}.
However, most previous results adopt a semiclassical perspective where the departure from the classical regime remains manageable.
For instance, the asymptotic results in \cite{erdHos2019random} allow non-identically distributed and dependent entries but adopt mean-field conditions limiting the amount of nonhomogeneity in the variance profile of the matrix as well as assumptions that impose decaying correlations; see \cite[Assumptions C to E]{erdHos2019random}.

That universality remains valid in the nonasymptotic and nonhomogeneous framework of matrix concentration results is not explained by such semiclassical results.
This is a more recent development due to Brailovskaya and Van Handel \cite{brailovskaya2022universality} who considered a setting with independent summands, and established that universality arises in such models through an operator-theoretic mechanism.
Their work served as an inspiration for our investigations.
As discussed in Section \ref{sec: IntroProofTech}, however, the extension from independent summands to a Markovian setting is nontrivial due to the dependence involved.
Our technical contribution is to identify the appropriate proof techniques to handle the dependence.

In the previous literature on matrix concentration inequalities with dependence, one can distinguish the following models: Markovian \cite{neeman2023concentration,garg2018matrix,qiu2020matrix,raohoeffding}, $\beta$ or $\tau$-mixing stochastic processes that can be non-Markovian \cite{banna2016bernstein,han2020moment}, martingales \cite{oliveira2009concentration,bacry2018concentration,tropp2011freedman}, exchangeable pairs \cite{mackey2014matrix,paulin2016efron}, and negative dependence \cite{aoun2020matrix,kyng2018matrix,adamczak2025matrix,kaufman2022scalar}.
None of these previous results can access sharp bounds with a sharp constant on the main term like \eqref{eq:WittyArm}.
Indeed, sharp bounds in the flexible setting of matrix concentration inequalities are a recent development, even for Gaussian matrices \cite{bandeira2021matrix,bandeira2024matrix}.
Specifically, free-probabilistic upper bounds in the Gaussian setting are due to Bandeira, Boedihardjo, and Van Handel \cite{bandeira2021matrix}, and lower bounds that establish the sharpness of \cite{bandeira2021matrix} were recently proven by Bandeira, Cipolloni, Schr{\"o}der, and Van Handel \cite{bandeira2024matrix}.

\pagebreak[3]
To our knowledge, all previous results that are applicable in Markovian settings with general matrix-valued summands involve variance proxies of order $ \geq D \times \sum_{i}\Vert \bbE[\bX_i^2] \Vert$ on the leading term with $D$ a dependence coefficient \cite{neeman2023concentration,garg2018matrix,qiu2020matrix,banna2016bernstein,han2020moment}.\footnote{The results in \cite{neeman2023concentration} are the only ones with this exact shape for the variance proxy, but the proxies used in \cite{garg2018matrix,qiu2020matrix,banna2016bernstein,han2020moment} are at least as large.
	Specifically, \cite{banna2016bernstein,han2020moment} replace $\sum_{i}\Vert \bbE[\bX_i^2] \Vert$ by $\max\{ (n/\#K) \sum_{i\in K}\Vert \bbE[\bX_i^2] \Vert: K\subseteq \{1,\ldots,n \} \}$ and the results in \cite{garg2018matrix,qiu2020matrix} only use the boundedness of the summands' operator norm as an input, without an explicit variance proxy, effectively replacing $\bbE[\bX_i^2]$ by its worst-case bound $\Vert \Vert \bX_i \Vert^2 \Vert_{L^\infty}$.
}
As explained earlier, having the norm inside the sum even results in polynomially suboptimal results in some applications like Wigner-type matrices.
In the special case where $\bX_i = f_i(Z_i) \bA_i$ for scalar-valued functions $f_i$ and deterministic matrices $\bA_i$, a variance proxy  $ D \times \Vert \sum_{i} \bbE[\bX_i^2] \Vert$ that moves the sum inside the norm follows from \cite[Theorem 1.3]{raohoeffding}.
That result yields bounds of the correct order of magnitude for Wigner-type matrices but does not cover general matrix-valued summands and does not give a sharp constant.

Moreover, note that all the aforementioned results involve a multiplicative dependence coefficient on their variance proxy and not the natural quantity $\Vert \bbE[(\sum_{i} \bX_i)^2] \Vert$ present in the leading-order term in \eqref{eq:WittyArm}.
That dependence coefficients occur somewhere in the results is necessary, of course, since one can not expect good concentration estimates if the summands may have arbitrary dependence.
In our results, however, the dependence coefficients will only occur in the error term $\varepsilon$ in \eqref{eq:WittyArm} which is often small in applications.
Thus, by passing through Gaussian theory, the leading-order term in our results can incorporate the dependence in the way that it actually appears in the summed matrix.
This can lead to more efficient estimates, especially if the Markov chain exhibits slow decay of dependence.

Previous results that are specific to time-homogeneous Markov chains have quantified the dependence using quantities based on a spectral gap \cite{neeman2023concentration,garg2018matrix,qiu2020matrix,raohoeffding}, while results that extend to time-inhomogeneous or non-Markovian settings have used  $\beta$-and $\tau$-mixing dependence coefficients \cite{banna2016bernstein,han2020moment}.
We here consider a Markovian setting that allows time-inhomogeneity and use a $\psi$-dependence coefficient; see \eqref{eq: Def_PsiDependence}.
This dependence coefficient is sufficient for the applications we have in mind, but note that it is a strong assumption relative to aforementioned notions of dependence.
It would be interesting future work if one could establish extensions of our results with weaker dependence coefficients.
In particular, extensions using a spectral gap would be interesting as these have applications in theoretical computer science such as randomness-efficient sampling methods \cite{garg2018matrix}.

\subsection{Proof techniques}\label{sec: IntroProofTech}

Our proof involves an interpolation $\{\bS(t): t\in [0,1] \}$ from the summed matrix $\sum_i \bX_i$ to an independent Gaussian matrix with matching mean and covariance structure.
Then, to show that the summed matrix and the Gaussian have similar spectral properties, it suffices to control the rate of change along the interpolation.
This approach is inspired by \cite{brailovskaya2022universality} who established universality in a setting with independent summands.
In their setting, the rate of change along the interpolation could be controlled using an expansion in terms of classical cumulants.
Such an expansion is efficient when one assumes independence due to an \emph{independence--implies--vanishing property} of classical cumulants.
This property allows \cite{brailovskaya2022universality} to neglect all terms involving independent summands.

We however lack independence---our summands have Markovian dependence---and hence risk a combinatorial explosion in the number of terms.
To solve this, one could hope that classical cumulants with approximately independent random variables are small although not necessarily zero.
Such a property indeed holds true to some extent, but we found that a proof using only classical cumulants does not yield efficient estimates.
This can be explained in hindsight from the fact that Markov chains are time-ordered, which classical cumulants do not respect because they are permutation invariant.

Our approach rather relies on \emph{Boolean cumulants}.
Boolean cumulants are not as well known as classical cumulants but enjoy a natural interaction with the underlying Markovian structure, thus solving the key issue mentioned above; see Proposition \ref{prop: BooleanMarkov}.
This interaction was historically first exploited by Statulevi{\v{c}}ius \cite{saulis1991limit,statulevivcius1969limit}.\footnote{In \cite{saulis1991limit,statulevivcius1969limit}, Boolean cumulants are called \emph{centered moments}.}
Our setting is however more delicate than the one studied in \cite{saulis1991limit,statulevivcius1969limit} because we are concerned with matrix-valued summands.
To be specific, issues arise from noncommutativity preventing one from reordering a product of matrices in a time-ordered fashion.
To circumvent this, we dualize the problem in terms of the transition densities of the Markov chain and subsequently rely on a change--of--measure argument to encode the decay of dependence in scalar-valued random variables; see Proposition \ref{prop: BookkeepingLambdaPsi}.
This allows us to leave the matrices in the original order and is one of the critical new ideas for our approach, as it is ultimately what enables us to circumvent the difficulties for a Markovian and noncommutative setting.

Combining this idea with a classical--to--Boolean relation due to Arizmendi, Hasebe, Lehner, and Vargas \cite{arizmendi2015relations} allows us to establish an expansion for the rate of change along the interpolation which efficiently incorporates the decay of dependence in the Markovian sequence; see Proposition \ref{prop: MarkovNewVariables}.
The contribution of the random variables which we used to encode the decay of dependence is here not immediately obvious.
The size of these random variables is namely linked to the nontrivial combinatorics associated with the classical--to--Boolean relation from \cite{arizmendi2015relations}.
We perform an analysis of the combinatorics involved in Section \ref{sec: DeltasPsi}.

We finally note that Boolean cumulants are also studied in the free-probabilistic literature \cite{speicher12boolean}.
This may lead one to believe that the free-probabilistic constant in \eqref{eq:WittyArm} arises in our proof in the step where we rely on Boolean cumulants.
This is however not the case.
As was sketched above, Boolean cumulants allow us to exploit the Markovianity whereas the free-probabilistic main term shows up when we use the Gaussian theory of \cite{bandeira2021matrix,bandeira2024matrix}.

\subsection{Notation}
For a real-valued random variable $W$, we denote $\Vert W \Vert_{\Linf}$ for the \emph{essential supremum} of $\lvert W \rvert$.
We equip $\bbC^d$ with the Euclidean norm $\Vert v \Vert \de \sqrt{\langle v,v \rangle}$.
For a matrix $\bM \in \bbC^{d\times d}$ we denote $\Vert \bM \Vert_{\op}\de \sup_{\Vert v \Vert = \Vert w \Vert = 1} \lvert \langle v, \bM w \rangle \rvert$ for the \emph{operator norm}.
Denote $\tr \bM \de \frac{1}{d}\Tr\bM$ for the \emph{normalized trace} of $\bM$ and let it be understood that $\tr \bM^p = \tr[\bM^p]$.
The collection of all self-adjoint $d\times d$ matrices is denoted $\Csa^{d\times d}$.

\subsection{Structure of this paper}
We state our results in Section \ref{sec: Results}.
Two illustrative applications of these results are discussed in Section \ref{sec: Applications}.
The proofs of the main universality result and its corollaries are given in Sections \ref{sec: MainUniversalityProof} and \ref{sec: BoundsOpNormMarkov}, respectively.
Some supplemental details related to the applications of our results are deferred to the appendices.

\section{Results}\label{sec: Results}
\subsection{Matrix models and parameters}\label{sec: MarkovModel}
Consider a sequence of random variables $Z \de (Z_i)_{i=1}^n$ such that $Z_i$ takes values in a standard Borel space $\cZ_i$ for any $i \geq 1$.
This sequence is said to be \emph{Markovian} if for each $i > 1$ and any measurable $E \subseteq \cZ_i$,
\begin{equation}
	\bbP(Z_i \in E \mid Z_1,\ldots,Z_{i-1})  = \bbP(Z_i \in E \mid Z_{i-1}),
\end{equation}
almost surely.
A sequence of self-adjoint random matrices $(\bX_i)_{i=0}^n$ is said to come from a \emph{Markovian model associated to Z} if $\bX_0$ is a deterministic matrix and there exist measurable functions $\bF_i:\cZ_i \to \Csa^{d\times d}$ such that $\bX_i = \bF_i(Z_i)$ with $\bbE[\bX_i]=0$ for every $i\geq 1$.
We are concerned with the sum:
\begin{equation}
	\bS \de \bX_0 + \sum_{i=1}^n \bX_i.\label{eq: Def_S}
\end{equation}
Note that this is more general than the setting in the introduction insofar that $\bS$ may have nonzero mean and the Markovian sequence may be time-inhomogeneous.
Our main result compares the spectral properties of $\bS$ with those of a matrix with jointly Gaussian entries.
\subsubsection{Gaussian model}\label{sec: GaussianModel}
Let $\Cov(\bS)$ denote the $d^2 \times d^2$ covariance matrix of the entries of the self-adjoint matrix $\bS$.
That is, for any $i,j,k,l \in \{1,\ldots,d \}$
\begin{equation}
	\Cov(\bS)_{ij,kl} \de \bbE\bigl[(\bS - \bbE[\bS])_{i,j}\overline{(\bS - \bbE[\bS])_{k,l}}\bigr].
\end{equation}
Here, $\overline{z}$ denotes the complex conjugate of a complex number $z\in \bbC$.
A self-adjoint random matrix $\bG$ satisfying the following properties is called a \emph{Gaussian model of $\bS$}:
\begin{enumerate}
	\item The $2d^2$-dimensional real-valued vector consisting of the real and imaginary parts of the entries, $\{\Re \bG_{i,j} , \Im \bG_{i,j}: i,j \in \{1,\ldots,d \} \}$, is Gaussian.
	\item The mean and covariance match: $\bbE[\bG] = \bbE[\bS]$ and $\Cov(\bG) = \Cov(\bS)$.
	\item The random matrices $\bG$ and $\bS$ are independent.
\end{enumerate}
We next introduce parameters that are used in our results to quantify to what extent the spectral properties of $\bS$ are matched by those of its Gaussian model.

\subsubsection{Dependence parameter}\label{sec: DependenceParameter}
Given a probability space $(\Omega, \cF, \bbP)$ and $\sigma$-algebras $\cA, \cB \subseteq \cF$, the \emph{$\psi$--dependence coefficient} of $\cA$ and $\cB$ is defined by
\begin{equation}
	\psi(\cA , \cB) \de \sup_{A \in \cA: \bbP(A) >0}\sup_{B \in \cB: \bbP(B) > 0}\, \Bigl\lvert\, \frac{\bbP(A \cap B) - \bbP(A)\bbP(B)}{\bbP(A)\bbP(B)}\, \Bigr\rvert.\label{eq: Def_PsiDependence}
\end{equation}
For two random variables $V,W$ defined on the same probability space we denote $\psi(V , W) \de \psi(\sigma(V) , \sigma(W))$ for the dependence coefficient between the associated $\sigma$-algebras.

The following parameter then allows us to bound the amount of dependence in the Markovian sequence of random variables $Z$:
\begin{equation}
	\Psi(Z) \de \min\Bigl\{ j \geq 1: \psi(Z_{i+j} , Z_i ) \leq \frac{1}{4} \ \text{ for all }i\in \{1,\ldots,n-j \}\Bigr\}.\label{eq: Def_psiMixing}
\end{equation}
Let us remark that the occurrence of $1/4$ in the definition \eqref{eq: Def_psiMixing} of $\Psi(Z)$ is not significant.
Similar results can be achieved if any other number  between zero and one is used.

\begin{remark}\label{rmk: Phi}
	Another quantity that is commonly used to quantify the dependence in a Markov chain is given by the \emph{total variation mixing time} $\tmix$  \cite[Section 4.5]{levin2017markov}.
	To apply our results, it may be useful to know that $\Psi(Z)$ can be bounded using the latter if $Z$ is a stationary ergodic Markov chain on a finite state space.
	Specifically, one then has $\Psi(Z) \leq (3 + \log_2(1/\min_i \pi_i)) \tmix$ with $\pi$ the stationary distribution; see Appendix \ref{apx: Psi} for a proof.
\end{remark}

\subsubsection{Matrix parameters}\label{sec: ParametersMarkovian}

All our results assume that the summands $\bX_i$ are bounded in operator norm.
We denote $R(\bX)$ for the corresponding parameter:
\begin{equation}
	R(\bX)\de \Bigl\Vert \max_{1\leq i \leq n} \Vert \bX_i \Vert_{\op} \Bigr\Vert_{\Linf} < \infty.\label{eq: Def_R(X)}
\end{equation}
We further introduce the following variance proxies:
\begin{equation}
	\sigma(\bS)^2 = \bigl\Vert \bbE\bigl[(\bS-\bbE[\bS])^2\bigr] \bigr\Vert_{\op} \ \textnormal{ and }\ \varsigma(\bX)^2 \de \biggl\Vert \bbE\biggl[\sum_{i=1}^n \bX_{i} ^2\biggr] \biggr\Vert_{\op}. \label{eq: Def_sigma}
\end{equation}
Note that $\varsigma(\bX) = \sigma(\bS)$ in the special case where $\bX_1,\ldots,\bX_n$ are independent.

Theorem \ref{thm: MainTracial} establishes that $\bS$ can be controlled in terms of the Gaussian model $\bG$.
Thereafter, we rely on results from \cite{bandeira2021matrix,bandeira2024matrix} which control $\bG$ in terms of a \emph{free-probabilistic model} $\bS_{\free}$.
The quality of the bounds in \cite{bandeira2021matrix,bandeira2024matrix} is determined by the following parameter:
\begin{equation}
	v(\bS)^2 \de \Vert \Cov(\bS) \Vert_{\op}. \label{eq: Def_v(S)}
\end{equation}
The only free-probabilistic quantity which is used in the statement of our results is the norm $\Vert \bS_{\free} \Vert$.
For the sake of brevity, we hence forego the precise definition\footnote{The free model of a Gaussian matrix is defined in \cite[(2.1)]{bandeira2021matrix} and we let $\bS_{\free} \de \bG_{\free}$ with $\bG$ the Gaussian model of $\bS$. } of the model itself and simply note that an identity by Lehner \cite{lehner1999computing}, \cite[Lemma 2.4]{bandeira2021matrix} allows one to compute the norm in terms of the mean and covariance structure of $\bS$ as
\begin{equation}
	\Vert \bS_{\free} \Vert =  \negsp \max_{\eta \in \{+1, - 1\}} \inf_{\bW \succ 0} \lambda_{\max}\Bigl(\bW^{-1} + \eta \bbE[\bS] + \bbE\bigl[(\bS- \bbE[\bS])\bW (\bS- \bbE[\bS])\bigr]  \Bigr),\label{eq: Lehner}
\end{equation}
where the infimum runs over deterministic positive definite $d\times d$ matrices $\bW$, and $\lambda_{\max}$ refers to the greatest eigenvalue.

\begin{remark}\label{rem: Pisier}
	An estimate by Pisier \cite{pisier2003introduction}, \cite[Lemma 2.5]{bandeira2021matrix} provides a user-friendly bound on the free-probabilistic quantity:
	\begin{equation}
		\max\bigl\{\Vert \bbE[\bS] \Vert_{\op}, \sigma(\bS)\bigr\}  \leq \Vert \bS_{\free} \Vert \leq \Vert \bbE[\bS] \Vert_{\op} + 2\sigma(\bS).\label{eq: Pisier}
	\end{equation}
	In particular, it follows that $\sigma(\bS) \leq \Vert \bS_{\free} \Vert \leq 2\sigma(\bS)$ if $\bbE[\bS] = 0$.
	For centered random matrices, the role of the exact formula \eqref{eq: Lehner} is hence that it enables sharp constants on the leading-order term.
	If one is not concerned with constant factors, then one can simply use the bounds in terms of $\sigma(\bS)$     and avoid computing $\Vert \bS_{\free} \Vert$.
\end{remark}
\begin{remark}\label{rem: sigma_v_psi}
	For an additional simplification of the parameters, consider the following upper bounds proven in Section \ref{sec: BoundSigmaV}:
	\begin{equation}
		\sigma(\bS)^2 \leq 3 \Psi(Z) \varsigma(\bX)^2 \ \textnormal{ and }\ v(\bS)^2 \leq 3 \Psi(Z) \biggl\Vert \sum_{i=1}^n \Cov(\bX_i) \biggr\Vert. \label{eq:OilyBird}
	\end{equation}
	These bounds are tight up to the absolute constant if the summands $\bX_i$ are independent.
	If the Markov chain mixes slowly, however, then directly using the definition of $\sigma(\bS)$ and $v(\bS)$ can sometimes yield significant improvements relative to these bounds, depending on the structure of the summed matrix; see the discussion after Lemma \ref{lem: QualitativeParamEstimates}.
\end{remark}

\subsection{Results}\label{sec: SubsecResultsMarkovian}
Our main result establishes universality for the tracial moments of $\bS$, showing that these can be approximated by the tracial moments of the Gaussian model:
\begin{theorem}\label{thm: MainTracial}
	There exists an absolute constant $c>0$ such that for every integer $p\geq 1$,
	\begin{equation}
		\lvert \bbE[\tr \bS^{2p}]^{\frac{1}{2p}} - \bbE[\tr \bG^{2p}]^{\frac{1}{2p}} \rvert \leq c R(\bX)^{\frac{1}{3}} \Psi(Z)^{\frac{2}{3}}\varsigma(\bX)^{\frac{2}{3}}  p^{\frac{2}{3}} + cR(\bX) \Psi(Z) p.\nonumber
	\end{equation}
\end{theorem}
Most relevant for applications of this result, one can deduce estimates on the operator norm of a random matrix given an estimate on its tracial moments by using that $\Vert \bM \Vert^{2p}/d \leq \tr \bM^{2p} \leq \Vert \bM \Vert^{2p}$ for any self-adjoint $d\times d$ matrix $\bM$.
This implies that practical bounds on $\Vert \bS \Vert$ can be deduced whenever such bounds are available for its Gaussian model; see Lemma \ref{lem: GeneralLpGaussian}.
For example, by combining Theorem \ref{thm: MainTracial} with recent results from \cite{bandeira2021matrix,bandeira2024matrix}, we find the following sharp two-sided bounds:

\begin{corollary}\label{cor: MarkovMarkovBound}
	There exists an absolute constant $C>0$ such that for every $p\geq 1$:
	\begin{equation*}
		d^{-\frac{1}{2p}}  \Vert \bS_{\free} \Vert_{\op} - C\cE(p) \leq \bbE\bigl[\Vert \bS \Vert^{2p}\bigr]^{\frac{1}{2p}} \leq d^{\frac{1}{2p}}  \Vert \bS_{\free} \Vert_{\op} + Cd^{\frac{1}{2p}}\cE(p)
	\end{equation*}
	where the error term $\cE(p)$ is defined by
	\begin{equation*}
		\cE(p) \de   R(\bX)^{\frac{1}{3}} \Psi(Z)^{\frac{2}{3}}  \varsigma(\bX)^{\frac{2}{3}}  p^{\frac{2}{3}} + R(\bX) \Psi(Z) p + v(\bS)^{\frac{1}{2}}\sigma(\bS)^{\frac{1}{2}}\Bigl( p^{\frac{1}{2}} + \ln(d+1)^{\frac{3}{4}} \Bigr).\nonumber
	\end{equation*}
\end{corollary}
Taking $p$ a sufficiently large multiple of $\ln(d)$ in Corollary \ref{cor: MarkovMarkovBound} ensures that $d^{1/2p} \leq 1 + \delta$ for some arbitrarily small $\delta >0$.
If the parameters occurring in the error term  $\cE$ are small, as is often the case in applications of the result, then it follows that both the lower and upper bounds in Corollary \ref{cor: MarkovMarkovBound} are given by $\Vert \bS_{\free} \Vert$ up to a small error.
In particular, this implies the upper bound summarized in \eqref{eq:WittyArm} for centered matrices since Remark \ref{rem: Pisier} then yields that $\Vert \bS_{\free} \Vert = c\sigma(\bS)$ for some $1 \leq c\leq 2$ and since $\bbE[\Vert \bS \Vert] \leq \bbE[\Vert \bS \Vert^{2p}]^{1/2p}$ by Jensen's inequality.

Further, combining Corollary \ref{cor: MarkovMarkovBound} with Markov's inequality implies that for every $\delta >0$ there exists a constant $c>0$ such that the following upper tail bound holds for every $x>0$:
\begin{equation}
	\bbP\bigl(  \Vert \bS \Vert_{\op} \geq  (1+\delta)\Vert \bS_{\free} \Vert  +  \cE(x)\bigr) \leq (d+1)\exp( - cx). \label{eq: MarkovMarkovBound}
\end{equation}
The proof details are given in Section \ref{sec: BoundsOpNormMarkov}.
We also expect a lower bound $\Vert \bS \Vert \geq (1-\delta)\Vert \bS_{\free} \Vert$ to hold with high probability, but this is not immediate from Corollary \ref{cor: MarkovMarkovBound} as extracting lower bounds from moments requires a separate concentration--of--measure ingredient; see Lemma \ref{LEM: SAMSON} and Remark \ref{rem: Samson} for further discussion in a special case with summands of the form $\bX_i = f_i(Z_i) \bB_i$ for scalar functions $f_i$ and deterministic matrices $\bB_i\in \Csa^{d\times d}$.

The price for the sharpness of the bounds is that it necessitates estimating quite a few matrix parameters when applying the results.
The universality principle can however also be combined with different results for Gaussian matrices, which may be more suitable if one is not concerned with sharp constants or logarithmic dimensional factors.
For instance, using the matrix Khintchine inequality of Lust--Piquard \cite{lust1986inegalites}, \cite[Section 2.3]{tropp2018second} instead of the free-probabilistic results from \cite{bandeira2021matrix,bandeira2024matrix} yields the following estimate removing $v(\bS)$ at the cost of a polylogarithmic factor, but still with the natural variance proxy $\sigma(\bS)$ on the main term:
\begin{corollary}\label{cor: Khintchine}
	Additionally assume that $\bbE[\bS]=0$.
	Then, there exists an absolute constant $c>0$ such that
	\begin{equation}
		\bbE[\Vert \bS \Vert] \leq c \ln(d+1)^{1/2}\sigma(\bS) + c R(\bX)^{\frac{1}{3}} \Psi(Z)^{\frac{2}{3}}\varsigma(\bX)^{\frac{2}{3}}  \ln(d+1)^{\frac{2}{3}} + cR(\bX) \Psi(Z) \ln(d+1).\nonumber
	\end{equation}
\end{corollary}

So,  combining the universality result with appropriately chosen bounds from the Gaussian literature can enable practical concentration estimates.
We illustrated this with Corollaries \ref{cor: MarkovMarkovBound} and \ref{cor: Khintchine}, and one could naturally also combine Theorem \ref{thm: MainTracial} with other results from the Gaussian literature.
Another potential use-case for the universality of tracial moments from Theorem \ref{thm: MainTracial} is to establish limiting laws for empirical eigenvalue distributions or empirical singular value distributions.
An example illustrating this is given in Section \ref{sec: LimitingSingValue}.

\begin{remark}\label{rem: FullSpec}
	One can extract concentration inequalities for extremal eigenvalues from Corollary \ref{cor: MarkovMarkovBound}.
	By replacing $\bS$ by $\bS' \de \bS + t\b1$ for $t > \Vert \bX_0 \Vert_{\op} + n R(\bX)$ one can namely ensure that $\Vert \bS' \Vert_{\op} = \lambda_{\max}(\bS) + t$.
	Similarly, one can consider $\bS - t\b1$ to establish a bound on $\lambda_{\min}(\bS)$.

	This trick does not apply to nonextremal eigenvalues.
	In this context, let us note that \cite[Theorem 2.4]{brailovskaya2022universality} establishes concentration for the entire spectrum with regard to the Hausdorff distance when the summands are independent.
	We believe that a similar result should hold true in a Markovian setting and that Boolean cumulants would be an important ingredient in the proof.
	For the sake of brevity, however, we do not pursue this extension here.
\end{remark}

\section{Examples}\label{sec: Applications}

\subsection{Markovian entries}\label{sec: ApplicationsMarkov}
We start by briefly considering a symmetric matrix with entries defined by a Markov chain.
The goal is to give some intuition on the matrix parameters in a simple setting.
Further, it is be possible to establish two-sided tail bounds in this case.

Consider scalar random variables of the form $f_t(Z_t)$ for $Z_1,\ldots,Z_n$ a Markovian sequence and $f_t$ real-valued functions.
Suppose that $n \de d(d+1)/2$ and fix a bijective function $\varphi: \{1,\ldots,n\} \to  \cI$ with $\cI$ the set of unordered pair $\{i,j \}$ satisfying $i, j \in \{1,\ldots,d \}$.
Then, we can define a symmetric $d\times d$ random matrix by
\begin{equation}
	\bS \de  \sum_{t\leq n} \bX_{t}\  \textnormal{ with }\  \bX_{t} \de
	\begin{cases}
		f_t(Z_t) (e_{i}e_j^{\transpose} +   e_j e_i^{\transpose}) & \textnormal{if }\varphi(t) = \{i,j \} \textnormal{ with }i \neq j, \\
		f_t(Z_t) e_{i}e_i^{\transpose}                            & \textnormal{if }\varphi(t) = \{i,i\},                              \\
	\end{cases} \label{eq:AngryParrot}
\end{equation}
where $e_1,\ldots,e_n \in \bbR^d$ is the standard basis.
Assume that $\bbE[f_t(Z_t)] = 0$ for all $t$.

\begin{lemma}\label{lem: MarkovEntriesParamEst}
	For \eqref{eq:AngryParrot}, it holds that
	\begin{equation}
		R(\bX) = \max_{i,j\leq d}\Vert \bS_{i,j} \Vert_{\Linf} \ \textnormal{ and }\ \varsigma(\bX)^2 =\max_{i\leq d}\sum_{j\leq d} \bbE\Bigl[\bS_{i,j}^2\Bigr].\label{eq:BlueBox}
	\end{equation}
	Further, we have $\sigma(\bS)^2 \leq 3 \Psi(Z)  \varsigma(\bX)^2$ and
	$
		v(\bS)^2 \leq 6 \Psi(Z) \max_{i,j \leq d} \bbE[\bS_{i,j}^2].
	$
\end{lemma}
\begin{proof}
	The bounds in \eqref{eq:BlueBox} are immediate from the definitions \eqref{eq: Def_R(X)} and \eqref{eq: Def_sigma}, where we use that the operator norm of a diagonal matrix is equal to the maximum of its entries when computing $\varsigma(\bX)^2$.
	The bound on $\sigma(\bS)^2$ is in \eqref{eq:OilyBird}.

	Finally, regarding the estimate on $v(\bS)$, recall \eqref{eq:OilyBird} and note that $\sum_{t\leq n}\Cov(\bX_{t})$ can be written in a block diagonal form with blocks of size $\leq 2$ associated with the symmetric entries $(i,j)$ and $(j,i)$.
	Those blocks have norm $\Vert \Cov(\bX_t) \Vert \leq 2\mathbb{E}[f_t(Z_t)^2]$ with equality if $\varphi(t) = \{i, j\}$ for $i\neq j$.
	The estimate then follows because the norm of a block diagonal matrix is the greatest norm of its blocks.
\end{proof}

For comparison, the variance proxy in the Markovian Bernstein inequality \eqref{eq:IcyCrow} from \cite{neeman2023concentration} has the following expression for \eqref{eq:AngryParrot}:
\begin{equation}
	\sum_{t \leq n} \Vert \bbE[\bX_{t}^2] \Vert  = \sum_{i=1}^d \sum_{j \geq i } \bbE\Bigl[f_{\varphi^{-1}(i,j)}(Z_{\varphi^{-1}(i,j)})^2  \Bigr] = \sum_{i=1}^d \sum_{j \geq i } \bbE\Bigl[\bS_{i,j}^2  \Bigr].   \label{eq:CalmPen}
\end{equation}
So, the difference between the variance proxies is that the sum over $i$ in \eqref{eq:CalmPen} is replaced by a maximum in $\varsigma(\bX)$ and $\sigma(\bS)$.
This allows for bounds with good dimensional dependence in applications like Wigner-type matrices, that are inaccessible with the weaker variance proxy.

For instance, suppose that the $Z_t$ are generated by some fixed $\psi$-mixing Markov chain and that $\Vert \bS_{i,j} \Vert_{\Linf} \leq 1$ for all $i,j\leq d$.
Then, using that $\Vert \bS_{\free} \Vert \leq 2\sigma(\bS)$ by \eqref{eq: Pisier}, substituting the estimates from Lemma \ref{lem: MarkovEntriesParamEst} in the upper bound in Corollary \ref{cor: MarkovMarkovBound} yields that $\bbE[\Vert \bS \Vert] = O(\sqrt{d})$ as $d\to \infty$.
Thus, we recover the correct asymptotic order in this example: recall that $\Vert \bS \Vert \approx 2\sqrt{d}$ in the special case where the entries are independent and identically distributed.
For comparison, bounds using $\sum_t \Vert \bbE[\bX_t^2] \Vert$ as variance proxy like \eqref{eq:IcyCrow} would give a suboptimal asymptotic order $O(d\sqrt{\ln(d+1)})$ which is loose by a factor of order $\sqrt{d \ln(d)}$.

The appropriate order of magnitude in this specific model could also have been extracted from earlier results in the literature such as \cite[Corollary 1.4]{raohoeffding}, but the sharpness of our results allows one to go further.
One could now also determine an explicit constant in the big-$O$ notation based on the free-probabilistic quantity $\Vert \bS_{\free} \Vert$.
The latter does not admit a simple exact expression at this level of generality, but this may be possible in some special cases depending on the covariance structure of the entries.
For instance, a simplified expression in the case where $\bS$ has independent entries may be found in \cite[Lemma 3.2]{bandeira2021matrix}.
Whatever it may be, the free-probabilistic bound is here asymptotically tight as we have two-sided tail bounds:
\begin{lemma}\label{LEM: SAMSON}
	For \eqref{eq:AngryParrot}, assume that $\Vert \bS_{i,j} \Vert_{\Linf} \leq 1$ for all $i,j$.
	Then, for every $\delta >0$, there exist constants $c,C>0$ such that for every $x>0$,
	\begin{equation}
		\bbP\Bigl(\min_{\lvert \gamma \rvert \leq \delta}\bigl\lvert  \Vert \bS \Vert  -  (1 + \gamma)\Vert \bS_{\free} \Vert \rvert  > x + C\tilde{\cE}  \Bigr) \leq   \exp\Bigl( - \frac{c x^2}{\Psi(Z)} \Bigr).\label{eq:NewGem}
	\end{equation}
	Here,
	$
		\tilde{\cE} \de  \Psi(Z)^{2/3} d^{1/3}  \ln(d+1)^{2/3} + \Psi(Z) \ln(d+1).
	$
\end{lemma}
The proof amounts to an application of a concentration--of--measure principle by Samson \cite{samson2000concentration} for convex functions of (weakly) dependent scalar random variables.
The latter yields that $\Vert \bS  \Vert - \bbE[\Vert \bS \Vert]$ has sub-Gaussian deviations after which \eqref{eq:NewGem} follows from the expectation bounds in Corollary \ref{cor: MarkovMarkovBound}.
We refer to Appendix \ref{apx: Samson} for the proof details.

\begin{remark}\label{rem: Samson}
	Note that the tail bound in \eqref{eq:NewGem} is sub-Gaussian and dimension-independent.
	In particular, this yields sharper bounds on the upper tails than the one in \eqref{eq: MarkovMarkovBound} that followed immediately from Corollary \ref{cor: MarkovMarkovBound}.
	This reflects a broader principle: the main content in matrix concentration results like ours lies in expectation bounds, not in the immediate tail bounds.
	Once a bound on $\bbE[\Vert \bS \Vert]$ like \eqref{eq:WittyArm} is known, scalar theory for controlling the deviations of $\Vert \bS \Vert - \bbE[\Vert \bS \Vert]$ can often be used to give sharper tail bounds.
	For instance, \cite{samson2000concentration} can also be applied to a more general \emph{matrix series model} with summands of the form $\bX_t = f_t(Z_t)\bB_t$ with $\bB_t$ deterministic matrices.
	Developing analogous results for the general case $\bX_t = \bF_t(Z_t)$ could be relevant future work.
\end{remark}

\subsection{Block Markov chains}\label{sec: ApplicationsBMC}
We next consider a model that is used to study clustering algorithms for sequential data, and whose spectral properties were previously studied using asymptotic and model-specific methods.
Our general-purpose results can here be used to painlessly establish nonasymptotic estimates and to sharpen previous asymptotic results.

Let $d\geq K\geq 1$ be positive integers, consider a partition $(\cV_{i})_{i=1}^K$ of $\{1,\ldots,d \}$ into nonempty subsets, and let $\mathbf{p} \in [0,1]^{K\times K}$ be the transition matrix of an ergodic Markov chain on $\{1,\ldots,K \}$.
Then, the \emph{block Markov chain} \cite{sanders2020clustering} with cluster transition matrix $\mathbf{p}$ and clusters $(\cV_i)_{i=1}^K$ is the Markov chain $(Z_t)_{t=1}^n$ on $\{1,\ldots,d \}$ whose transition probabilities only depend on the states' clusters:
\begin{equation}
	\bbP(Z_{t} = j \mid Z_{t-1} = i)  = \frac{\mathbf{p}_{k,m}}{\#\cV_{m}}\  \text{ for all }\  i\in \cV_{k}\ \text{ and }\  j\in \cV_m. \label{eq: Def_PBMC}
\end{equation}
A schematic depiction of a block Markov chain may be found in Figure \ref{fig: BMC}.

The analysis of spectral clustering algorithms \cite{sanders2020clustering,sanders2023spectral,jedra2023nearly,zhang2020spectral,van2024estimating} which recover the clusters $\cV_i$ based on an observed sample path crucially require concentration estimates for the \emph{sample frequency matrix} $\hat{\bN}$ associated with the sample path defined by
\begin{equation}
	\hat{\bN} \de (\hat{\bN}_{i,j})_{i,j = 1}^d\ \text{ where }\ \hat{\bN}_{i,j} \de \sum_{t=1}^{n-1}\bb1\{Z_{t} = i, Z_{t+1} = j\}.
\end{equation}
Using a model-specific analysis, the asymptotic order of magnitude of $\Vert \hat{\bN} - \bbE[\hat{\bN}] \Vert$ was established in \cite{sanders2021spectral} and the limiting distribution of all singular values was established in \cite{vanwerde2023singular}.
We establish refinements of these results in  Theorem \ref{thm: BMC_Noise_Free} and Proposition \ref{prop: SingValN}.

\begin{figure}[h]
	\centering
	\includegraphics[width = 0.85\textwidth]{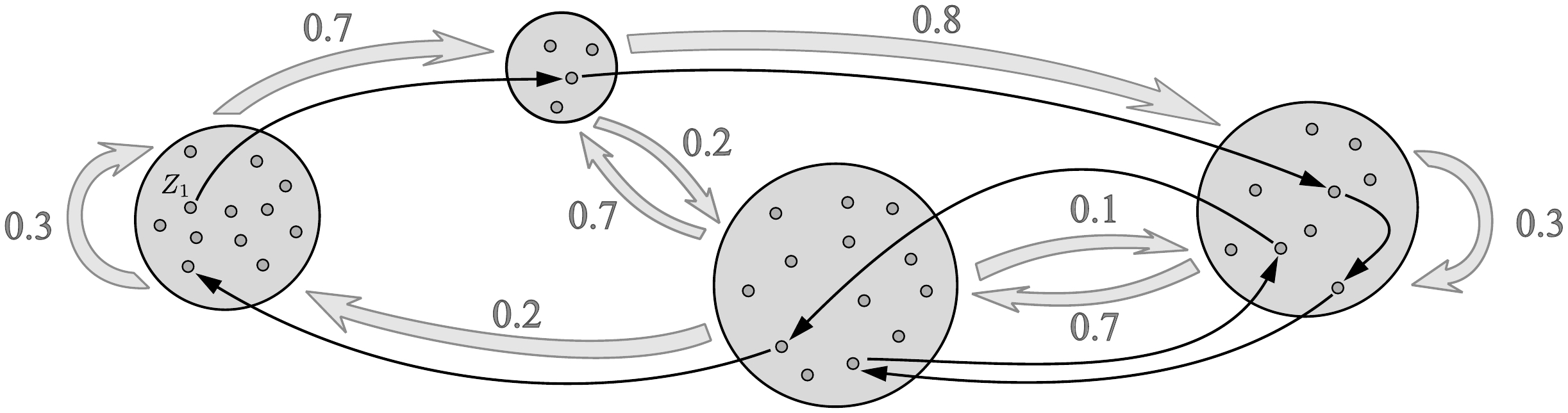}
	\caption{Visualization of a block Markov chain on $K=4$ clusters with cluster transition matrix $\mathbf{p}  =\allowbreak [[0.3,0.7,0,0],\allowbreak [0,0,0.2,0.8],[0.2,0.7,0,0.1],[0,0,0.7,0.3]]$ where we number the clusters from left to right. The thick arrows visualize the transition probabilities $\mathbf{p}_{i,j}$ and the thin arrows visualize the transitions $(Z_{t},Z_{t+1})$ in the observed sample path $(Z_t)_{t=1}^n$. }
	\label{fig: BMC}
\end{figure}

\subsubsection{Concentration and sharp limiting value}
For simplicity, we assume that $(Z_t)_{t=1}^n$ starts in stationarity.
This is to say that $\bbP(Z_1 = i) = \pi_k/\# \cV_k$ for all $i \in \cV_k$ and all $k\in \{1,\ldots,K \}$ where $\pi \in [0,1]^K$ is the stationary distribution associated with $\mathbf{p}$.

Let us denote $\bM \de \sqrt{d/n}(\hat{\bN} - \bbE[\hat{\bN}])$.
Our goal is to estimate $\Vert \bM \Vert_{\op}$.
Since the matrix $\bM$ is not self-adjoint with high probability, we need to consider a preliminary reduction before applying our results.
We define a $2d\times 2d$ matrix, called the \emph{self-adjoint dilation} of $\bM$, by
\begin{equation}
	\bS \de \begin{pmatrix}
		0                & \bM \\
		\bM^{\transpose} & 0
	\end{pmatrix}. \label{eq: Def_S_Dilation}
\end{equation}
Note that $\Vert \bS \Vert_{\op} = \Vert \bM \Vert_{\op}$.
One can represent $\bS$ in terms of a Markovian model associated with the Markov chain of transitions $E \de (E_t)_{t=1}^{n-1}$ defined by $E_{t} \de (Z_{t}, Z_{t+1})$.
Namely, note that with $e_i\in \bbR^d$ the $i$th standard basis vector we can write  $\bS = \sum_{t=1}^{n-1} \bX_{t}$ with
\begin{equation}
	\bX_t \de \sqrt{\frac{d}{n}}\sum_{i,j = 1}^d \bigl(\bb1(E_t = (i,j)) - \bbP(E_t = (i,j))\bigr) \begin{pmatrix}
		0                     & e_{i}e_j^\transpose \\
		e_{j}e_{i}^\transpose & 0
	\end{pmatrix}. \label{eq: Def_Xt}
\end{equation}
Let us write $\hat{\alpha}_i \de \# \cV_i / d$ and $\hat{\alpha}_{\min} \de \min_{i\leq K} \hat{\alpha}_i$.
Further, let $\Psi(\mathbf{p})$ denote the $\psi$-mixing time for the Markov chain on $\{1,\ldots,K \}$ associated with $\mathbf{p}$.
The following lemma then provides an indication of the typical size of the matrix parameters.
\begin{lemma}\label{lem: QualitativeParamEstimates}
	There exist constants $c_1,c_2,c_3>0$ depending only on $\hat{\alpha}_{\min}$ such that
	\begin{align}
		R(\bX) \leq c_1\sqrt{d/n}, \qquad \Psi(E) \leq \Psi(\mathbf{p}) + 1,\qquad                & \varsigma(\bX)^2 \leq \max\{\pi_i/\hat{\alpha}_i: i\leq K\}, \nonumber \\
		\sigma(\bS)^2 \leq  \max\{\pi_v/\hat{\alpha}_i: i\leq K\} + c_2 \Psi(\mathbf{p})/d, \quad & \quad  v(\bS)^2  \leq c_3\Psi(\mathbf{p})/d. \label{eq:HugeHen}
	\end{align}
\end{lemma}
\begin{proof}
	This follows from Lemma \ref{lem: RefinedParamEstimates} which provides more precise estimates.
\end{proof}

Note that the bounds on $\sigma(\bS)^2$ and $\varsigma(\bX)^2$ have the same leading-order contribution in \eqref{eq:HugeHen} because the dependence on $\Psi(\mathbf{p})$ only occurs in a subleading term which is suppressed by a factor $1/d$.
This illustrates another advantage of universality-based concentration results: the variance proxy $\sigma(\bS)^2$ incorporates how the dependence appears in the summed matrix which can be more efficient than worst-case bounds of the form $D\times\varsigma(\bX)^2$ with $D$ a dependence coefficient such as $\Psi(\mathbf{p})$ or quantities based on a spectral gap.

Moreover, using Corollary \ref{cor: MarkovMarkovBound}, it follows that $\Vert \bS \Vert_{\op} \approx \Vert \bS_{\free} \Vert$.
In the asymptotic regime $d\to \infty$, this allows us to refine one of the results in \cite{sanders2023spectral}.
We let $n$ and the clusters $(\cV_{k})_{k=1}^K$ depend on $d$ but assume that the cluster transition matrix $\mathbf{p}$ is kept fixed.
Further, assume that there exist strictly positive numbers $\alpha_1,\ldots,\alpha_k >0$ such that $\lim_{d\to \infty} \# \cV_k /d = \alpha_k$ for every $k$.
Then, in the regime where $n \gg d \ln(d)^4$, the following theorem improves upon \cite[Theorem 3]{sanders2023spectral} in the fact that we can determine the exact limiting value whereas \cite{sanders2023spectral} only proves a nonexplicit upper bound by characterizing the asymptotic order.
Let us note however that \cite{sanders2023spectral} has a weaker assumption: it is there assumed that $n \gg  d \ln(d)$.
\begin{theorem}\label{thm: BMC_Noise_Free}
	Assume that $\lim_{d\to \infty} d\ln(d)^4/n = 0$.
	Then, the random variable $\Vert \bM  \Vert_{\op}$ converges in probability to the scalar $\mathfrak{m}>0$ defined by
	\begin{equation}
		\mathfrak{m} \de \inf_{x \in \bbR_{>0}^{2K}} \max_{i= 1,\ldots,2K} \Bigl\{ \frac{1}{x_i} + \sum_{j=1}^{2K} c_{i,j} x_j \Bigr\} \label{eq: Def_m}
	\end{equation}
	where the infimum runs over all vectors $x$ with strictly positive coordinates and the coefficients $(c_{i,j})_{i,j=1}^{2K}$ are defined by
	\begin{equation}
		c_{i,j} \de \begin{cases}
			0                                               & \text{ if } i \leq K \text{ and }j\leq K, \\
			\alpha_{i}^{-1} \pi_i\mathbf{p}_{i,j - K}       & \text{ if } i\leq K \text{ and }j > K,    \\
			0                                               & \text{ if }  i > K \text{ and }j >K,      \\
			\alpha_{i-K}^{-1}  \pi_{j} \mathbf{p}_{j, i -K} & \text{ if } i >K \text{ and }j \leq K.
		\end{cases}\nonumber
	\end{equation}
\end{theorem}

The proof of this result is given in Section \ref{sec: OpNormBMC} and relies on Corollary \ref{cor: MarkovMarkovBound} to establish an upper bound on $\Vert \bM \Vert$.
Note, however, that we here not only provide an upper bound but rather the exact limiting value.
The proof of this two-sidedness relies on our next result Proposition \ref{prop: SingValN} which implies, in particular, that there is asymptotically an abundance of singular values which are close to the upper bound from Corollary \ref{cor: MarkovMarkovBound}.
This is visualized by Figure \ref{fig: Singvaldensity} in Section \ref{sec: ExperimentalSetup} below.

\subsubsection{Limiting singular value distribution}\label{sec: LimitingSingValue}
Universality of tracial moments can also be used to establish limiting laws for the singular value distribution.
We start by introducing some terminology.
The $i$th largest singular value of a square matrix $\bM \in \bbC^{d\times d}$ can be defined in terms of the $i$th largest eigenvalue of $\bM\bM^*$ as $s_i(\bM)  \de (\lambda_i(\bM \bM^*))^{1/2}$.
The \emph{singular value distribution} of $\bM$ is then the probability measure $\nu_{\bM}$ defined by
\begin{equation}
	\nu_\bM([a,b]) \de \frac{1}{d} \#\bigl\{i \in \{1,\ldots,d \}: a \leq s_i(\bM) \leq b \bigr\}. \label{eq: Def_SingvalMeasure}
\end{equation}
A sequence of random probability measures $\nu_n$ is said to converge \emph{weakly in probability} to a deterministic probability measure $\nu$ if $\int f(x)\intd \nu_n(x)$ converges in probability to $\int f(x) \intd \nu(x)$ for every continuous bounded function $f:\bbR \to \bbR$.

The \emph{Stieltjes transform} of a probability measure $\nu$ on $\bbR$ is the analytic function $s: \bbC^+ \to \bbC^-$ given by $s(z) \de \int 1/(z-x) \intd \nu(x)$ where $\bbC^+ \de \{z \in \bbC: \Im(z) > 0\}$ is the upper half-plane and $\bbC^- \de \{z \in \bbC : \Im(z) < 0 \}$ is the lower half-plane.
The measure can be recovered from its Stieltjes transform using the Stieltjes inversion formula \cite[Theorem B.8]{bai2010spectral} which states that for any continuity points $a < b$ of $\nu$,
\begin{equation}
	\nu([a,b]) = - \frac{1}{\pi} \lim_{\varepsilon \to 0^+} \int_a^b \Im\bigl(s(x + \sqrt{-1}\varepsilon)\bigr).
\end{equation}
Let us warn that \cite{bai2010spectral} employs a definition for the Stieltjes transform which differs from the one used here by a minus sign.
Finally, the \emph{symmetrization} of a measure $\nu$ on $\bbR_{\geq 0}$ is the measure $\sym(\nu)$ on $\bbR$ given by $A \mapsto (\nu(A \cap \bbR_{\geq 0}) +  \nu((-A) \cap \bbR_{\geq 0})/2$ where $A$ ranges over all measurable subsets of $\bbR$ and $-A \de \{-a: a \in A \}$.

The following results improve upon \cite[Theorem 1.1]{vanwerde2023singular}.
In \cite{vanwerde2023singular} it was namely assumed that $n \approx cd^2$ for some constant $c>0$ whereas the following result only requires that $n \gg d$.
This relaxed assumption allows short sample paths which give rise to a sparser sample frequency matrix.
The assumption that $n \gg d$ is optimal, as the result has to fail if $n \approx cd$ for fixed $c>0$.
For instance, suppose that $K=1$ so that the sequence $Z_1,\ldots,Z_n$ consists of independent and identically distributed random variables.
Then, it follows from $n \approx cd$ that there exists some $u>0$ such that the number of states in $\cV$ which are not visited by $Z_t$ is at least $ud$ with high probability.
Then, $\bM$ has at least $ud$ rows equal to zero which implies that $\nu_{\bM}(0) \geq u$.
On the other hand, when $K=1$, the distribution $\sym(\nu_\infty)$ in Proposition \ref{prop: SingValN} is the semicircular law and hence continuous.
Considering $\int f(x)\intd \nu_\bM(x)$ with $f:\bbR \to \bbR$ supported in a small neighborhood of $0$ then shows that $\nu_{\bM}$ does not converge weakly in probability to $\nu_{\infty}$.
\begin{proposition}
	\label{prop: SingValN}
	Assume that $\lim_{d\to \infty} d/n = 0$.
	Then, $\nu_{\bM}$ converges weakly in probability to a compactly supported probability measure $\nu_{\infty}$ on $\bbR_{\geq 0}$.
	Moreover, $\sym(\nu_{\infty})$ has Stieltjes transform $s(z) = \sum_{i=1}^{K} \alpha_i (a_i(z)\allowbreak + a_{K + i}(z))/2$ where $a_1,\ldots,a_{2K}$ are the unique analytic functions from $\mathbb{C}^+$ to $\mathbb{C}^-$ such that the following system of equations is satisfied
	\begin{align}
		a_i(z)^{-1}     & = z -\sum_{j=1}^{K}  \alpha_i^{-1} \pi_i \mathbf{p}_{i,j} a_{K+j}(z) \nonumber \\
		a_{i+K}(z)^{-1} & = z - \sum_{j=1}^{K}  \alpha_i^{-1} \pi_j \mathbf{p}_{j,i} a_{j}(z) \nonumber
	\end{align}
	for $i = 1,\ldots,K$.
\end{proposition}
\begin{corollary}\label{cor: SingValNN}
	Assume that $\lim_{d\to \infty} d/n = 0$ and let $\nu_{\infty}$ be as in Proposition \ref{prop: SingValN}.
	Then, $\nu_{\sqrt{d/n}\hat{\bN}}$ converges weakly in probability to $\nu_{\infty}$.
\end{corollary}
The proof is given in Section \ref{sec: SingvalDistribution} and relies on the moment universality from Theorem \ref{thm: MainTracial}.
It is interesting to note that universality gives a simpler proof than in \cite{vanwerde2023singular}.
The main difficulty in \cite{vanwerde2023singular} is namely to show that all joint moments of the entries of $\bM$ behave approximately as in the independent case.
For the second moments, \ie covariance, this is not too difficult \cite[Proposition 4.8]{vanwerde2023singular}.
Estimating higher-order joint moments is however significantly more technical; see \cite[Section 6.3.4]{vanwerde2023singular}.
A univerality-based approach allows bypassing this technical step because the higher moments of a Gaussian are determined by its covariance.

\begin{remark}\label{rem: TracialTricks}
	The quantitative bounds in Theorem \ref{thm: MainTracial} are only for tracial moments of even order.
	This is sufficient for singular value distributions, but other applications such as eigenvalue distributions of self-adjoint random matrices also requires tracial moments of odd order.
	In this context, note that
	$
		\bbE[\tr (\bS + t\b1)^{2p}] = \sum_{k=0}^{2p} \binom{2p}{k}t^{2p-k} \bbE[\tr \bS^{k}]
	$.
	Since pointwise convergence of a polynomial implies convergence of coefficients, this also allows extracting asymptotic universality for odd tracial moments from Theorem \ref{thm: MainTracial}.
	A similar trick applied to $(\bS \otimes \b1 + t \b1 \otimes \bS)^{2p}$ yields universality for the variance of tracial moments; see the proof of Lemma \ref{lem: BMC_tracialmoment} in Section \ref{sec: ProofBMC} for details.
\end{remark}

\subsubsection{Visualization of results}\label{sec: ExperimentalSetup}

Figure \ref{fig: Singvaldensity} visualizes Theorem \ref{thm: BMC_Noise_Free} and Proposition \ref{prop: SingValN}.
Observe that the edge of the support of the empirical singular value distribution is visually indistinguishable from $\Vert \bS_{\free} \Vert$.
This illustrates the sharpness of the leading-order term: recall that the greatest singular value of a matrix corresponds to its operator norm.
Such sharp leading-order terms are inaccessible with previous general-purpose matrix concentration results with dependencies.

The experimental setup for this figure corresponds to the block Markov chain visualized in Figure \ref{fig: BMC}.
More precisely, we sampled a trajectory from a block Markov chain with $K=4$, $\mathbf{p}  =\allowbreak [[0.3,0.7,0,0],\allowbreak [0,0,0.2,0.8],[0.2,0.7,0,0.1],[0,0,0.7,0.3]]$, $(\alpha_1,\alpha_2, \alpha_3, \alpha_4) = (0.2,0.1,0.4,0.3)$, $d = 10 000$, and $n = 100d$.
The system of equations in Proposition \ref{prop: SingValN} was solved using the algorithm of \cite{helton2007operator} as implemented in \texttt{BMCToolkit} \cite{vanwerde2022detection}.

\begin{figure}[h]
	\centering
	\includegraphics[width = 0.95\textwidth]{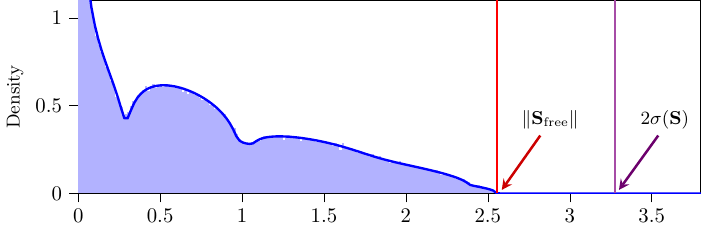}
	\caption{
		The bars display the empirical singular value distribution of the recentered and normalized sample frequency matrix $\bM$ for a sample path of a block Markov chain.
		The blue line displays the asymptotic theoretical prediction from Proposition \ref{prop: SingValN}.
		The red and purple vertical lines respectively display $\Vert \bS_{\free} \Vert$ and $2\sigma(\bS)$ with $\bS$ as in \eqref{eq: Def_S_Dilation}.
	}
	\label{fig: Singvaldensity}
\end{figure}

\section{Proof of \texorpdfstring{Theorem \ref{thm: MainTracial}}{Theorem}}\label{sec: MainUniversalityProof}
Recall that Theorem \ref{thm: MainTracial} is a universality principle stating that the tracial moments of $\bS$ are well-approximated by those of $\bG$.
To establish this universality statement, we interpolate using
\begin{equation}
	\bS(t)\de \bbE [\bS] + \sqrt{t} (\bS - \bbE[\bS]) + \sqrt{1 -t} (\bG - \bbE[\bG])\label{eq: Def_S(t)}
\end{equation}
for $t\in [0,1]$.
The reason for using weights $\sqrt{t}$ and $\sqrt{1-t}$ in the interpolation is to ensure that the covariance of the entries of $\bS(t)$ remains constant as $t$ varies.
In order to prove that $\bbE[\tr \bS^{2p}]^{1/2p} \approx \bbE[\tr\bG^{2p}]^{1/2p}$, it is now sufficient to show that $\lvert\dd{t}\bbE[\tr \bS(t)^{2p}]\rvert$ is small.

Section \ref{sec: CombinePsiBool} establishes an exact expansion for $\dd{t}\bbE[g(\bS(t))]$ where $g:\bbC^{d\times d} \to \bbC$ is a polynomial map.
This involves combinatorics associated with classical--to--Boolean cumulant relations from \cite{arizmendi2015relations} whose contribution we estimate in Section \ref{sec: DeltasPsi}.
Further, the individual terms in the expansion are defined in terms of directional derivatives of $g$.
We estimate these individual terms using trace inequalities in Section \ref{sec: DerivativesTracial}.
These ingredients are finally combined to bound $\dd{t}\bbE[\tr\bS(t)^{2p}]$ in Section \ref{sec: UniversalityEvenOrderMarkov}, concluding the proof.

\subsection{Expansion for the rate--of--change along the interpolation}\label{sec: CombinePsiBool}

\subsubsection{Boolean joint cumulants}\label{sec: BooleanJointCumulants}
The \emph{Boolean joint cumulant} $b(Y_{1},\ldots,Y_{k})$ of a sequence of bounded real random variables $Y_{1},\ldots,Y_{k}$ is defined in terms of the joint moments as
\begin{equation}
	b(Y_1,\ldots,Y_k)\label{eq: Def_Boolean} \de \sum_{m=0}^{k-1}\sum_{1\leq j_1 < \cdots < j_m \leq k-1 }\negsp \negsp (-1)^{m}\bbE[Y_1 \cdots Y_{j_1}]\bbE[Y_{j_1 +1} \cdots Y_{j_2}]\cdots\bbE[Y_{j_{m} +1} \cdots Y_{k}].
\end{equation}
Let us warn that Boolean cumulants are \emph{not} permutation invariant.
That is, it can occur that $b(Y_1,\ldots,Y_k) \neq b(Y_{\rho(1)},\ldots,Y_{\rho(k)})$ for some permutation $\rho \in \cS_k$.
This warning is relevant in some of the subsequent proofs where we have to ensure that the random variables occur in the appropriate order to be able to exploit the decay of dependence in the underlying Markovian sequence $(Z_1,\ldots,Z_n)$.

For any sequence of real values $i = (i_1,\ldots,i_k)$ let us denote $\sort(i)$ for the unique sequence of length $k$ which is a nondecreasing permutation of $i$:
\begin{equation}
	\min\{i_1,\ldots,i_k \} = \sort(i)_1 \leq \sort(i)_2 \leq \cdots \leq \sort(i)_k = \max\{i_1,\ldots,i_k\}.\label{eq: Def_sort}
\end{equation}
For a nonempty subset $\cJ\subseteq \{1,\ldots,k \}$ and $1 \leq \ell\leq \# \cJ$ we denote $\cJ(\ell) \de \sort((j)_{j\in \cJ})_\ell$ for the $\ell$th smallest element in $\cJ$.
We then denote $b(Y_j: j\in \cJ)$ for the Boolean cumulant associated to $\cJ$ with indices in increasing order:
\begin{equation}
	b(Y_j: j\in \cJ) \de b(Y_{\cJ(1)},Y_{\cJ(2)},\ldots,Y_{\cJ(k)}). \label{eq: bYJ}
\end{equation}
Denote $\cS_k$ and $\cP_k$ for the sets consisting of all permutations or partitions of $\{1,\ldots,k\}$, respectively.
An \emph{(increasing) run} in a permutation $\rho \in \cS_k$ is an increasing subsegment of $(\rho(1),\ldots,\rho(k))$ of maximal length.
Here, a \emph{subsegment} is a sequence of the form $(\rho(i), \rho(i+1),\rho(i+2), ...,\rho(i+\ell))$.
We denote $\pi_\rho\in \cP_k$ for the partition of $\{1,\ldots,k \}$ consisting of the runs.
That is, for $i <j$ it holds that $i\sim j$ in the equivalence relation induced by $\pi_\rho$ if and only if $i = \rho(r)< \rho(r+1) <\ldots < \rho(r + \ell) =j$ for some $r,\ell$.
See Figure \ref{fig: runs} for a visualization of a permutation $\rho$ and the associated partition $\pi_{\rho}$.

\begin{figure}[h]
	\centering
	\includegraphics[width = 0.92\textwidth]{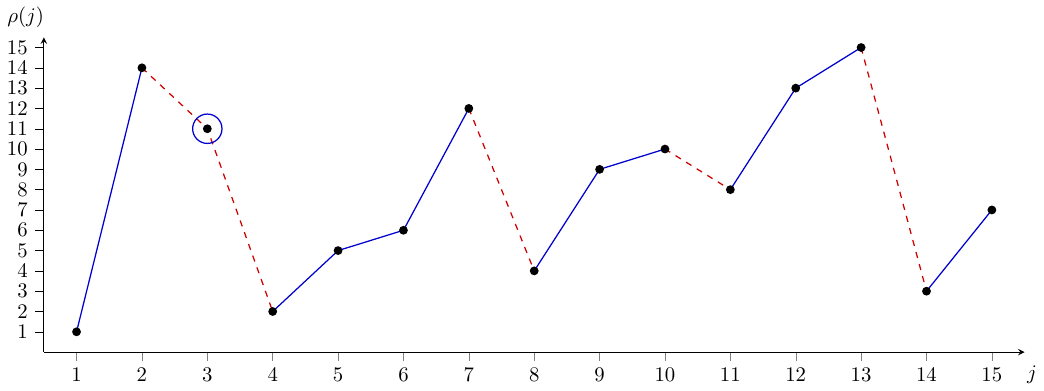}
	\caption{This figure displays a permutation $\rho \in \cS_{15}$.
		Every part in the induced partition of runs $\pi_\rho$ here corresponds to a connected component of the solid blue lines by way of the $y$-values in this connected component.
		The displayed permutation is given by $(\rho(1),\rho(2),\ldots, \rho(15)) = (1,14,11,2,5,6,12,4,9,10,8,13,15,3,7)$.
		The induced partition of runs is given by $\pi_\rho=\{\{1,14\},\{11\},\{2,5,6,12\},\{4,9,10\},\{8,13,15 \},\{3,7\}\}$. }
	\label{fig: runs}
\end{figure}

\begin{proposition}\label{prop: BooleanInterpolation}
	Let $V_1,\ldots,V_n$ be a sequence of, possibly dependent, centered and bounded $\bbR^D$-valued random vectors and consider a sequence of centered Gaussian random vectors satisfying that $\bbE[G_{i_1}G_{i_2}^{\transpose}]  = \bbE[V_{i_1}V_{i_2}^{\transpose}]$ for any $i_1,i_2\in \{1,\ldots,n \}$.
	Assume that $\bV \de (V_1,\ldots,V_n)^\transpose$ and $\bG \de (G_1,\ldots,G_n)^\transpose$ are independent and let $
		\bV(t) \de \sqrt{t}\, \bV + \sqrt{1 - t}\, \bG.
	$

	Then, for any polynomial $g:\bbR^{n\times D} \to \bbC$ and any $t\in [0,1]$
	\begin{align*}
		\dd{t}\bbE\bigl[g\bigl(\bV(t)\bigr)\bigr] & = \frac{1}{2}\sum_{k=3}^\infty\frac{t^{\frac{k}{2}-1}}{(k-1)!} \sum_{i\in \{1,\ldots,n \}^k}\sum_{\rho \in \cS_k:\rho(1) = 1}(-1)^{\#\pi_\rho- 1}\sum_{\gamma \in \{1,\ldots,D \}^k}                      \\
		                                          & \times\Bigl\{\prod_{\cJ \in \pi_\rho} b(V_{\sort(i)_j,\gamma_j}: j\in \cJ) \Bigr\}\bbE\Bigl[\frac{\partial^k g}{\partial v_{\sort(i)_1,\gamma_1} \cdots \partial v_{\sort(i)_k,\gamma_k}}(\bV(t)) \Bigr].
	\end{align*}
\end{proposition}
\begin{proof}
	We will use an expansion in terms of classical cumulants from \cite[Theorem 4.3]{brailovskaya2022universality} as our starting point and rephrase by using a classical--to--Boolean formula from \cite[Corollary 1.6]{arizmendi2015relations}.
	In this context, let us note that the classical cumulant of a sequence of bounded random variables $Y_1,\ldots,Y_k$ is defined by
	\begin{equation}
		\kappa(Y_{1},\ldots,Y_{k}) \de \sum_{\pi \in \cP_k}  (-1)^{\# \pi -1}(\# \pi - 1)!  \prod_{\cJ\in \pi} \bbE\Bigl[\prod_{j\in \cJ} Y_{j}\Bigr].\label{eq: Def_Cumulant}
	\end{equation}
	Note that it is immediate from \eqref{eq: Def_Cumulant} that classical cumulants are permutation invariant, meaning that $\kappa(Y_1,\ldots,Y_k) = \kappa(Y_{\rho(1)},\ldots,Y_{\rho(k)})$ for all $\rho \in \cS_k$.

	We can also view $\bV(t)$ as an $nD$-dimensional random vector whose entries are indexed by tuples $(i,\gamma)\in \{1,\ldots,n \} \times \{1,\ldots,D \}$.
	Hence, by the special case of the classical cumulant expansion \cite[Theorem 4.3]{brailovskaya2022universality} for a \emph{single} $nD$-dimensional random vector,
	\begin{align}
		 & \dd{t} \bbE\bigl[g\bigl(\bV(t)\bigr)\bigr]\label{eq: BVH_cumulantexpansion}                                                                                                                                                                                                                                            \\
		 & = \frac{1}{2}\sum_{k=3}^\infty \frac{t^{\frac{k}{2}-1}}{(k-1)!} \sum_{i\in \{1,\ldots,n \}^k} \sum_{\gamma\in \{1,\ldots,D \}^k}\negsp \negsp \negsp\kappa(V_{i_1,\gamma_1},\ldots,V_{i_k,\gamma_k})\bbE\Bigl[\frac{\partial^k g}{\partial v_{i_1, \gamma_1} \cdots \partial v_{i_k,\gamma_k}}(\bV(t))\Bigr] \nonumber
	\end{align}
	for any $t\in [0,1]$.
	Hence, since $\sort(i)$ is a permutation of $i$ and since classical cumulants are permutation invariant,
	\begin{align}
		 & \sum_{\gamma\in \{1,\ldots,D \}^k} \kappa(V_{i_1,\gamma_1},\ldots,V_{i_k,\gamma_k})\bbE\Bigl[\frac{\partial^k g}{\partial v_{i_1, \gamma_1} \cdots \partial v_{i_k,\gamma_k}}(\bV(t))\Bigr]                                           \\
		 & \ = \sum_{\gamma\in \{1,\ldots,D \}^k} \kappa(V_{\sort(i)_1,\gamma_1},\ldots,V_{\sort(i)_k,\gamma_k})\bbE\Bigl[\frac{\partial^k g}{\partial v_{\sort(i)_1, \gamma_1} \cdots \partial v_{\sort(i)_k,\gamma_k}}(\bV(t))\Bigr].\nonumber
	\end{align}
	We next apply the classical--to--Boolean formula from \cite[Corollary 1.6]{arizmendi2015relations}.

	This formula involves quantities $B_{\operatorname{runs}(\rho)}$ and $d(\rho)$ where $\rho$ is a permutation.
	In our notation, it holds that $\operatorname{runs}(\rho) = \pi_{\rho}$ and $d(\rho) = \# \pi_{\rho}-1$ due to \cite[Section 7, page 79, item (1)]{arizmendi2015relations}.
	Further, for any partition $\pi$, the quantity $B_{\pi}$ is defined in terms of a product of Boolean joint cumulants in \cite[Section 2, page 62]{arizmendi2015relations}.
	Let us finally remark that our definitions \eqref{eq: Def_Cumulant} and \eqref{eq: Def_Boolean} for classical and Boolean joint cumulants agree with the definitions used in \cite{arizmendi2015relations}: take $\pi = \{\{1,\ldots,k \}\}$ and $\varphi = \bbE$ in \cite[(2.10) \& (2.12)]{arizmendi2015relations}.
	Hence, in our notation, \cite[Corollary 1.6]{arizmendi2015relations} states that
	\begin{equation}
		\kappa(V_{\sort(i)_1,\gamma_1},\ldots,V_{\sort(i)_k,\gamma_k}) = \sum_{\rho \in \cS_k: \rho(1) = 1} (-1)^{\#\pi_\rho-1}  \prod_{\cJ\in \pi_\rho} b(V_{\sort(i)_j,\gamma_j} : j\in \cJ).  \label{eq: ClassicalToBool}
	\end{equation}
	Combine \eqref{eq: BVH_cumulantexpansion}--\eqref{eq: ClassicalToBool} to find the desired result.
\end{proof}

The reason why Boolean cumulants are useful in our Markovian setting is due to the following expression.
This identity can also be found in \cite[(1.62)]{saulis1991limit}.
\begin{proposition}\label{prop: BooleanMarkov}
	Consider a Markovian sequence of random variables $W_1,\ldots,W_k$ with values in standard Borel spaces $\cW_1,\ldots,\cW_k$ as well as measurable functions $g_i:\cW_i \to \bbR$.
	Then, denoting $Y_i =g_i(W_i)$ for any $i\in \{1,\ldots,k\}$,
	\begin{equation}
		b(Y_1,\ldots,Y_k)\label{eq: Prop_BooleanMarkov} = \negsp \int_{\cW_1}\negsp \negsp \cdots \negsp\int_{\cW_k}\negsp \Bigl\{ \prod_{i=2}^k g_i(w_i) \bigl(\intd\bbP_{W_i \mid W_{i-1}=w_{i-1}}(w_{i}) - \intd\bbP_{W_i}(w_i) \bigr)  \Bigr\} g_1(w_1) \intd\bbP_{W_1}(w_1).
	\end{equation}
\end{proposition}
\begin{proof}
	Expand the right-hand-side of \eqref{eq: Prop_BooleanMarkov} into a sum with $2^{k-1}$ terms and note that each of these terms corresponds to one of the terms in \eqref{eq: Def_Boolean}.
	For instance the term $(-1)\int_{\cW_1} \int_{\cW_2} \int_{\cW_3} g_1(w_1)g_2(w_2)g_3(w_3)\intd\bbP_{W_3}(w_3) \intd \bbP_{W_2 \mid W_1 = w_1}(w_2)\intd \bbP_{W_1}(w_1)$ in the expansion of \eqref{eq: Prop_BooleanMarkov} when $k=3$ corresponds to the term $(-1) \bbE[Y_1Y_2]\bbE[Y_3]$ in \eqref{eq: Def_Boolean}.
\end{proof}

\subsubsection{Properties of the \texorpdfstring{$\psi$}{psi}-dependence coefficient}\label{sec: psiprops}
Note the factors with $i\geq 2$ in the Boolean cumulant expression \eqref{eq: Prop_BooleanMarkov} provide a suppression when $\bbP_{W_i \mid W_{i-1}=w_{i-1}} \allowbreak\approx \bbP_{W_i}$.
That is, the Boolean cumulant tends to be small if the Markovian sequence is almost a sequence of independent random variables.
We will exploit this fact together with our assumption that the $\psi$-dependence in the Markov chain is decaying by using the following properties:
\begin{proposition}\label{prop: PsiRadonNikodym}
	Assume that $V, W$ are random variables taking values in standard Borel spaces such that $\psi(V , W) < \infty$.
	Then, $\bbP_{V,W}$ is absolutely continuous with respect to $\bbP_{V}\otimes \bbP_{W}$, and for $(\bbP_{V}\otimes \bbP_W)$-almost every $(v,w)$,
	\begin{equation}
		1 - \psi(V , W)\leq \frac{\intd \bbP_{V,W}}{\intd (\bbP_V \otimes \bbP_W)}(v,w) \leq 1 + \psi(V , W). \nonumber
	\end{equation}
\end{proposition}

\begin{proposition}\label{prop: ExponentialPsi}
	Let $Z$ be a Markovian sequence of random variables.
	Then, for any $i\in \{1,\ldots,n \}$ and $j \geq \Psi(Z)$,
	\begin{equation}
		\psi(Z_{i+j} , Z_i) \leq  \Bigl(\frac{1}{4}\Bigr)^{\lfloor j/\Psi(Z) \rfloor}. \nonumber
	\end{equation}
\end{proposition}
Both properties follow from the definition \eqref{eq: Def_PsiDependence} by direct computations using the Radon--Nikodym theorem and the Markov property.
This is classical, dating back to the 1963 work of Blum, Hanson, and Koopmans \cite[Pages 8--10]{blum1963strong}, so we omit the details.\footnote{Detailed computations can also be found in the first arXiv version of this paper; see arXiv:2307.11632v1.}

\subsubsection{Encoding suppression in random variables}

If we were interested in results for real-valued random variables, then we could now exploit our assumptions regarding decay of dependence by applying H\"older's inequality to \eqref{eq: Prop_BooleanMarkov}; see \eg \cite[Chapter 4]{saulis1991limit} and \cite[Section 10]{feray2018weighted} for such arguments.
Our primary interest is however in random matrices, and we only pass by the real-valued random variables as an intermediate stage.
This makes it so that we have to postpone the application of estimates.
To this end, we will encode the decay of dependence in scalar-valued random variable which allows us to maintain exact equalities, leading to a practical expansion for the rate of change along the interpolation in Proposition \ref{prop: MarkovNewVariables}.
\begin{proposition}\label{prop: BookkeepingLambdaPsi}
	Let $W_1,\ldots,W_k$ be a Markovian sequence of random variables taking values in standard Borel spaces $\cW_1,\ldots,\cW_k$, respectively.
	Then, there exist random variables $W_1',\ldots,W_k'$ and $\Delta'$ with the following three properties:
	\begin{description}
		\item[\textbf{(1) Marginal distribution:}] \label{item: BookkeepingLambdaPsi_Distribution} For any fixed $j\leq k$, it holds that $W_j'$ has the same law as $W_{j}$.
		\item[\textbf{(2) Suppression:}] \label{item: BookkeepingLambdaPsi_Lambda} The random variable $\Delta'$ takes values in $\bbR$ and satisfies that almost surely
		      \begin{equation}
			      \lvert \Delta' \rvert \leq 2^{k-1} \prod_{j=2}^k \min\Bigl\{1, \psi(W_{j} , W_{j - 1})\Bigr\}. \nonumber
		      \end{equation}
		      When $k = 1$, this bound should be understood as the statement that $\lvert \Delta' \rvert \leq 1$.
		\item[\textbf{(3) Expression for Boolean cumulants:}] \label{item: BookkeepingLambdaPsi_Expansion}For any sequence of real random variables $Y_1,\ldots,\allowbreak Y_k$ associated to the Markovian sequence, \ie with $Y_i = g_i(W_i)$ for certain $g_i:\cW_i\to \bbR$,
		      \begin{equation}
			      b(Y_1,\ldots,Y_k) =  \bbE[ \Delta' Y_{1}' \cdots Y_{k}'], \nonumber
		      \end{equation}
		      where $Y_{i}' \de g_i(W_{i}')$.
	\end{description}
\end{proposition}
\begin{proof}
	Let us partition the index set as $\{1,\ldots,k \} = \{1 \} \cup \cI_1 \cup \cI_2$ with
	\begin{equation}
		\cI_1 \de \{i\geq 2: \psi(W_i, W_{i-1}) < 1  \}\ \text{ and }\
		\cI_2 \de \{i\geq 2:\psi(W_i , W_{i-1}) \geq 1  \}.
	\end{equation}
	The general idea of the subsequent argument is to execute a change--of--measure on all factors with $i\in \cI_1$ in \eqref{eq: Prop_BooleanMarkov} and to expand all remaining factors in a summation.
	The random variable $\Delta'$ will then arise from the Radon--Nikodym derivative in the change--of--measure.

	We start by introducing some notation for bookkeeping purposes.
	For any $i\in \{1,\ldots,k \}$ and $\alpha_i \in \{0,1 \}$, let $\bbQc{\alpha_i}_{W_i \mid W_{i-1}}: \cB(\cW_i)\times \cW_{i-1} \to \bbR$ denote the regular conditional probability measure defined for every measurable $E\subseteq \cW_i$ and $w_{i-1} \in \cW_{i-1}$ by
	\begin{equation}
		\bbQc{\alpha_i}_{W_i \mid W_{i-1} = w_{i-1}}(E) \de \begin{cases}
			\bbP_{W_i}(E)                        & \text{ if }i\in \cI_1 \cup \{1 \},               \\
			\bbP_{W_i \mid W_{i-1} = w_{i-1}}(E) & \text{ if }\alpha_i = 0 \text{ and }i \in \cI_2, \\
			\bbP_{W_i}(E)                        & \text{ if }\alpha_i = 1 \text{ and }i\in \cI_2.  \\
		\end{cases}\label{eq: Def_bbQa}
	\end{equation}
	Then, expanding the factors with $i\in \cI_2$ in the expression \eqref{eq: Prop_BooleanMarkov} for the Boolean cumulant into a sum yields that
	\begin{align}
		 & b(Y_1,\ldots,Y_k) =\sum_{\alpha_i\in \{0,1 \} : i\in \{1 \}\cup \cI_2 }\int_{\cW_1}\negsp \cdots\negsp \int_{\cW_k}\frac{(-1)^{\#\{i \in \cI_2: \alpha_i = 1\}}}{2}\times\label{eq: BooleanIntegrals1}                   \\
		 & \prod_{i\in \{1 \}\cup\cI_2}\negsp g_i(w_i) \intd \bbQc{\alpha_i}_{W_i\mid W_{i-1} = w_{i-1}}(w_i)\prod_{i\in \cI_1} g_i(w_i) \bigl(\intd\bbP_{W_i \mid W_{i-1}=w_{i-1}}(w_{i}) - \intd\bbP_{W_i}(w_i) \bigr). \nonumber
	\end{align}
	The factor $1/2$ here accounts for double counting: recall that the factor with $i=1$ in \eqref{eq: Prop_BooleanMarkov} does not involve a difference of conditional and unconditional probability measures.

	We next apply a change--of-measure to encode the decay of dependence.
	For any $i\in \cI_1$ let us define a measurable function $\delta_{i}:\cW_i \times \cW_{i-1} \to \bbR$ by
	\begin{equation}
		\delta_{i}(w_i, w_{i-1}) \de \frac{\intd \bbP_{W_i, W_{i-1} }}{\intd (\bbP_{W_i}\otimes \bbP_{W_{i-1}})}(w_i,w_{i-1}) - 1. \label{eq: Def_deltai}
	\end{equation}
	Note that \eqref{eq: Def_deltai} is well-defined due to Proposition \ref{prop: PsiRadonNikodym} and that for $(\bbP_{W_i}\otimes \bbP_{W_{i-1}})$-almost every $w_i, w_{i-1}$,
	\begin{equation}
		-\psi(W_i ,  W_{i-1}) \leq \delta_{i}(w_i, w_{i-1}) \leq \psi(W_{i} , W_{i-1}). \label{eq: deltaibound}
	\end{equation}
	Further, by the definition \eqref{eq: Def_deltai} of $\delta_i$ in terms of a Radon--Nikodym derivative, we have that $\intd \bbP_{W_{i}\mid W_{i-1} = w_{i-1} }(w_i) - \intd \bbP_{W_i}(w_i) =  \delta_{i}(w_i,w_{i-1}) \intd \bbP_{W_i}(w_i)$.
	Hence, using this on the factors with $i\in \cI_1$ in \eqref{eq: BooleanIntegrals1} and subsequently using that $\bbP_{W_i} = \bbQc{\alpha_i}_{W_i\mid W_{i-1}}$ for any $i\in \cI_1$,
	\begin{equation}
		b(Y_1,\ldots,Y_k) \label{eq: FinalB} = \negsp \negsp \negsp \sum_{\alpha_1,\ldots,\alpha_{k} \in \{0,1 \}}\int_{\cW_1}\negsp \negsp \cdots\negsp \int_{\cW_k} \negsp\negsp  \delta (w_1,\ldots,w_k)    \prod_{i = 1}^k g_i(w_i) \intd \bbQc{\alpha_i}_{W_i\mid W_{i-1} =
		w_{i-1}}(w_i),
	\end{equation}
	where $\delta:\cW_1\times \ldots \cW_k\to \bbR$ is defined by
	\begin{equation}
		\delta(w_1,\ldots,w_k) \de \frac{(-1)^{\#\{i\in \cI_2: \alpha_i = 1\}}}{2^{\# \cI_1 + 1}}\prod_{i\in \cI_1} \delta_i(w_i, w_{i -1}). \label{eq: Def_dela}
	\end{equation}
	The additional factor $(1/2)^{\#\cI_1}$ relative to \eqref{eq: BooleanIntegrals1} accounts for double counting: note that the sum in \eqref{eq: BooleanIntegrals1} runs only over $\alpha_i$ with $i\in \{1 \}\cup \cI_2$, while \eqref{eq: FinalB} also allows $i\in \cI_1$.

	For any $\alpha = (\alpha_1,\ldots,\alpha_k)$ as in \eqref{eq: FinalB}, let us define $(\Wc{\alpha}_1,\ldots,\Wc{\alpha}_k)$ to be random variables with joint law given by $\prod_{i=1}^k\intd\bbQc{\alpha_i}_{W_i\mid W_{i-1}}$.
	Further, define $\Delc{\alpha} \de \delta(\Wc{\alpha}_1,\ldots,\Wc{\alpha}_k)$.
	Recall from statement \nref{item: BookkeepingLambdaPsi_Expansion}{(3)} of Proposition \ref{prop: BookkeepingLambdaPsi} that $Y_i = g_i(W_i)$.
	Hence, \eqref{eq: FinalB} yields that
	\begin{equation}
		b(Y_1,\ldots,Y_k) = \sum_{\alpha \in \{0,1 \}^k} \bbE[\Delc{\alpha} g_1(\Wc{\alpha}_1) \cdots g_k(\Wc{\alpha}_k) ].
	\end{equation}
	To remove the dependence on $\alpha$ from the notation, we can encode the summation into a mixture of measures.
	Let $A \sim \Unif\{0,1 \}^k$ be uniformly distributed, independent from the preceding data.
	Then, it holds with $W_i' \de \Wc{A}_i$ and $\Delta' \de 2^k\Delc{A}$ that
	\begin{equation}
		b(Y_1,\ldots,Y_k) = \bbE[\Delta' g_1(W_1') \cdots g_k(W_k') ]. \label{eq: Conclusion_BooleanIntegrals}
	\end{equation}
	It remains to verify that the properties claimed in Proposition \ref{prop: BookkeepingLambdaPsi} are satisfied.

	First, it follows from \eqref{eq: Def_bbQa} that $\prod_{i}\bbQc{\alpha_i}_{W_i \mid W_{i-1}}$ has marginal distribution $\bbP_{W_i}$ at the $i$th coordinate; to verify this use induction on $i$ together with the Markov property and Bayes' theorem.
	Consequently, $\Wc{\alpha}_i$ has the same law as $W_i$ for every $\alpha$, and it follows that the same holds for the mixture $W_i'$.
	This yields the distributional property in item \nref{item: BookkeepingLambdaPsi_Distribution}{(1)}.
	Second, note that $\Vert \Delta' \Vert_{\Linf} \leq 2^k\max_{\alpha} \Vert \Delc{\alpha} \Vert_{\Linf}$.
	Item \nref{item: BookkeepingLambdaPsi_Lambda}{(2)} hence follows from \eqref{eq: deltaibound} and \eqref{eq: Def_dela} since $\#\cI_1 \geq 0$.
	Finally, the expansion in item \nref{item: BookkeepingLambdaPsi_Expansion}{(3)} is explicit in \eqref{eq: Conclusion_BooleanIntegrals}.
\end{proof}
We next combine Propositions \ref{prop: BooleanInterpolation} and \ref{prop: BookkeepingLambdaPsi}.
First, however, let us set up some notation.
Given a smooth function $g:\bbC^{d\times d}\to \bbC$, let it be understood that $\partial_{\bB}g:\bbC^{d\times d} \to \bbC$ denotes the directional derivative of $g$ in the direction of a matrix $\bB \in \bbC^{d\times d}$:
\begin{equation}
	(\partial_{\bB}g)(\bM) \de \lim_{\varepsilon \to 0} \frac{g(\bM  + \varepsilon \bB) - g(\bM)}{\varepsilon}.
\end{equation}
Recall that the expansion in Proposition \ref{prop: BooleanInterpolation} involves a sum over permutations $\rho \in \cS_k$ with $\rho(1)=1$.
To bookkeep these permutations as well as indices $i_1,\ldots,i_k\in \{1,\ldots,n\}$ for the Markovian sequence, we define an index set for any $k,n\geq 1$ by
\begin{equation}
	\cI_{k,n} \de \{\rho \in \cS_k : \rho(1) = 1 \} \times \{1,\ldots,n \}^k.\label{eq: Def_Ikn}
\end{equation}
Proposition \ref{prop: MarkovNewVariables} below establishes an expansion for the rate--of--change along the interpolation $\bS(t)$ from \eqref{eq: Def_S(t)} using new random variables whose law depends on an index $(\rho,i)\in \cI_{k,n}$.
The exact joint distribution is however not required in the subsequent arguments, so we only explicitly state those properties that we need and refer to the proof for the construction.
\begin{proposition}\label{prop: MarkovNewVariables}
	Adopt the assumptions and notation from Section \ref{sec: MarkovModel}.
	That is, let $Z_1,\ldots,Z_n$ be a Markovian sequence and consider the associated centered random matrices $\bX_i = \bF_i(Z_i)$ for $i\geq 1$ as well as $\bS=\bX_0 + \sum_{i=1}^n\bX_i$ and a Gaussian model $\bG$.

	Then, there exist random variables $(\Zci_{i,1},\ldots,\Zci_{i,k}, \Delci_i )_{(\rho,i)\in \cI_{k,n}}$ for every $k\geq 3$ that satisfy the following four properties for any $(\rho,i) \in \cI_{k,n}$ and $j\in \{1,\ldots,k\}$:
	\begin{description}
		\item[\textbf{(1) Marginal distribution:}]\label{item:Marginal} It holds that $\Zci_{i,j}$ has the same law as $Z_{i_j}$.
		\item[\textbf{(2) Independence:}]\label{item:Independence} The tuple of random variables $(\Zci_{i,1},\ldots,\allowbreak \Zci_{i,k}, \allowbreak \Delci_i)$ is independent of the Markovian sequence $Z_1,\ldots,Z_n$ and the Gaussian model $\bG$.
		\item[\textbf{(3) Suppression:}]\label{item:Suppression} The random variable $\Delci_i$ takes values in $\bbR$ and satisfies that almost surely
		      \begin{equation*}
			      \lvert \Delci_i \rvert \leq 2^{k-\#\pi_\rho}\prod_{\cJ \in \pi_\rho} \prod_{j = 2}^{\# \cJ}\min\Bigl\{ 1,\psi\bigl(Z_{\sort(i)_{\cJ(j)}} , Z_{\sort(i)_{\cJ(j-1)}}\bigr) \Bigr\}
		      \end{equation*}
		      where we recall that $\cJ(j)$ is the $j$th smallest element in $\cJ$, and we use the convention that an empty product yields unity when $\#\cJ = 1$.
		\item[\textbf{(4) Expansion for Gaussian interpolations:}]\label{item:Expansion}
		      Define random matrices by $\bXci_{i,j}\de \bF_{i_j}(Z_{i,j}^{\rho})$.
		      Then, for every polynomial $g:\bbC^{d\times d}\to \bbC$ and $t\in [0,1]$,
		      \begin{equation*}
			      \dd{t} \bbE\bigl[g\bigl(\bS(t)\bigr)\bigr] = \frac{1}{2}\sum_{k=3}^\infty \frac{t^{\frac{k}{2}-1}}{(k-1)!}  \bbE\Bigl[ \sum_{(\rho,i)\in \cI_{k,n}} \Delci_i \partial_{\bXci_{i,1}} \cdots \partial_{\bXci_{i,k}} g(\bS(t))\Bigr].
		      \end{equation*}
	\end{description}
\end{proposition}
\begin{proof}
	We may assume without loss of generality that $\bG = \bX_0 + \sum_{i=1}^n \bG_i$ where $(\bG_1,\ldots,\bG_n)$ is a Gaussian model for the rectangular matrix $(\bX_1,\ldots,\bX_n)$.\footnote{For instance, because the independence from all the other random variables in Proposition \ref{prop: MarkovNewVariables} implies that the statement can only depend on the law of $\bG$.
		Alternatively, if one insists on maintaining the original matrix, one can use the joint distribution of $(\bG_1,\ldots,\bG_n)$ and $\bX_0 + \sum_{i}\bG_i$ to define the new matrices with a coupling to $\bG$.}
	Then, viewing $\bS(t)$ as a function of these rectangular matrices and identifying complex matrices with vectors in $\bbR^{2d^2}$ by taking the entries' real and imaginary parts, Proposition \ref{prop: BooleanInterpolation} yields that
	\begin{align}
		\dd{t}\bbE\bigl[g\bigl(\bS(t)\bigr)\bigr] =  {} & {}  \frac{1}{2}\sum_{k=3}^\infty\frac{t^{\frac{k}{2}-1}}{(k-1)!}\sum_{i\in \{1,\ldots,n \}^k}\sum_{\rho \in \cS_k:\rho(1) = 1} \sum_{\Gamma_1,\ldots,\Gamma_k \in \{1,\ldots d \}^2\times \{\operatorname{Re}, \operatorname{Im}\}} (-1)^{\#\pi_\rho- 1} \nonumber \\
		                                                & \times\Bigl\{\prod_{\cJ \in \pi_\rho} b(\bX_{\sort(i)_j,\Gamma_j}: j\in \cJ) \Bigr\}\bbE\Bigl[\frac{\partial^k g}{\partial \bM_{\Gamma_1}  \cdots \partial \bM_{\Gamma_k}}(\bS(t)) \Bigr]\label{eq:AngryNut}
	\end{align}
	where the $\Gamma$-subscripts refer to the entries' real and imaginary parts.
	For instance, given $\Gamma = (v,w,\operatorname{Re})$ with $v,w\in \{1,\ldots,d \}$, we use $\bX_{i_j,\Gamma}$ to refer to the real part of the $vw$th entry of the matrix $\bX_{i_j}$.
	Similarly, the partial derivatives in \eqref{eq:AngryNut} are taken in the direction of the real/imaginary part of the corresponding argument of the function $g:\bbC^{d\times d} \to \bbC$.

	To rewrite this expansion, we next apply Proposition \ref{prop: BookkeepingLambdaPsi} to the Boolean cumulants and then exploit that expectation factorizes over products of independent random variables.
	To make this precise, define random variables $(\Zci_{i,1},\ldots,\Zci_{i,k}, \{ \Delci_{i,\cJ}: \cJ\in \pi_{\rho}\}  )_{(\rho,i)\in \cI_{k,n}, k\geq 3}$ that are independent of $Z_1,\ldots,Z_n$ and $\bG$ with joint law specified by the following properties:
	\begin{enumerate}
		\item Given $(\rho,i) \in \cI_{k,n}$ and some part $\cJ \in \pi_{\rho}$ in the partition induced by $\rho$, set
		      \begin{equation}
			      (W_1 ,\ldots,W_{\#\cJ}) \de  (Z_{\sort(i)_{\cJ(1)}}, \ldots, Z_{\sort(i)_{\cJ(\#\cJ)}}). \label{eq:VagueIce}
		      \end{equation}
		      Then, we define the joint law of $\Delci_{i,\cJ}$ and the random variables $\Zci_{i,j}$ with $j\in \cJ$ to be given by the random variables resulting from Proposition \ref{prop: BookkeepingLambdaPsi} applied to $W_1,\ldots,W_{\#\cJ}$, with the new variables ordered in such a fashion to remove the $\operatorname{sort}(\cdot)$-operation from \eqref{eq:VagueIce}.
		      More precisely, if $\mu_i \in \cS_{k}$ is a permutation such that $i_{\mu_i(j)} = \sort(i)_{j}$ for every $j\leq k$, then
		      \begin{equation}
			      (\Zci_{i,\mu_i(\cJ(1))}, \ldots, Z_{i, \mu_i(\cJ(\#\cJ))}, \Delci_{i,\cJ}) \sim   (W_1',\ldots, W_{\# \cJ}', \Delta').  \label{eq:TinyFog}
		      \end{equation}
		\item The random variables $\Zci_{i,j}$ and $\Delci_{i,\cJ}$ associated with different $(\rho,i)\in \cI_{k,n}$ or different parts $\cJ\in \pi_\rho$ are independent.
		      More precisely, if $T_{i,\rho,\cJ}$ is the tuple on the left-hand side of \eqref{eq:TinyFog}, then the variables $(T_{i,\rho,\cJ}:k\geq 3, (i,\rho)\in \cI_{k,n}, \cJ\in \pi_{\rho} )$ are jointly independent.
	\end{enumerate}
	Recall that $\bX_{i_j} = \bF_{i_j}(Z_{i_j})$.
	Hence, for any fixed $(\rho,i) \in \cI_{k,n}$, the reference to Proposition \ref{prop: BookkeepingLambdaPsi} in the definitions \eqref{eq:VagueIce}--\eqref{eq:TinyFog} ensures that for any $\Gamma_1,\ldots,\Gamma_k$ and any $\cJ\in \pi_{\rho}$,
	\begin{equation}
		b\bigl(\bX_{\sort(i)_j,\Gamma_j}: j\in \cJ\bigr) = \bbE\biggl[\Delci_{i,\cJ}\prod_{j\in \cJ} \bXci_{i, \mu_i(j), \Gamma_{j}}  \biggr] \label{eq:RipeQuiz}
	\end{equation}
	where we recall from item \nref{item:Expansion}{(4)} in Proposition \ref{prop: MarkovNewVariables} that $\bXci_{i_j} = \bF_{i_j}(\Zci_{i,j})$.
	Now, using \eqref{eq:RipeQuiz} and the fact that the variables were defined to be independent of $Z_1,\ldots,Z_n$ and $\bG$ as well as the joint independence in the second property of the definition,
	\begin{align}
		\Bigl(\prod_{\cJ \in \pi_\rho} b(\bX_{\sort(i)_j,\Gamma_j}: j\in \cJ) {} & {}\Bigr)\bbE\Bigl[\frac{\partial^k g}{\partial \bM_{\Gamma_1}  \cdots \partial \bM_{\Gamma_k}}(\bS(t)) \Bigr]\label{eq:RipeGum}                                                                                           \\
		                                                                         & = \bbE\biggl[\Bigl(\prod_{\cJ \in \pi_\rho} \Delci_{i,\cJ} \prod_{j\in \cJ} \bXci_{i, \mu_i(j), \Gamma_{j}}  \Bigr)\frac{\partial^k g}{\partial \bM_{\Gamma_1}  \cdots \partial \bM_{\Gamma_k}}(\bS(t)) \biggr].\nonumber
	\end{align}
	Hence, if we define $\Delci_i \de  (-1)^{\#\pi_\rho- 1} \prod_{\cJ\in \pi_\rho}\Delci_{i,\cJ}$, then
	\begin{align}
		 & (-1)^{\#\pi_\rho- 1} \sum_{\Gamma_1,\ldots,\Gamma_k \in \{1,\ldots d \}^2 \times \{\operatorname{Re}, \operatorname{Im} \}}\Bigl(\prod_{\cJ \in \pi_\rho} b(\bX_{\sort(i)_j,\Gamma_j}: j\in \cJ) \Bigr)\bbE\Bigl[\frac{\partial^k g}{\partial \bM_{\Gamma_1}  \cdots \partial \bM_{\Gamma_k}}(\bS(t)) \Bigr]\nonumber                 \\
		 & \quad = \sum_{\tilde{\Gamma}_1,\ldots,\tilde{\Gamma}_k \in \{1,\ldots d \}^2 \times \{\operatorname{Re}, \operatorname{Im} \}} \bbE\biggl[\Delci_i  \prod_{j = 1}^k \bXci_{i, j, \tilde{\Gamma}_{j}}\frac{\partial^k g}{\partial \bM_{\tilde{\Gamma}_1}  \cdots \partial \bM_{\tilde{\Gamma}_k}}(\bS(t)) \biggr]\label{eq:PinkDragon} \\
		 & \quad = \bbE\Bigl[\Delci_i  \partial_{\bXci_{i, 1}}\cdots \partial_{\bXci_{i, k}}g(\bS(t)) \Bigr]\nonumber
	\end{align}
	where the first equality renumbered the indices by taking $\tilde{\Gamma}_j \de \Gamma_{\mu_i^{-1}(j)}$, and the second equality follows by recognizing the standard entry-wise decomposition for a directional derivative.
	Substitution of \eqref{eq:PinkDragon} in \eqref{eq:AngryNut} yields the expansion in item \nref{item:Expansion}{(4)} of Proposition \ref{prop: MarkovNewVariables}.

	It remains to verify items \nref{item:Marginal}{(1)} to \nref{item:Suppression}{(3)}.
	The marginal distribution in item \nref{item:Marginal}{(1)} follows from corresponding property in Proposition \ref{prop: BookkeepingLambdaPsi} after chasing the indices through \eqref{eq:VagueIce} and \eqref{eq:TinyFog}.
	The independence in item \nref{item:Independence}{(2)} is explicit in the definitions; see the paragraph preceding \eqref{eq:VagueIce}.
	Finally, by \eqref{eq:TinyFog} and the suppression property in Proposition \ref{prop: BookkeepingLambdaPsi}, for any $(i,\rho)$ and $\cJ\in \pi_{\rho}$,
	\begin{equation}
		\lvert \Delci_{i,\cJ} \rvert  \leq 2^{\#\cJ -1} \prod_{j=2}^k \min\bigl\{1,\psi\bigl(Z_{\sort(i)_{\cJ(j)}} , Z_{\sort(i)_{\cJ(j -1)}}\bigr)\bigr\}.
	\end{equation}
	The desired suppression property in item \nref {item:Suppression}{(3)} hence follows using that
	$
		\lvert \Delci_i \rvert  = \prod_{\cJ \in \pi_{\rho}} \lvert  \Delci_{i,\cJ} \rvert
	$
	and that $\sum_{\cJ\in \pi_{\rho}}(\#\cJ -1)= k-\#\pi_\rho$ since $\pi_\rho$ is a partition of $\{1,\ldots,k \}$.
\end{proof}

\subsection{Combining classical--to--Boolean combinatorics and the decay of dependence}\label{sec: DeltasPsi}
Passing through Boolean cumulants in the proofs leading up to Proposition \ref{prop: MarkovNewVariables} allowed us to efficiently exploit the Markovian structure.
This however comes at a cost: we now have to understand the combinatorics associated with the classical--to--Boolean relations.

Specifically, our goal in this section is to estimate the total contribution of the scalar-valued random variables $\Delci_i$ that encode the suppression due to the decay of dependence.
Regarding the individual random variables, we have the following upper bound:
\begin{lemma}\label{lem: Delci_Individual}
	With notation as in Proposition \ref{prop: MarkovNewVariables}, it holds for every $(\rho,i) \in \cI_{k,n}$ that
	\begin{equation}
		\Vert \Delci_i   \Vert_{\Linf} \leq 2^{k-1}\exp\Bigl(-\sum_{j=2}^k \bb1\{\sort(i)_{\rho(j)} > \sort(i)_{\rho(j-1)}\}\Bigl\lfloor \frac{\sort(i)_{\rho(j)} - \sort(i)_{\rho(j-1)}}{\Psi(Z)}\Bigr\rfloor  \Bigr).\nonumber
	\end{equation}
\end{lemma}
\begin{proof}
	Due to Proposition \ref{prop: ExponentialPsi} it holds that for all $l\in \{1,\ldots,n \}$ and $t\in \{0,\ldots,n-l \}$,
	\begin{equation}
		\min\{1 ,\psi(Z_{l+t} , Z_l) \} \leq \Bigl(\frac{1}{4}\Bigr)^{\lfloor t/\Psi(Z) \rfloor} \leq \exp\Bigl(-\Bigl\lfloor \frac{t}{\Psi(Z)} \Bigr\rfloor\Bigr).
	\end{equation}
	Hence, using item \nref{item:Suppression}{(3)} in Proposition \ref{prop: MarkovNewVariables} and that $\#\pi_{\rho}\geq 1$,
	\begin{equation}
		\Vert \Delci_i \Vert_{\Linf} \leq 2^{k-1}
		\prod_{\cJ \in \pi_\rho} \prod_{j = 2}^{\# \cJ}\exp\Bigl(-\Bigl\lfloor \frac{\sort(i)_{\cJ(j)} - \sort(i)_{\cJ(j-1)}}{\Psi(Z)} \Bigr\rfloor \Bigr).\label{eq:HauntedLion}
	\end{equation}

	Recall from Section \ref{sec: BooleanJointCumulants} that $\pi_\rho$ denotes the partition of increasing runs associated with the permutation $\rho\in \cS_k$.
	So, for instance, the part $\cJ_1 \in \pi_{\rho}$ that contains $\rho(1)$ is of the form $\{\rho(1), \rho(2),\ldots, \rho(\ell_1)\}$ where $\ell_1\geq 1$ is the least value with $\rho(\ell_1+1) < \rho(\ell_1)$, or where $\ell_1 = k$ if no such value exists.
	Hence, bringing the product inside the exponential,
	\begin{equation}
		\prod_{j=2}^{\#\cJ_1} \exp\Bigl(-\Bigl\lfloor \frac{\sort(i)_{\cJ_1(j)} - \sort(i)_{\cJ_1(j-1)}}{\Psi(Z)} \Bigr\rfloor \Bigr) = \exp\Bigl(-\sum_{j=2}^{\ell_{1}} \Bigl\lfloor \frac{\sort(i)_{\rho(j)} - \sort(i)_{\rho(j-1)}}{\Psi(Z)}\Bigr\rfloor  \Bigr),\nonumber
	\end{equation}
	where it is to be understood that the right-hand side is unity when $\ell_1 = 1$.
	Continuing sequentially, the part that contains $\rho(\ell_{1}+1)$ is $\cJ_{2} \de \{\rho(\ell_{1}+1),\ldots,\rho(\ell_{2})\}$ where $\ell_{2} \geq \ell_1 + 1$ is the least value with $\rho(\ell_{2} +1) < \rho(\ell_{2})$, or where $\ell_2 = k$ if no such value exists.
	Hence,
	\begin{align}
		\prod_{m=1}^2{} & {}\prod_{j=2}^{\#\cJ_m} \exp\Bigl(-\Bigl\lfloor \frac{\sort(i)_{\cJ_m(j)} - \sort(i)_{\cJ_m(j-1)}}{\Psi(Z)} \Bigr\rfloor \Bigr)                                                                                                                      \\
		                & = \exp\Bigl(-\sum_{j=2}^{\ell_{1}} \Bigl\lfloor \frac{\sort(i)_{\rho(j)} - \sort(i)_{\rho(j-1)}}{\Psi(Z)}\Bigr\rfloor + \sum_{j=\ell_1+2}^{\ell_2} \Bigl\lfloor \frac{\sort(i)_{\rho(j)} - \sort(i)_{\rho(j-1)}}{\Psi(Z)}\Bigr\rfloor\Bigr)\nonumber \\
		                & =  \exp\Bigl(-\sum_{j=2}^{\ell_2} \bb1\{\rho(j) > \rho(j-1)\}\Bigl\lfloor \frac{\sort(i)_{\rho(j)} - \sort(i)_{\rho(j-1)}}{\Psi(Z)}\Bigr\rfloor  \Bigr), \nonumber
	\end{align}
	where the final equality used the definition of $\ell_1$ and $\ell_2$.
	Continue in this fashion until $\ell_m$ reaches $k$.
	Then, using that $\bb1\{\rho(j) >\rho(j-1)\}\bb1\{\sort(i)_{\rho(j)} \neq \sort(i)_{\rho(j-1)} \} = \bb1\{\sort(i)_{\rho(j)} > \sort(i)_{\rho(j-1)} \}$ because $\operatorname{sort}(i)$ is nondecreasing concludes the proof.
\end{proof}
We next sum over $\rho$ and $i$ to quantify the total contribution of the random variables $\Delci_i$.
Specifically, we study a restricted sum wherein one of the coordinates of $i \in \{1,\ldots,n \}^k$ is kept fixed.
(This restriction is used in the proof of Lemma \ref{lem: Lem_S(t)DerivativeBound_Psi}.)
We start with initial simplifications that only exploit the general structure of the right-hand side of Lemma \ref{lem: Delci_Individual}.

\subsubsection{General-purpose simplifications}
For any sequence $i \in \bbR^k$ and permutation $\rho \in \cS_k$, let us abbreviate $i_\rho$ for the permuted sequence $(i_{\rho(j)})_{j=1}^k$.

\begin{lemma}\label{lem: Swap}
	Fix some  $I\in \{1,\ldots,n \}$ and $J\in \{1,\ldots,k \}$.
	Then, for every nonnegative function $E:\{1,\ldots,n \}^k \to \bbR_{\geq 0}$,
	\begin{equation}
		\sum_{\substack{\rho\in \cS_k\\ \rho(1) =1}} \sum_{\substack{i\in \{1,\ldots,n \}^k\\ i_J = I}} \negsp\negsp E\bigl( \sort(i)_\rho \bigr) \leq  k!  \sum_{T=1}^k\negsp \sum_{\substack{i\in \{1,\ldots,n \}^k\\ i_{T} = I,\ i_1 = \min(i_j:j\leq k)  }}\negsp\negsp\negsp\negsp\negsp   E(i).\label{eq:WeepyVan}
	\end{equation}
\end{lemma}
\begin{proof}
	Fix some arbitrary sequence $t\in \{1,\ldots,n \}^k$.
	We verify that \eqref{eq:WeepyVan} holds for the corresponding basis function $E(\cdot) \de \bb1\{\cdot = t \}$.
	The case with arbitrary nonnegative $E$ then follows since both sides of the desired inequality are linear in $E$.

	In particular, sorting $i$ and thereafter applying an arbitrary permutation $\rho$ with $\rho(1) = 1$ brings the minimal element to the front.
	Consequently, the left-hand side of \eqref{eq:WeepyVan} can only be nonzero if $t_1 = \min(t_j:j\leq k)$ and $t_T = I$ for some $T \leq k$.
	The desired bound \eqref{eq:WeepyVan} is trivial if the left-hand side is zero, so we may assume that $t$ satisfies the latter conditions.
	Then, there is at least one nonzero term in the right-hand side of \eqref{eq:WeepyVan} and it suffices to show that the left-hand side is $\leq k!$.

	To this end, note that it follows from the assumption that $E(\cdot) = \bb1\{\cdot = t \}$ that
	\begin{align}
		\sum_{\substack{\rho\in \cS_k                                                                                                  \\ \rho(1) =1}} \sum_{\substack{i\in \{1,\ldots,n \}^k\\ i_J = I}} \negsp\negsp\negsp &E\bigl( \sort(i)_\rho \bigr)
		= \#\{(i,\rho) \in \{1,\ldots,n \}^k \times \cS_k: \sort(i)_\rho = t,\, i_J = I,\,  \rho(1)= 1 \} \nonumber                    \\[-3ex]
		 & \leq \#\{i\in \{1,\ldots,n \}^k: \sort(i) = \sort(t) \} \times \#\{\rho \in \cS_k: \sort(t)_{\rho} = t\}\label{eq:JustLeaf}
	\end{align}
	where the inequality follows by neglecting final two constraints in the first line.

	Now, consider the (right) group action of $\cS_k$ on $\{1,\ldots,k \}^n$ defined by $i*\rho \de i_{\rho}$.
	Then, the first set on the right-hand side of \eqref{eq:JustLeaf} is simply the orbit of $t$ under this group action:
	\begin{equation}
		\{i\in \{1,\ldots,n \}^k: \sort(i) = \sort(t) \} = \{t_{\rho} : \rho \in \cS_k \}.
	\end{equation}
	Moreover, the cardinality of the second set agrees with that of the stabilizer of $t$,
	\begin{equation}
		\#\{\rho \in \cS_k: \sort(t)_{\rho} = t\} = \#\{\theta \in \cS_k: t_{\theta} = t\},
	\end{equation}
	since a bijection may be defined by setting $\theta \de \mu \circ \rho$ with $\mu\in \cS_k$ a permutation satisfying $t_{\mu} = \sort(t)$.
	It hence follows from the orbit-stabilizer theorem that
	\begin{equation}
		\#\{t_{\rho} : \rho \in \cS_k \} \times \#\{\theta \in \cS_k: t_{\theta} = t\}= \# \cS_k = k!. \label{eq:GreenNose}
	\end{equation}
	Combine \eqref{eq:JustLeaf}--\eqref{eq:GreenNose} to conclude that the left-hand side of \eqref{eq:WeepyVan} is $\leq k!$, as desired.
\end{proof}

\begin{lemma}
	\label{lem: Translation}
	Fix some integer $I\in \bbZ$ and $T\in \{1,\ldots,k \}$ as well as a  nonnegative function $E:\bbZ^{k} \to \bbR_{\geq 0}$.
	Assume that $E$ is translation invariant in the sense that $E(i_1,\ldots,i_k) = E(i_1 +x,\ldots,i_k + x )$ for every $i_1,\ldots,i_k\in \bbZ$ and $x\in \bbZ$. Then,
	\begin{equation}
		\sum_{\substack{i \in \bbZ^k,\ i_T = I\\
				i_1 = \min(i_j:j\leq k)}} \negsp\negsp\negsp E(i) = \sum_{i \in \bbZ_{\geq 0}^k,\ i_1 = 0} \negsp\negsp\negsp E(i)\label{eq:WeepyGym}
	\end{equation}
\end{lemma}

\begin{proof}
	Note that $i$ runs over integer sequences in the left-hand side of \eqref{eq:WeepyGym} that may also include negative values.
	We subdivide the sum by the minimal value $\mathfrak{i}_{\min}$ attained by the sequence $i$ and subsequently use translation invariance with $x=-\mathfrak{i}_{\min}$
	\begin{equation}
		\sum_{\substack{i \in \bbZ^k,\ i_T = I\\
				i_1
				= \min(i_j:j\leq k)}} \negsp\negsp\negsp E(i)
		= \sum_{\mathfrak{i}_{\min}  =-\infty }^I\negsp \sum_{\substack{i \in \bbZ^k,\ i_T = I\\
				i_1 = \min(i_j:j\leq k) = \mathfrak{i}_{\min}}} \negsp\negsp\negsp\negsp\negsp\negsp\negsp   E(i)
		=
		\sum_{\mathfrak{i}_{\min}  =-\infty }^I
		\sum_{\substack{i \in \bbZ^k,\ i_T = I -\mathfrak{i}_{\min}\\
				i_1 = \min(i_j:j\leq k) = 0}} \negsp\negsp\negsp\negsp\negsp   E(i).\label{eq:RestlessOwl}
	\end{equation}
	Here, making a change of variables $\mathfrak{I} \de I -\mathfrak{i}_{\min}$ and rewriting the constraints in the second sum into an equivalent formulation,
	\begin{equation}
		\sum_{\mathfrak{i}_{\min}  =-\infty }^I
		\sum_{\substack{i \in \bbZ^k,\, i_T = I -\mathfrak{i}_{\min}\\
				i_1 = \min(i_j:j\leq k) = 0}} \negsp\negsp\negsp\negsp\negsp   E(i) =
		\sum_{\mathfrak{I} \geq 0} \sum_{\substack{i\in \bbZ_{\geq 0}^k\\
				i_{T} = \mathfrak{I},\ i_1 = 0 }}\negsp\negsp\negsp\negsp\negsp  E(i) = \sum_{i \in \bbZ_{\geq 0}^k,\ i_1 = 0} \negsp\negsp\negsp E(i). \label{eq:EmptyArm}
	\end{equation}
	The combination of \eqref{eq:RestlessOwl} and \eqref{eq:EmptyArm} yields \eqref{eq:WeepyGym}, as desired.
\end{proof}
\begin{corollary}\label{cor: Del_FixIJ_exp}
	For any fixed $I\in \{1,\ldots,n \}$ and $J\in \{1,\ldots,k \}$,
	\begin{equation}
		\sum_{(\rho,\alpha,i) \in \cI_{k,n}: i_J = I} \negsp \negsp\negsp  \Vert \Delci_i  \Vert_{\Linf}\leq 4^{k-1} k! \negsp \sum_{i \in \{0\} \times \bbZ_{\geq 0}^{k-1}} \negsp\negsp \exp\Bigl(-\sum_{j=2}^k \bb1\{i_j > i_{j-1}\} \Bigl\lfloor \frac{i_j - i_{j-1}}{\Psi(Z)}\Bigr\rfloor\Bigr). \label{eq:CozyMan}
	\end{equation}
\end{corollary}
\begin{proof}
	Let us define a function $E:\bbZ^{k} \to \bbR_{\geq 0}$ by
	\begin{equation}
		E(i) \de 2^{k-1}\exp\Bigl(-\sum_{j=2}^k \bb1\{i_j > i_{j-1}\} \Bigl\lfloor \frac{i_j - i_{j-1}}{\Psi(Z)}\Bigr\rfloor\Bigr).\label{eq: Def_Ei}
	\end{equation}
	Then, Lemma \ref{lem: Delci_Individual} yields that $
		\Vert \Delci_i \Vert_{\Linf} \leq  E(\sort(i)_{\rho}).$
	Hence,  recalling the definition of $\cI_{k,n}$ from \eqref{eq: Def_Ikn} and using Lemma \ref{lem: Swap},
	\begin{equation}
		\sum_{(\rho,i) \in \cI_{k,n}: i_J = I} \negsp\negsp\negsp\negsp\negsp\negsp  \Vert \Delci_i  \Vert_{\Linf} = \sum_{\substack{\rho \in \cS_k\\
				\rho(1) = 1}} \sum_{\substack{i\in \{1,\ldots,n \}^k\\
				i_J = I}}\negsp\negsp\negsp\negsp    E(\sort(i)_{\rho}) \leq k!    \sum_{T=1}^k\negsp\negsp\negsp \sum_{\substack{i\in \{1,\ldots,n \}^k\\ i_{T} = I,\ i_1 = \min(i_j:j\leq k)  }}\negsp\negsp\negsp\negsp\negsp\negsp\negsp\negsp    E(i).\label{eq:ValidZero}
	\end{equation}
	Note that $E(i)$ is translation invariant.
	Hence, by enlarging the sum and using Lemma \ref{lem: Translation},
	\begin{equation}
		\sum_{\substack{i\in \{1,\ldots,n \}^k\\ i_{T} = I,\ i_1 = \min(i_j:j\leq k)  }}\negsp\negsp\negsp\negsp\negsp\negsp\negsp\negsp    E(i) \leq \sum_{\substack{i\in \bbZ^k\\ i_{T} = I,\ i_1 = \min(i_j:j\leq k)  }}\negsp\negsp\negsp\negsp\negsp\negsp\negsp\negsp    E(i)  = \sum_{i\in \{0 \} \times \bbZ^{k-1}_{\geq 0}} E(i).  \label{eq:LoudOwl}
	\end{equation}
	The desired estimate \eqref{eq:CozyMan} now follows by combining \eqref{eq:ValidZero} and \eqref{eq:LoudOwl}, where we use that $k \leq 2^{k-1}$ to bound the factor arising from the sum over $T$ in \eqref{eq:ValidZero}.
\end{proof}

The sum in \eqref{eq:CozyMan} can be interpreted as the normalization constant of a $\bbZ_{\geq 0}$-valued stochastic process for which going up by more than $\Psi(Z)$ in a single step is penalized by an exponential factor, and with steps downward or up by less than $\Psi(Z)$ being cost-free.
Our goal in the subsequent arguments is to establish an upper bound on this normalization constant.
\begin{remark}
	The aforementioned interpretation suggests that a significant contribution to the sum should come from the case where the path starts with a number of penalized steps upwards and subsequently only takes unpenalized steps, as this creates the most possible combinations given a fixed amount of penalization.
	An earlier version of this work utilized this structure to bound the normalization constant; the interested reader is referred to arXiv:2307.11632v2.
	A direct computation however turns out to be shorter, which we present below.
	We thank an anonymous referee for suggesting the more efficient argument.
\end{remark}

\subsubsection{Bounding the ``normalization constant''}
\label{sec: NormConst}

We start with preparatory estimates:

\begin{lemma}
	\label{lem: IntegralApproximations}

	Fix integers $i,P\geq 1$ and $q\geq 0$.
	Then,
	\begin{align}
		\sum_{\ell = 0}^{i-1} \ell^{q} \leq \frac{i^{q+1}}{q+1}
		\quad
		\textnormal{ and }
		\quad
		\sum_{\ell = 0}^\infty \ell^q \exp\Bigl(-\Bigl\lfloor \frac{\ell}{P}\Bigr\rfloor\Bigr) \leq \euler 2^q P^{q+1} q!
		.
		\label{eq:SpikyLeaf}
	\end{align}
	Here, it should be understood that $\ell^q = 1$ if $\ell = 0 = q$.
\end{lemma}

\begin{proof}
	The first estimate in \eqref{eq:SpikyLeaf} is an equality if $q=0$.
	Now let $q> 0$.
	Because $x\mapsto x^{q}$ is nondecreasing for $x\geq 0$ and $q>0$, we have $\sum_{\ell = 0}^{i-1} \ell^{q} \leq \int_{1}^ix^q \,\intd x$.
	The first estimate in \eqref{eq:SpikyLeaf} then follows from $\int_{1}^ix^q < \int_{0}^ix^q \, \intd x = i^{q+1}/(q+1)$.

	For the second estimate, we start by subdividing the sum based on $\lfloor \ell/P\rfloor$:
	\begin{equation}
		\sum_{\ell = 0}^\infty \ell^q \exp\Bigl(-\Bigl\lfloor \frac{\ell}{P}\Bigr\rfloor\Bigr)  = \sum_{j=0}^\infty \sum_{\ell = jP}^{(j+1)P- 1}    \ell^q \exp\bigl(- j\bigr)  \leq P^{q+1} \sum_{j=0}^\infty  (j+1)^q  \exp\bigl(- j\bigr).
	\end{equation}
	The inequality here used that $\ell^q \leq P^{q} (j+1)^q$ for all $\ell \in \{jP, \ldots, (j+1)P-1 \}$.
	Note that $x+1 \geq j+1$ and $\exp(-(x-1))\geq \exp(-j)$ for all $x \in [j,j+1)$.
	Consequently,
	\begin{equation}
		\sum_{j=0}^\infty  (j+1)^q  \exp\bigl(- j\bigr) \leq \int_0^\infty (x+1)^q \exp\bigl(-(x-1)\bigr) \, \intd x = \euler  \int_0^\infty (x+1)^q \exp\bigl(-x\bigr) \, \intd x.
	\end{equation}
	Conclude by noting that $(x+1)^q \leq 2^q x^q$ and $\int_0^\infty x^q \exp(-x)\, \intd x =\Gamma(q+1) =  q!$.
\end{proof}

\begin{lemma}
	\label{lem: ExpSum}

	For integers $P\geq 1$ and $k\geq 2$,
	\begin{equation}
		\sum_{i\in \{ 0\}\times \bbZ_{\geq 0}^{k-1}} \exp\Bigl(-\sum_{j=2}^k \bb1\{i_j > i_{j-1}\} \Bigl\lfloor \frac{i_j - i_{j-1}}{P}\Bigr\rfloor\Bigr)
		\leq
		16^{k-1} P^{k-1}
		.
		\label{eq:HauntedUnit}
	\end{equation}
\end{lemma}

\begin{proof}
	We consider a more general quantity.
	For an integer $q \geq 0$, define
	\begin{align}
		S(k,q)
		\de
		\sum_{ (i_1, \ldots, i_k) \in \{0\} \times \bbZ_{\geq 0}^{k-1} }
		i_k^q \exp\Bigl(- \sum_{j=2}^k \bb1\{i_j > i_{j-1} \} \Bigl\lfloor \frac{i_j - i_{j-1}}{P} \Bigr\rfloor  \Bigr).\label{eq:SpikyElf}
	\end{align}
	Here, it should be understood that $i_k^q = 1$ if $i_k = 0 = q$.
	Observe that $S(k,0)$ equals the left-hand side of \eqref{eq:HauntedUnit}.
	It hence suffices to prove that
	\begin{align}
		S(k,q)
		\leq
		16^{k-1 + q/2} P^{k-1 + q} q!
		.
		\label{eq:EasyZebra}
	\end{align}
	We proceed by induction on $k \geq 2$.

	First, note that $S(2,q) = \sum_{i_2=0}^\infty i_2^q \exp(-\lfloor i_2/P\rfloor)$.
	The second estimate in Lemma \ref{lem: IntegralApproximations} then yields that $S(2,q) \leq \euler 2^q P^{q+1} q!$.
	This shows that \eqref{eq:EasyZebra} holds for $k=2$.

	Next, suppose that \eqref{eq:EasyZebra} holds for a given $k\geq 2$ and all $q\geq 0$.
	We will now show that the bound then remains valid for $k+1$.
	For any fixed $i_{k} \geq 0$,
	\begin{align}
		{} & {}\sum_{i_{k+1}=0}^\infty  i_k^q \exp\Bigl(-  \bb1\{i_{k+1}> i_{k} \} \Bigl\lfloor \frac{i_{k+1} - i_{k}}{P} \Bigr\rfloor  \Bigr)
		=
		\sum_{i_{k+1}=0}^{i_k - 1} i_{k+1}^{q} + \sum_{\ell = 0}^\infty (\ell + i_k)^q \exp\Bigl(-\Bigl\lfloor\frac{\ell}{P} \Bigr\rfloor \Bigr) \nonumber \\
		   &
		=
		\negsp
		\sum_{i_{k+1}=0}^{i_k - 1}\negsp i_{k+1}^{q}
		+
		\sum_{j=0}^q \binom{q}{j}i_k^{q-j}  \sum_{\ell = 0}^\infty\ell^j \exp\Bigl(-\Bigl\lfloor\frac{\ell}{P} \Bigr\rfloor \Bigr)
		\leq
		\frac{i_{k}^{q+1}}{q+1} + \euler \sum_{j=0}^q  \binom{q}{j}i_k^{q-j} 2^j P^{j+1} j!
		.
		\label{eq:QuirkyFig}
	\end{align}
	Here, we used the binomial theorem for the second equality and Lemma \ref{lem: IntegralApproximations} for the inequality.
	Substituting \eqref{eq:QuirkyFig} in \eqref{eq:SpikyElf} and using the induction hypothesis \eqref{eq:EasyZebra} yields
	\begin{align}
		S(k+1,q)
		 &
		= \frac{1}{q+1}S(k,q+1) + \euler \sum_{j=0}^q \binom{q}{j}2^{j}P^{j+1} j!  S(k,q-j)
		\label{eq:HauntedDen}
		\\
		 &
		\leq
		\Bigr(
		16^{-1/2}
		+
		\euler \sum_{j=0}^q 2^{j}  16^{-j/2 -1}
		\Bigl)
		16^{k + q/2} P^{k + q} q!
		.
		\nonumber
	\end{align}
	Note that the term within brackets in \eqref{eq:HauntedDen} is no greater than $1$ to conclude that \eqref{eq:EasyZebra} also holds when $k$ is replaced by $k+1$.
	This concludes the proof.
\end{proof}

\begin{corollary}
	\label{cor: RestrictedSum}

	For all fixed $k\geq 3$, $I\in \{1,\ldots,n\}$, $J\in \{1,\ldots,k \}$,
	\begin{equation}
		\sum_{(\rho,\alpha,i) \in \cI_{k,n}: i_J = I}
		\negsp \negsp\negsp
		\Vert \Delci_i  \Vert_{\Linf}
		\leq
		64^{k-1}k! \Psi(Z)^{k-1}.
	\end{equation}
\end{corollary}

\begin{proof}
	This follows by combining Corollary \ref{cor: Del_FixIJ_exp} with Lemma \ref{lem: ExpSum}.
\end{proof}

\subsection{Derivative of tracial moments}\label{sec: DerivativesTracial}
Recall that our goal is to control the rate of change of tracial moments along the interpolation $\bS(t)$.
To control the individual terms in the expansion of  Proposition \ref{prop: MarkovNewVariables}, we must then understand the directional derivatives of $g(\bM) \de \tr[\bM^p]$.
The following lemma, which also appears in \cite[Lemma 6.3]{brailovskaya2022universality}, provides an explicit expression:
\begin{lemma} \label{lem: DerivativeMoment}
	For any positive integers $1 \leq k \leq p$ and any matrices $\bB_1,\ldots,\bB_k\in \bbC^{d\times d}$ the polynomial function from $\bbC^{d\times d}$ to $\bbC$ defined by $\bM \mapsto \tr[\bM^{p}]$ satisfies
	\begin{equation}
		\partial_{\bB_1} \cdots \partial_{\bB_k} \tr[\bM^{p}] = \sum_{\theta \in \cS_k} \sum_{\substack{r_1, \ldots,r_{k+1} \geq 0\\ r_1 + \cdots + r_{k+1} = p - k}} \tr[\bM^{r_{1}} \bB_{\theta(1)} \bM^{r_{2}} \bB_{\theta(2)} \cdots \bM^{r_{k}} \bB_{\theta(k)} \bM^{r_{k+1}}]. \nonumber
	\end{equation}
\end{lemma}
\begin{proof}
	This follows by entry-wise expansion for the trace and the product rule.
\end{proof}
We next establish a trace inequality that will be used to control the terms that arise on the right-hand side of Lemma \ref{lem: DerivativeMoment}.
For a random matrix $\bM$ and a scalar $1 \leq p \leq \infty$, we define
\begin{equation}
	\Vert \bM \Vert_p \de \begin{cases}
		\bbE\bigl[\tr[ \lvert\bM \rvert^{p}]\bigr]^{\frac{1}{p}} & \text{ if }p<\infty,   \\
		\bigl\Vert \Vert \bM \Vert_{\op} \bigr\Vert_{\Linf}      & \text{ if }p = \infty.
	\end{cases}\label{eq: Def_Lpnorm}
\end{equation}
It is known that $\Vert \cdot \Vert_p$ defines a norm on bounded random matrices \cite[(26)]{nelson1974notes}.
The following result is related to \cite[Proposition 5.1]{brailovskaya2022universality}, as is its proof.
\begin{lemma} \label{lem: TraceGeneralI}
	Fix a positive integer $k \geq 2$ and consider a finite set $\cI$.
	Let $\bX \de  (\bX_{i,j})_{i\in \cI, j\in \{1,\ldots,k \}}$ and $\bM \de (\bM_{j})_{j=1}^k$ be bounded self-adjoint random matrices and consider bounded real-valued random variables $\Delta \de (\Delta_{i})_{i\in \cI}$.
	Suppose that $\bM$ is independent of $(\bX, \Delta)$.
	Then, for any scalars $1 \leq p_1,\ldots,p_k \leq \infty$ with $\sum_{j=1}^k 1/p_j = 1$
	\begin{equation}
		\Bigl\lvert \bbE\Bigl[\sum_{i\in \cI} \Delta_i\tr[ \bX_{i, 1}\bM_1 \bX_{i, 2}\bM_2 \cdots \bX_{i,k} \bM_k]\Bigr] \Bigr\rvert\leq R_\cI(\bX )^{k-2} \varsigma_{\Delta}(\bX)^2 \prod_{j=1}^k \Vert \bM_j \Vert_{p_j} \label{eq: MatrixHolder}
	\end{equation}
	where
	\begin{equation}
		R_\cI(\bX ) = \max_{j=1,\ldots,k}  \max_{i\in \cI} \bigl\Vert\Vert \bX_{i,j} \Vert_{\op} \bigr\Vert_{\Linf}\ \text{ and } \ \varsigma_{\Delta}(\bX)^2 \de\negsp \max_{j=1,\ldots,k} \Bigl\Vert \bbE\Bigl[\sum_{i\in \cI} \lvert \Delta_i \rvert \bX_{i,j}^2 \Bigr] \Bigr\Vert_{\op}. \nonumber
	\end{equation}
\end{lemma}
\begin{proof}
	As in \cite[Proof of Proposition 5.1, Step 1]{brailovskaya2022universality} one can employ a convexity result, namely \cite[Lemma 5.2]{brailovskaya2022universality}, and the cyclic property of the trace to reduce to the case where $p_k = 1$.
	Note that it then necessarily holds that $p_j = \infty$ for any $j \neq k$.
	By rescaling both sides of \eqref{eq: MatrixHolder} it may further be assumed that $\Vert \bM_k \Vert_{1} = 1$ and $\Vert \bM_j \Vert_{\infty} = 1$ for any $j<k$.

	For the sake of notational simplicity, let $I$ be a random index which is uniformly distributed in $\cI$ and independent of $\bX, \Delta$ and $\bM$.
	Then, also using the tower property,
	\begin{align}
		\Bigl\lvert \bbE\Bigl[\sum_{i\in \cI} \Delta_i\tr [\bX_{i, 1}\bM_1 & \cdots \bM_{k-1}\bX_{i,k} \bM_k]\Bigr] \Bigr\rvert = \# \cI\, \bigl\lvert  \bbE\bigl[\tr[\Delta_I \bX_{I, 1}\bM_1   \cdots \bM_{k-1}\bX_{I,k} \bM_k]\bigr] \bigr\rvert \nonumber \\
		                                                                   & \leq \# \cI\, \bigl\lvert \bbE\bigl[\tr \bbE[ \Delta_I\bX_{I, 1}\bM_1  \cdots \bM_{k-1}\bX_{I,k} \mid \bM]\, \bM_k  \bigr]\bigr\rvert. \label{eq: Step_RandomIndex}
	\end{align}
	The norm \eqref{eq: Def_Lpnorm} admits a H\"older-type inequality that implies that $\lvert \bbE[\tr \bA \bB] \rvert \leq \Vert \bA \Vert_p \Vert \bB \Vert_q$ for any random matrices $\bA,\bB$ and $1/p + 1/q = 1$; see \eg \cite[Lemma 5.3]{brailovskaya2022universality} for a proof.
	Using this with $\bA = \bbE[ \Delta_I\bX_{I, 1}\bM_1  \cdots \bM_{k-1}\bX_{I,k} \mid \bM]$ and $\bB =  \bM_k$ and $p=\infty$ and $q = 1$,
	\begin{equation}
		\bbE\bigl[\tr \lvert \bbE[ \Delta_I\bX_{I, 1}\bM_1  \cdots \bM_{k-1}\bX_{I,k} \mid \bM]\, \bM_k \rvert \bigr]\label{eq: Step_OpNY}
		\leq \Vert \Vert \bbE[ \Delta_I\bX_{I, 1}\bM_1  \cdots \bM_{k-1}\bX_{I,k} \mid \bM]   \Vert_{\op}\Vert_{\Linf}
	\end{equation}
	where we used that $\bbE[\tr \lvert \bM_k \rvert] = \Vert \bM_k \Vert_{1} = 1$ by the preliminary reductions.
	\pagebreak[2]

	Denote $\bZ = \bM_1 \bX_{I,2}\bM_3 \bX_{I,3} \cdots \bM_{k-1}$ and note that, almost surely, $\bZ \bZ^* \preceq R_{\cI}(\bX)^{2k -4} \b1$ with respect to the positive semidefinite order where $\b1$ is the identity matrix.
	By definition of the operator norm, the Cauchy--Schwarz inequality for random vectors, and the assumption that $(\bX,\Delta)$ and $\bM$ are independent, it follows that almost surely
	\begin{align}
		\Vert \bbE[ \Delta_I\bX_{I, 1}\bM_1 & \cdots \bM_{k-1}\bX_{I,k} \mid \bM]   \Vert_{\op}                                                                                                                                                                                \\
		                                    & = \sup_{v,w \in S^{d-1}}   \bigl\lvert \bbE\bigl[ ( v^{*} \operatorname{sign}(\Delta_I)\lvert \Delta_I \rvert^{\frac{1}{2}} \bX_{I, 1} \bZ) (\lvert \Delta_I \rvert^{\frac{1}{2}}\bX_{I,k}w) \mid \bM\bigr] \bigr\rvert\nonumber \\
		                                    & \leq \sup_{v,w \in S^{d-1}}  \bbE[ v^{*}\lvert \Delta_I \rvert\bX_{I, 1}\bZ \bZ^* \bX_{I,1}v \mid \bM]^{\frac{1}{2}} \bbE[ w^{*}\lvert \Delta_I \rvert\bX_{I, k} \bX_{I,k}w \mid \bM]^{\frac{1}{2}} \nonumber                    \\
		                                    & \leq R_\cI^{k - 2}(\bX)   \sup_{v, w \in S^{d-1}}\bbE[ v^{*}\lvert \Delta_I \rvert\bX_{I, 1}^2v ]^{\frac{1}{2}} \bbE[ w^{*}\lvert \Delta_I \rvert\bX_{I, k}^2w ]^{\frac{1}{2}} \nonumber
	\end{align}
	with $S^{d-1} \subseteq \bbC^d$ the unit sphere.
	Here, for any $j \in \{1,\ldots,k \}$ and $v\in S^{d-1}$
	\begin{equation}
		\bbE[ v^{*}\lvert \Delta_I \rvert\bX_{I, j}^2v ] = (\# \cI)^{-1}  v^{*}\bbE\Bigl[ \sum_{i\in \cI} \lvert \Delta_i \rvert\bX_{i,j}^2 \Bigr] v  \leq   (\# \cI)^{-1} \varsigma_{\Delta}(\bX)^2.  \label{eq: Step_vvsigma}
	\end{equation}
	Combine \eqref{eq: Step_RandomIndex}--\eqref{eq: Step_vvsigma} to conclude the proof.
\end{proof}

\subsection{Proof of \texorpdfstring{Theorem \ref{thm: MainTracial}}{Theorem}}\label{sec: UniversalityEvenOrderMarkov}
We finally combine the preceding ingredients with a direct calculation to control the rate of change of $\bbE[\tr \bS(t)^{2p}]$ along the interpolation from \eqref{eq: Def_S(t)} and use this to prove the desired universality principle for the tracial moments of even order.
\begin{lemma}\label{lem: Lem_S(t)DerivativeBound_Psi}
	For any integer $p\geq 2$ and any $t\in [0,1]$
	\begin{align}
		\Bigl\lvert \dd{t} & \bbE[\tr \bS(t)^{2p}] \Bigr\rvert
		\\
		                   & \leq
		(400p)^3 R(\bX)   \Psi(Z)^2 \varsigma(\bX)^2  \max\bigl\{\bbE[\tr \bS(t)^{2p}]^{1 - \frac{3}{2p}},  (400 p  R(\bX)\Psi(Z) )^{2p-3}  \bigr\}. \nonumber
	\end{align}
\end{lemma}
\begin{proof}
	The combination of Proposition \ref{prop: MarkovNewVariables} and Lemma \ref{lem: DerivativeMoment} yields that
	\begin{align}
		\dd{t} \bbE[ & \tr \bS(t)^{2p}]= \frac{1}{2}\sum_{k=3}^{2p} \frac{t^{\frac{k}{2}-1}}{(k-1)!}\sum_{\theta \in \cS_k}\sum_{\substack{r_1, \ldots,r_{k+1} \geq 0                              \\ r_1 + \cdots + r_{k+1} = 2p - k}}  \label{eq: Step_trS(t)2p_Expansion}  \\
		             & \times  \bbE\Bigl[\sum_{(\rho,i)\in \cI_{k,n}}\tr[\Delci_i\bS(t)^{r_{1}} \bXci_{i, \theta(1)}  \cdots \bS(t)^{r_{k}}\bXci_{i,\theta(k)} \bS(t)^{r_{k+1}}] \Bigr]. \nonumber
	\end{align}
	We here used that $\bM \to \tr \bM^{2p}$ is a polynomial of degree $2p$ to neglect all terms in Proposition \ref{prop: MarkovNewVariables} with $k> 2p$.
	Any $\theta \in \cS_k$ defines a bijection from $\cI_{k,n}$ into itself as $(\rho,i) \mapsto (\rho, i_\theta)$ where we recall that $(i_\theta)_j = i_{\theta(j)}$.
	Hence, the sum over $\theta$ in \eqref{eq: Step_trS(t)2p_Expansion} can be eliminated in exchange for a factor $\# \cS_k = k!$:
	\begin{align}
		\dd{t} \bbE[\tr \bS(t)^{2p}]
		 & \leq \frac{1}{2}\sum_{k=3}^{2p} \sum_{\substack{r_1, \ldots,r_{k+1} \geq 0                                                                                                \\ r_1 + \cdots + r_{k+1} = 2p - k}} t^{\frac{k}{2}-1} k \times   \label{eq: Step_ApplicationDerivativeExpansion}   \\
		 & \Bigl\lvert\bbE\Bigl[\sum_{(\rho,i)\in \cI_{k,n}}\tr [\Delci_i\bS(t)^{r_{1}} \bXci_{i, 1}  \cdots \bS(t)^{r_{k}}\bXci_{i,k} \bS(t)^{r_{k+1}}]\Bigr]\Bigr\rvert. \nonumber
	\end{align}

	We next apply Lemma \ref{lem: TraceGeneralI} with $\bM_j = \bS(t)^{r_{j+1}}$ and $p_j = (2p-k)/r_{j+1}$ for $j< k$ and $\bM_k = \bS(t)^{r_{k+1} + r_1}$ and $p_k = (2p-k)/(r_{k+1} + r_1)$ for $j=k$.
	Note that the independence condition in Lemma \ref{lem: TraceGeneralI} is then satisfied because $\bS(t)$ only depends on $\bG$ and $Z_1,\ldots,Z_n$ which are independent of $\bXci_{i_j} = \bF_{i_j}(\Zci_{i,j})$ and $\Delci_i$ due item \nref{item:Independence}{(2)} in Proposition \ref{prop: MarkovNewVariables}.
	Thus,
	\begin{align}
		\Bigl\lvert \bbE\Bigl[\sum_{(\rho,i)\in  \cI_{k,n}}\tr [\Delci_i\bS(t)^{r_{1}} & \bXci_{i,1}  \cdots \bS(t)^{r_{k}}\bXci_{i,k} \bS(t)^{r_{k+1}}]\Bigr] \Bigr\rvert \label{eq: Step_AppliedTracePsi} \\
		                                                                               & \leq R_{\cI_{k,n}}(\bXci)^{k-2} \varsigma_{\Delta}(\bXci)^2 \bbE[\tr \lvert \bS(t) \rvert^{2p-k}] \nonumber
	\end{align}
	where
	\begin{align}
		R_{\cI_{k,n}}(\bXci)        & \de \max_{j = 1,\ldots,k}\max_{(\rho,i) \in \cI_{k,n}} \bigl\Vert \Vert \bXci_{i,j} \Vert_{\op} \bigr\Vert_{\Linf},\label{eq:LoudEgg}                                  \\
		\varsigma_{\Delta}(\bXci)^2 & \de \max_{j=1,\ldots,k} \Bigl\Vert \bbE\Bigl[\sum_{(\rho,i)\in  \cI_{k,n}} \lvert \Delci_i \rvert\, (\bXci_{i,j})^2 \Bigr] \Bigr\Vert_{\op}\label{eq: Def_SigmaDelta}.
	\end{align}
	If follows from item \nref{item:Marginal}{(1)} in Proposition \ref{prop: MarkovNewVariables} that $\bXci_{i,j}$ has the same marginal distribution as $\bX_{i_j}$.
	Hence,
	$
		R_{\cI_{k,n}}(\bXci) = R(\bX)
	$
	and $\bbE[(\bXci_{i,j})^2] = \bbE[\bX_{i_j}^2]$ with $R(\bX)$ as in \eqref{eq: Def_R(X)}.
	Thus, the following inequality holds with respect to the positive semidefinite order for any fixed $j\leq k$:
	\begin{align}
		\bbE\Bigl[\sum_{(\rho,i)\in  \cI_{k,n}} \lvert \Delci_i \rvert\, (\bXci_{i,j})^2 \Bigr]
		 & \preceq \sum_{(\rho,i)\in  \cI_{k,n}}  \Vert \Delci_i \Vert_{\Linf} \bbE\Bigl[(\bXci_{i,j})^2\Bigr] \label{eq: Step_PosDefO} \\
		 & = \sum_{I = 1}^n \Bigl(\sum_{(\rho,i)\in  \cI_{k,n}: i_j = I}
		\Vert  \Delci_i \Vert_{\Linf}\Bigr)\bbE\Bigl[\bX_{I}^2 \Bigr]. \nonumber
	\end{align}
	Here, we have $\sum_{(\rho,i)\in  \cI_{k,n}: i_j = I}
		\Vert  \Delci_i \Vert_{\Linf} \allowbreak \leq  64^{k-1} k! \Psi(Z)^{k-1}$ by Corollary \ref{cor: RestrictedSum}.
	Recall the definitions of $\varsigma(\bX)^2$ and $\varsigma_{\Delta}(\bXci)^2$ from \eqref{eq: Def_sigma} and \eqref{eq: Def_SigmaDelta}, respectively.
	We conclude that
	\begin{align}
		\varsigma_{\Delta}(\bXci)^2 \leq  64^{k-1} k! \Psi(Z)^{k-1}  \varsigma(\bX)^2. \label{eq: SigmaDelta}
	\end{align}

	We next combine \eqref{eq: Step_ApplicationDerivativeExpansion}--\eqref{eq: SigmaDelta}.
	Note that the number of $(k+1)$-tuples $(r_1, \ldots,r_{k+1})$ satisfying $r_j \geq 0$ and $\sum_{j=1}^{k+1} r_j = 2p-k$ is equal to $\binom{2p}{k} \leq (2p)^k / k!$.
	Hence,
	\begin{align}
		\Bigl\lvert \dd{t}\bbE[\tr \bS(t)^{2p}] \Bigr\rvert & \leq  \frac{1}{2}\sum_{k=3}^{2p} t^{\frac{k}{2}-1}k 64^{k-1} (2p)^{k} R(\bX)^{k-2} \Psi(Z)^{k-1}\varsigma(\bX)^2\bbE[\tr\lvert \bS(t) \rvert^{2p-k}]\nonumber \\
		                                                    & \leq \frac{1}{2}\sum_{k=3}^{2p}  (200 p)^{k}R(\bX)^{k-2} \Psi(Z)^{k-1}\varsigma(\bX)^2 \bbE[\tr\bS(t)^{2p}]^{1-\frac{k}{2p}} \label{eq: ddtSc}
	\end{align}
	where the second inequality used that $\bbE[\tr\lvert \bS(t) \rvert^{2p-k}] \leq \bbE[\tr\bS(t)^{2p}]^{1-\frac{k}{2p}}$ by Jensen's inequality as well as the fact that $t^{\frac{k}{2}-1} \leq 1$ and $k(64)^{k-1} \leq 100^k$ for all $k \geq 3$.
	Here, we can further simplify by using H\"older's inequality with the fact that $\sum_{k=3}^{2p} (1/2)^k \leq \sum_{k=0}^{2p} (1/2)^k =2$,
	\begin{align}
		\frac{1}{2}\sum_{k=3}^{2p}  (200 p)^{k}{} & {}R(\bX)^{k-2} \Psi(Z)^{k-1}\varsigma(\bX)^2 \bbE[\tr\bS(t)^{2p}]^{1-\frac{k}{2p}} \label{eq: Step_SimplifyingS(t)Psi}                     \\
		                                          & = \frac{1}{2}\sum_{k=3}^{2p} 2^{-k}  (400p)^k   R(\bX)^{k-2}\Psi(Z)^{k-1} \varsigma(\bX)^2 \bbE[\tr\bS(t)^{2p}]^{1-\frac{k}{2p}} \nonumber \\
		                                          & \leq \max_{3\leq k \leq 2p}  (400p)^k R(\bX)^{k-2} \Psi(Z)^{k-1}\varsigma(\bX)^2 \bbE[\tr\bS(t)^{2p}]^{1-\frac{k}{2p}}.\nonumber
	\end{align}
	Finally, note that the function $k\mapsto x^k$ is convex for any fixed $x\geq 0$.
	It follows that the maximum on the right-hand side of \eqref{eq: Step_SimplifyingS(t)Psi} is achieved at $k = 3$ or at $k= 2p$.
	Combine \eqref{eq: ddtSc} and \eqref{eq: Step_SimplifyingS(t)Psi} to complete the proof.
\end{proof}
Solving the differential inequality in Lemma \ref{lem: Lem_S(t)DerivativeBound_Psi} now yields the desired result:
\begin{proof}[Proof of {Theorem \ref{thm: MainTracial}}]
	If $p = 1$, then the result is true because $\bbE[\bG^2] = \bbE[\bS^2]$ by definition of a Gaussian model.
	Now assume that $p\geq 2$.

	Every differentiable function $f:[0,1]\to \bbR_{\geq 0}$ with $\lvert \dd{t}f(t) \rvert \leq C\max\{f(t)^{1-\alpha}, K^{1-\alpha}\}$ for constants $C,K \geq 0$ and $\alpha\in [0,1]$ satisfies that $\lvert f(1)^{\alpha} - f(0)^{\alpha} \rvert \leq C\alpha  + K^\alpha$; see \cite[Lemma 6.6]{brailovskaya2022universality} for a proof.
	Applying this with $\alpha = 3/(2p)$ to the conclusion of Lemma \ref{lem: Lem_S(t)DerivativeBound_Psi}
	\begin{equation}
		\lvert \bbE[\tr \bS^{2p}]^{\frac{3}{2p}} - \bbE[\tr \bG^{2p}]^{\frac{3}{2p}} \rvert \leq \frac{3}{2p} (400p)^3 R(\bX) \Psi(Z)^2 \varsigma(\bX)^2 + (400 p R(\bX) \Psi(Z))^3.\label{eq:UnripeTea}
	\end{equation}
	Using that $v\mapsto (w+v)^{1/3} -v^{1/3}$ is a nonincreasing function in $v\geq 0$ for any fixed $w\geq 0$, it can be deduced that $\lvert x^{1/3} - y^{1/3} \rvert \leq \lvert x-y \rvert^{1/3}$ and $(x+y)^{1/3} \leq x^{1/3} + y^{1/3}$ for any $x,y\geq 0$.
	The desired result hence follows with absolute constant $c= 500$ by raising both sides of \eqref{eq:UnripeTea} to the power $1/3$ and using that $(3/2)^{1/3}400  \leq 500$.
\end{proof}

\section{Proof of \texorpdfstring{Corollaries \ref{cor: MarkovMarkovBound} and \ref{cor: Khintchine}}{Corollaries} as well as Equation \texorpdfstring{\eqref{eq: MarkovMarkovBound}}{Equation}}\label{sec: BoundsOpNormMarkov}

We now demonstrate how the universality of tracial moments proved in Theorem \ref{thm: MainTracial} can be combined with Gaussian theory to give bounds on the operator norm.
We use yet another variance proxy:
\begin{equation}
	\sigma_*(\bS)^2 \de \sup_{\Vert v \Vert = \Vert w \Vert 1} \bbE\bigl[\lvert \langle v , (\bS - \bbE[\bS])w \rangle \rvert^2 \bigr].
\end{equation}
This quantity is used to state the following general-purpose estimate in its strongest form, but it will suffice in our application that $\sigma_*(\bS) \leq v(\bS)$ and $\sigma_*(\bS) \leq \sigma(\bS)$ \cite[Section 2.1]{bandeira2021matrix}.
\begin{lemma}\label{lem: GeneralLpGaussian}
	There exists an absolute constant $C>0$ such that for every integer $p\geq 1$,
	\begin{equation}
		d^{-\frac{1}{2p}}\bbE[\Vert \bG \Vert] - C\cE_*(p) \leq \bbE[\Vert \bS \Vert^{2p}]^{1/2p} \leq d^{\frac{1}{2p}}\bbE[\Vert \bG \Vert] + Cd^{\frac{1}{2p}}\cE_*(p)\label{eq:ColdForce}
	\end{equation}
	where $\cE_*(p) \de  R(\bX)^{1/3} \Psi(Z)^{2/3}\varsigma(\bX)^{2/3}  p^{2/3} + R(\bX) \Psi(Z) p + \sigma_*(\bS)\sqrt{p}$.
\end{lemma}
\begin{proof}
	Using that $d^{-1}\Vert \bM \Vert^{2p} \leq \tr[\bM^{2p}] \leq \Vert \bM \Vert^{2p}$ for any $\bM \in \Csa^{d\times d}$, it follows from Theorem \ref{thm: MainTracial} that there exists an absolute constant $c>0$ such that
	\begin{equation}
		d^{-\frac{1}{2p}}\bE[\Vert \bG \Vert^{2p}]^{\frac{1}{2p}} - cE(p) \leq \bbE[\Vert \bS^{2p} \Vert]^{\frac{1}{2p}} \leq d^{\frac{1}{2p}}\bbE[  \Vert \bG \Vert^{2p} ]^{\frac{1}{2p}} + cd^{\frac{1}{2p}} E(p) \label{eq:LazyBee}
	\end{equation}
	where $E(p) \de R(\bX)^{1/3} \Psi(Z)^{2/3}\varsigma(\bX)^{2/3}  p^{2/3} + R(\bX) \Psi(Z) p$.

	Viewing the operator norm $\Vert \bG \Vert = \sup_{\Vert v \Vert = \Vert w \Vert = 1} \operatorname{Re}(\langle v,\bG w \rangle)$ as a Gaussian process, it follows from \cite[Theorem 5.8]{boucheron2013concentration} that
	$
		\bbP(\Vert \bG \Vert - \bbE \Vert \bG \Vert \geq x) \leq 2\exp(-x^2/2\sigma_*^2(\bG))
	$.
	Hence, using that sub-Gaussianity is equivalent to moment bounds \cite[Theorem 2.1]{boucheron2013concentration}, we have that
	\begin{equation}
		\bbE\bigl[ \bigl\lvert \Vert \bG \Vert   - \bbE[\bG]\bigr\rvert^{2p}\bigr]^{\frac{1}{2p}} \leq 2\sigma_*(\bG) \sqrt{p}.  \label{eq:LuckyToy}
	\end{equation}
	Note that $\sigma_*(\bG) = \sigma_*(\bS)$ since this quantity only depends on the covariance structure.
	Combining \eqref{eq:LazyBee} and \eqref{eq:LuckyToy} using the triangle inequality for the $L^p$-norm hence yields \eqref{eq:ColdForce}.
\end{proof}

\begin{proof}[Proof of \texorpdfstring{Corollary \ref{cor: MarkovMarkovBound}}{Corollary}]
	It is shown in \cite[Theorem 2.3]{bandeira2021matrix} that $\lvert \bbE[\Vert \bG \Vert] - \Vert \bG_{\free} \Vert\rvert \leq C v(\bS)^{1/2} \sigma(\bS)^{1/2} (\log d)^{3/4}$.
	The bounds in Corollary \ref{cor: MarkovMarkovBound} are then immediate from Lemma \ref{lem: GeneralLpGaussian} since $\Vert \bG_{\free} \Vert  = \Vert \bS_{\free} \Vert$ and $\sigma_*(\bS) \leq \min\{v(\bS), \sigma(\bS) \} \leq  v(\bS)^{1/2} \sigma(\bS)^{1/2}$.
\end{proof}
\begin{proof}[Proof of \texorpdfstring{Corollary \ref{cor: Khintchine}}{Corollary}]
	The matrix Khintchine inequality of Lust--Piquard \cite{lust1986inegalites}, \cite[Corollary 2.4]{tropp2018second} implies that $\bbE[\Vert \bG \Vert] \leq C\sqrt{\ln(d+1)}\sigma(\bS)$ for some absolute constant $C>0$.
	Substituting this in Lemma \ref{lem: GeneralLpGaussian} with $p = \lceil \ln(d+1)\rceil$ and using that $\sigma_*(\bS) \leq \sigma(\bS)$ yields the bound in Corollary \ref{cor: Khintchine} since $\bbE[\Vert \bS \Vert] \leq \bbE[\Vert \bS \Vert^{2p}]^{1/2p}$ by Jensen's inequality.
\end{proof}
Finally, the following result implies the tail bound that was claimed in \eqref{eq: MarkovMarkovBound}.
\begin{proposition}\label{prop: MarkovMarkovTail}
	There exists an absolute constant $c>0$ such that for every $0 < \delta \leq 1$ and $x>0$ it holds with $\cE$ as in Corollary \ref{cor: MarkovMarkovBound} that
	\begin{equation}
		\bbP\bigl(\Vert \bS \Vert \geq (1+\delta) \Vert \bS_{\free} \Vert   + c\cE(x)\bigr) \leq (d+1)(1+\delta)^{-x}.\label{eq:JustCup}
	\end{equation}
\end{proposition}
\begin{proof}
	Using the upper bound in Corollary \ref{cor: MarkovMarkovBound} and Markov's inequality, there exists an absolute constant $C>0$ such that for every $y>0$ and every integer $p\geq 1$,
	\begin{equation}
		\bbP(\Vert \bS \Vert \geq y (\Vert \bS_{\free} \Vert + C \cE(p))  ) \leq  \bbP\bigl( \Vert \bS \Vert \geq yd^{-1/2p}\bbE[\Vert \bS \Vert^{2p}]^{1/2p}  \bigr) \leq d y^{-2p}.\label{eq:OilyCamel}
	\end{equation}
	Let $y \de 1+\delta$ and $p \de \lceil x/2 \rceil$.
	We may assume without loss of generality that $x \geq \log_{1+\delta}(d+1)$, since \eqref{eq:JustCup} is vacuous otherwise.
	Using this in \eqref{eq:OilyCamel} as well as the fact that $x/2 \leq p \leq x$,
	\begin{equation}
		\bbP(\Vert \bS \Vert \geq  (1+\delta)\Vert \bS_{\free} \Vert + (1+\delta) C \cE(x) ) \leq (d+1) (1+\delta)^{-x}.
	\end{equation}
	This yields \eqref{eq:JustCup} with $c \de 2C$ since $1+\delta \leq 2$ by the assumption that $\delta \leq 1$.
\end{proof}

\section*{Acknowledgments}
This work is part of the project Clustering and Spectral Concentration in Markov Chains with project number OCENW.KLEIN.324 of the research programme Open Competition Domain Science -- M which is partly financed by the Dutch Research Council (NWO).

We thank Ramon van Handel for discussions at the 50th Probability Summer School of Saint--Flour which led to a simplified proof for Lemma \ref{lem: TraceGeneralI}.
This manuscript further benefitted significantly from feedback from an anonymous referee, who we thank sincerely for an exceptionally helpful report.

\section*{Declarations}
The authors have no relevant financial or non-financial interests to disclose.

\pagebreak[4]
\appendix

\section{Properties of the \texorpdfstring{$\psi$}{psi}-dependence coefficient}\label{apx: Psi}
In this section, we prove some properties of the $\psi$-dependence coefficient that may be helpful in the estimation of the parameters when one needs to apply our results.
First, we prove a bound on $\Psi(Z)$ in terms of the total variation mixing time of the Markov chain in Section \ref{sec: Psi_TVmix}.
Second, we prove bounds on $\sigma(\bS)$ and $v(\bS)$ in Section \ref{sec: BoundSigmaV}.

\subsection{Bound on \texorpdfstring{$\Psi(Z)$}{Psi(Z)} in terms of total variation mixing time}\label{sec: Psi_TVmix}
Another common way to quantify the decay of dependence in a Markov chain is the \emph{total variation mixing time} \cite[Section 4.5]{levin2017markov}.
As was announced in Remark \ref{rem: sigma_v_psi}, one can bound $\Psi(Z)$ in terms of the mixing time whenever the state space is finite.
We prove this claim in Proposition \ref{prop: Psipi}.

\begin{lemma}\label{lem: Psi(Z)_finitestatespace}
	Let $Z = (Z_i)_{i=1}^n$ be an ergodic stationary Markov chain on a finite state space $\cZ$.
	Denote $\bP\in [0,1]^{\cZ \times \cZ}$ and $\pi \in [0,1]^{\cZ}$ for the transition matrix and stationary distribution of $Z$.
	Then,
	\begin{equation}
		\Psi(Z) = \min\Bigl\{ n,\, \min\Bigl\{t \geq 1: \max_{i,j \in \cZ}\Bigl\lvert\, \frac{(\bP^t)_{i,j} - \pi_j}{\pi_j} \, \Bigr\rvert \leq \frac{1}{4} \Bigr\} \Bigr\}. \label{eq: PsiZ}
	\end{equation}
\end{lemma}
\begin{proof}
	We claim that if $X$ and $Y$ are random variables with values supported on a finite set $\cZ$, then the suprema in \eqref{eq: Def_PsiDependence} are realized by singleton sets:
	\begin{equation}
		\psi(X , Y) =  \max_{x \in \cZ, y \in \cZ }\ \Bigl\lvert  \frac{\bbP(X = x, Y = y) -  \bbP(X = x) \bbP(Y = y) }{ \bbP(X = x) \bbP(Y=y)} \Bigr\rvert. \label{eq:MadSun}
	\end{equation}
	To see this, note that for any absolutely continuous measures $\mu\ll \nu$,
	\begin{equation}
		\sup_{E: \nu(E)>0} \Bigl\lvert  \frac{\mu(E)}{\nu(E)} - 1\Bigr\rvert = \sup_{E: \nu(E)>0}  \frac{1}{\nu(E)} \Bigl\lvert \int  \bb1_E \Bigl(\frac{\intd \mu}{\intd \nu} - 1\Bigr) \intd \nu \Bigr\rvert   \leq \Bigl\Vert \frac{\intd \mu}{\intd \nu} - 1 \Bigr\Vert_{\Linf}.
	\end{equation}
	Applying this with $\mu \de \bbP_{X,Y}$ and $\nu \de \bbP_X\otimes \bbP_Y$ shows that the left-hand side of \eqref{eq:MadSun} no greater than the right-hand side.
	The other inequality follows by using singleton sets in \eqref{eq: Def_PsiDependence}.
	Using \eqref{eq:MadSun} in the definition \eqref{eq: Def_PsiDependence} of $\Psi(Z)$ now yields \eqref{eq: PsiZ}.
\end{proof}

The \emph{total variation distance} between two probability measures $\mu, \nu$ on the same space $\cZ$ is
\begin{equation}
	d_{\TV}(\mu,\nu) \de \max_{E \subseteq \cZ}\, \lvert \mu(A) - \nu(E)\rvert.
	\label{eq: defTV}
\end{equation}
For any $\varepsilon>0$ the \emph{$\varepsilon$-mixing time} of an ergodic Markov chain $Z$ on a finite state space is defined as $\tmix(\varepsilon) \de \min\{t\geq 1: d(t) \leq \varepsilon\}$ where
$
	d(t) \de \sup_{i\in \cZ}d_{\mathrm{TV}}(\mathbb{P}(Z_t = \cdot \mid Z_0 = i),\pi).
$
We refer to $\tmix\de \tmix(1/4)$ as the \emph{total variation mixing time} of $Z$.
\begin{proposition}\label{prop: Psipi}
	Let $Z$ and $\pi \in [0,1]^{\cZ}$ be as in Lemma \ref{lem: Psi(Z)_finitestatespace} and denote $\pi_{\min} \de \min \{ \pi_x: x \in \cZ \}$.
	Then,
	$
		\Psi(Z) \leq (\log_{2}(1/\pi_{\min}) + 3) \tmix
	$
\end{proposition}
\begin{proof}
	By taking $E = \{j\}$ in \eqref{eq: defTV} it follows that for every $t\geq 1$,
	$
		\max_{i,j \in \cZ} \lvert \bP^t_{i,j} - \pi_j  \rvert \leq d(t).
	$
	Denote $\ell \de \lceil \log_2(1/\pi_{\min}) + 2 \rceil$ and note that $\ell \geq \log_2(1/\pi_{\min}) + 2$.
	It consequently follows from \cite[(4.35)]{levin2017markov} that $
		d(\ell \tmix) \leq 2^{-\ell} \leq 4^{-1}\pi_{\min}.$
	Hence,
	\begin{equation}
		\max_{i,j \in \cZ} \lvert (\bP^{\ell \tmix}_{i,j} - \pi_j)/ \pi_j \rvert \leq 1/4. \label{eq: Step_Mixtopsi3}
	\end{equation}
	The desired result now follows from Lemma \ref{lem: Psi(Z)_finitestatespace} and the fact that $\ell \leq \log_2(1/\pi_{\min})\allowbreak  + 3$.
\end{proof}

\subsection{Bounds on \texorpdfstring{$\sigma(\bS)$}{sigma(S)} and \texorpdfstring{$v(\bS)$}{v(S)}}\label{sec: BoundSigmaV}
We now prove Remark \ref{rem: sigma_v_psi} in Lemmas \ref{lem: sigmaBound} and \ref{lem: vBound}.

\begin{lemma}\label{lem: sigmaBound}
	It holds that $\sigma(\bS)^2 \leq 3\Psi(Z) \varsigma(\bX)^2$.
\end{lemma}
\begin{proof}
	By expanding the summation in the definition of $\bS$ and regrouping terms,
	\begin{equation}
		\bbE\bigl[(\bS - \bbE[\bS])^2\bigr] = \frac{1}{2}\sum_{i=1}^n\sum_{j=1}^n   \bbE\bigl[\bX_i \bX_j + \bX_j \bX_i\bigr] . \label{eq:OccultHog}
	\end{equation}
	For any $i,j$ with $\psi(Z_i, Z_j) < \infty$ using the definition that $\bX_i = \bF_{i}(Z_i)$,
	\begin{align}
		 & \bbE\bigl[\bX_i \bX_j + \bX_j \bX_i\bigr] = \int_{\cZ_i\times \cZ_j}\bigl( \bF_i(z_i)\bF_j(z_j)  + \bF_j(z_j)\bF_i(z_i) \bigr)\, \intd\bbP_{Z_i,Z_j}(z_i,z_j)\label{eq:FancyWisp}                                       \\
		 & =\int_{\cZ_i\times \cZ_j} \bigl( \bF_i(z_i)\bF_j(z_j)  + \bF_j(z_j)\bF_i(z_i) \bigr)\frac{\intd \bbP_{Z_i,Z_j}}{\intd \bbP_{Z_i}\otimes \bbP_{Z_j}}(z_i,z_j) \, \intd(\bbP_{Z_i}\otimes \bbP_{Z_j})(z_i,z_j). \nonumber
	\end{align}
	Let $\tilde{Z}_i$ and $\tilde{Z_j}$ be independent random variables with the same marginal distribution as $Z_i$ and $Z_j$, respectively.
	Then, with $\Delta_{i,j} \de \frac{\intd \bbP_{Z_i,Z_j}}{\intd \bbP_{Z_i}\otimes \bbP_{Z_j}}(\tilde{Z}_i, \tilde{Z}_j)  - 1$ and $\tilde{\bX}_i = \bF_{i}(\tilde{Z}_i)$, it follows from \eqref{eq:FancyWisp} and the assumption that $\bX_{i}$ is centered that
	\begin{equation}
		\bbE\bigl[\bX_i \bX_j + \bX_j \bX_i\bigr] = \bbE[(1 + \Delta_{i,j}) (\tilde{\bX}_i \tilde{\bX}_j + \tilde{\bX}_j \tilde{\bX}_i)]  =  \bbE[\Delta_{i,j} (\tilde{\bX}_i \tilde{\bX}_j + \tilde{\bX}_j \tilde{\bX}_i)].\label{eq:JumpyQuiz}
	\end{equation}
	In general, for any $\bA, \bB\in \Csa^{d\times d}$ and scalar $\delta \in \bbR$, we have $\delta(\bA \bB +\bB \bA) \preceq \lvert \delta \rvert (\bA^2 + \bB^2)$ with respect to the positive semidefinite order.\footnote{This follows by using that $\delta(\bA-\bB)^2\succeq 0$ if $\delta\geq 0 $ and by using that $\delta(\bA+\bB)^2\preceq 0 $ if $\delta \leq 0$.}
	Hence, using Proposition \ref{prop: PsiRadonNikodym} and \eqref{eq:JumpyQuiz} if and only if $\psi(Z_i, Z_j) <1$, and applying the inequality directly otherwise,
	\begin{equation}
		\bbE\bigl[\bX_i \bX_j + \bX_j \bX_i\bigr] \preceq  \min\{1,\psi(Z_i, Z_j) \}\bigl(\bbE[\bX_i^2] + \bbE[\bX_j^2]\bigr) \label{eq:JustFace}
	\end{equation}
	for all $i,j \in \{1,\ldots,n \}$.
	Substitution in \eqref{eq:OccultHog} and reorganizing the terms again yields that
	\begin{equation}
		\bbE\bigl[(\bS - \bbE[\bS])^2\bigr] \preceq \sum_{i=1}^n \sum_{j=1}^n  \min\{1,\psi(Z_i, Z_j) \} \bbE[\bX_i^2]. \label{eq:NormalZap}
	\end{equation}
	Finally, using that $ \min\{1,\psi(Z_i, Z_j) \}\leq (1/4)^{\lfloor \lvert i-j \rvert/\Psi(Z)\rfloor}$ by Proposition \ref{prop: ExponentialPsi}, we here have that $\sum_{j=1}^n\min\{1,\psi(Z_i, Z_j) \} \leq 2\Psi(Z)\sum_{v = 0}^\infty (1/4)^{v} \leq 3\Psi(Z)$.
	The claimed estimate hence follows from the fact that the operator norm respects the positive semidefinite order when restricted to the set of positive semidefinite matrices.
\end{proof}
\begin{lemma}\label{lem: vBound}
	It holds that $v(\bS)^2 \leq 3\Psi(Z) \Vert \sum_{i=1}^n \Cov(\bX_i) \Vert$.
\end{lemma}
\begin{proof}
	For any random vector $V$, it holds that $\Vert \Cov(V) \Vert = \sup_{\Vert W \Vert \leq 1}\bbE[\lvert \langle V, W \rangle \rvert^2]$.
	In particular, $\Vert \operatorname{Cov}(\bS) \Vert = \sup_{\Tr \lvert \bM \rvert^2 \leq 1} \bbE[\lvert \Tr[\bS \bM] \rvert^2]$ with $\Tr[\bM]  = \sum_{i} \bM_{i,i}$ the unnormalized trace.
	Here, proceeding similarly to \eqref{eq:OccultHog}--\eqref{eq:NormalZap},
	\begin{equation}
		\bbE\bigl[\lvert \Tr[\bS \bM] \rvert^2\bigr] \leq  \sum_{i=1}^n \sum_{j=1}^n \min\{1,\psi(Z_i,Z_j) \} \bbE\bigl[ \lvert \Tr[\bX_i \bM ] \rvert^2 \bigr]. \label{eq:YellowHog}
	\end{equation}
	Thus, taking the supremum over $\bM$ and bounding $\sum_{j=1}^n\min\{1,\psi(Z_i, Z_j) \} \leq 3\Psi(Z)$, we have $\Vert \Cov(\bS) \Vert \leq 3\Psi(Z) \Vert \sum_{i} \Cov(\bX_i) \Vert$.
	This yields the desired result.
\end{proof}

\section{Proof of Lemma \texorpdfstring{\ref{LEM: SAMSON}}{.}}\label{apx: Samson}

The following is a special case of a result by Samson \cite{samson2000concentration}.
\begin{lemma}\label{lem: MatrixSamson}
	Consider scalar random variables of the form $Y_i = f_i(Z_i)$ for functions $f_i:\cZ_i \to [0,1]$, and consider deterministic matrices $\bB_1,\ldots,\bB_n \in \bbR^{d\times d}$.
	Then, there exists an absolute constant $C>0$ such that for every $x\geq 0$,
	\begin{equation}
		\bbP\Bigl( \Bigl\lvert   \Bigl\Vert \sum_{i=1}^n Y_i \bB_i \Bigr\Vert - \bbE\Bigl\Vert \sum_{i=1}^n Y_i \bB_i \Bigr\Vert \Bigr\rvert \geq x \Bigr) \leq \exp\Bigl(-C \frac{x^2}{\varsigma_*(\bB)^2\Psi(Z) }\Bigr)
	\end{equation}
	where $\varsigma_*^2(\bB) \de \sup_{\Vert v \Vert = \Vert w \Vert = 1} \sum_{i=1}^n \langle v, \bB_i w \rangle^2$.
\end{lemma}
\begin{proof}
	We use \cite[(2.23)]{samson2000concentration}, which provides a concentration--of--measure principle for random sums of deterministic vectors $b_1,\ldots,b_n$ in an arbitrary Banach space $(B,\Vert \cdot \Vert)$:
	\begin{equation}
		\bbP\Bigl( \Bigl\lvert\Bigl\Vert \sum_{i=1}^n Y_i b_i \Bigr\Vert - \bbE\Bigl\Vert  \sum_{i=1}^n Y_i b_i \Bigr\Vert \Bigr\rvert \geq x\Bigr) \leq \exp\Bigl( -\frac{x^2}{8 \varsigma_*(b)^2 \Vert \Gamma \Vert^2}\Bigr)\label{eq:WeepyHotel}
	\end{equation}
	where $\Vert \Gamma \Vert$ is a dependence quantity \cite[Section 2]{samson2000concentration} and
	$
		\varsigma_*(b)^2 \de \sup_{\xi \in B^*: \Vert \xi  \Vert\leq 1 } \sum_{i=1}^n  \xi(b_i)^2
	$
	with supremum runing over linear functionals $\xi:B\to \bbR$ of norm $\leq 1$.

	Suppose that the Banach space is $\bbR^{d\times d}$ with the operator norm.
	Then, the linear functionals of norm $\leq 1$ are convex combinations of those of the form $\xi(\bM) = \langle v, \bM w \rangle$ for fixed vectors $v,w \in \bbR^d$ with $\Vert v \Vert = \Vert w \Vert = 1$.
	Hence, taking $b_i = \bB_i$ and using that $\sum_{i=1}^n \xi(b_i)^2$ depends convexly on $\xi$ yields the variance proxy specified in Lemma \ref{lem: MatrixSamson}.

	Regarding the dependence quantity $\Vert \Gamma \Vert$, it suffices for our purposes that it can be bounded whenever the dependence in the Markov chain decays at an exponential rate.
	Specifically, since the total variation distance between the laws of $Z_i$ and $Z_j$ is at most $\min\{2,\psi(Z_i,Z_j) \}$ it follows from Proposition \ref{prop: ExponentialPsi} with the same argumentation as in \cite[pages 421 to 422]{samson2000concentration} that
	\begin{equation}
		\Vert \Gamma \Vert \leq \sum_{k=0}^{n-1} \sqrt{\min\bigl\{2, (1/4)^{k/\Psi(Z)}\bigr\}}.\label{eq:ZenPop}
	\end{equation}
	Thus, $\Vert \Gamma \Vert \leq c \Psi(Z)$ for some absolute constant $c>0$.
	Use this in \eqref{eq:WeepyHotel} to conclude.
\end{proof}
\begin{corollary}\label{cor: SamsonLp}
	With notation as in Lemma \ref{lem: MatrixSamson}, there exists an absolute constant $c>0$ such that for any integer $p\geq 1$,
	\begin{equation}
		\Bigl\lvert \bbE\Bigl[\Bigl\Vert \sum_{i=1}^n Y_i \bB_i \Bigr\Vert \Bigr] - \bbE\Bigl[\Bigl\Vert \sum_{i=1}^n Y_i \bB_i \Bigr\Vert^{2p} \Bigr]^{1/2p}   \Bigr\rvert  \leq c\sqrt{p \Psi(Z)} \varsigma_*(\bB).
	\end{equation}
\end{corollary}
\begin{proof}
	This is immediate from Lemma \ref{lem: MatrixSamson} since sub-Gaussianity is equivalent to moment bounds \cite[Theorem 2.1]{boucheron2013concentration}.
\end{proof}
\begin{proof}[Proof of \texorpdfstring{Lemma \ref{LEM: SAMSON}}{Lemma}]
	Note that the model \eqref{eq:AngryParrot} is of the form described in Lemma \ref{lem: MatrixSamson}, with $\bB_t = e_ie_j^{\transpose} + e_j e_i^{\transpose}$ or $\bB_t = e_i e_i^{\transpose}$ depending on $\varphi(t) = \{i,j \}$.
	Here,
	\begin{align}
		\varsigma_*^2(\bB) =\negsp \sup_{\Vert v \Vert = \Vert w \Vert = 1}\sum_{i\leq j} \langle v, \bB_{\varphi^{-1}(i,j)} w  \rangle^2 = \negsp\sup_{\Vert v \Vert = \Vert w \Vert = 1}\sum_{i=1}^n \sum_{j=1}^n v_i^2 w_j^2 = 1
	\end{align}
	where we used that $\Vert v_i \Vert = \Vert w_j \Vert = 1$.
	Thus, using Lemma \ref{lem: MatrixSamson} and Corollary \ref{cor: SamsonLp} with the triangle inequality, there exist $c,C>0$ such that for every $p\geq 1$,
	\begin{equation}
		\bbP\bigl(\bigl\lvert \Vert \bS \Vert  - \bbE[\Vert \bS \Vert^{2p}]^{1/2p} \bigr\rvert \geq  c\sqrt{p \Psi(Z)}   + x \bigr) \leq \exp(-Cx^2/\Psi(Z)). \label{eq:MadDog}
	\end{equation}
	On the other hand, taking $p$ a sufficiently large multiple of $\ln(d+1)$ depending on the desired multiplicative error $\delta$, and using the two-sided bounds in Corollary \ref{cor: MarkovMarkovBound} with the parameter estimates from Lemma \ref{lem: MarkovEntriesParamEst} and the assumption that $\Vert \bS_{i,j} \Vert_{\Linf} \leq 1$,
	\begin{align}
		\max_{-\delta \leq \gamma \leq \delta} \bigl\lvert  \bbE[{} & {}\Vert \bS \Vert^{2p}]^{1/2p} - (1+\gamma)\Vert \bS_{\free} \Vert \bigr\rvert \label{eq:NewSquid}                                                             \\
		                                                            & \leq c\Psi(Z)^{\frac{2}{3}}d^{\frac{1}{3}}\ln(d+1)^{\frac{2}{3}} + c\Psi(Z) \ln(d+1) + c \Psi(Z)^{\frac{1}{2}} d^{\frac{1}{4}} \ln(d+1)^{\frac{3}{4}}\nonumber
	\end{align}
	for some $c>0$.
	Combine \eqref{eq:MadDog}--\eqref{eq:NewSquid} and simplify using that $\sqrt{\ln(d+1)\Psi(Z)} \leq \ln(d+1)\Psi(Z)$ and $\Psi(Z)^{1/2} d^{1/4} \ln(d+1)^{3/4} \leq c' \Psi(Z)^{2/3}d^{1/3}\ln(d+1)^{2/3}$ for some $c'>0$.
\end{proof}

\section{Proofs concerning block Markov chains}\label{sec: ProofBMC}
This section concerns the proofs for Section \ref{sec: ApplicationsBMC}.
We adopt the notation that was used there.
In particular, $\bM = \sqrt{d/n} (\hat{\bN} - \bbE[\hat{\bN}])$, and $\bS$ is the self-adjoint dilation defined in \eqref{eq: Def_S_Dilation}.

This section is split in the following main parts.
In Section \ref{sec: ImprovedBMCEstimates}, we estimate the entries of $\Cov(\bM)$ and use these to provide precise the estimates on the parameters of $\bS$.
We prove Proposition \ref{prop: SingValN} in Section \ref{sec: SingvalDistribution}.
Finally, the proof of Theorem \ref{thm: BMC_Noise_Free} is given in Section \ref{sec: OpNormBMC} where we also establish a nonasymptotic concentration inequality in Proposition \ref{prop: BMC_Explicit_Concentration}.

\subsection{Estimates on the parameters of block Markovian random matrices}\label{sec: ImprovedBMCEstimates}
For any fixed $d,n\geq 1$, we introduce the following abbreviations
\begin{align}
	\mathfrak{c}_1 & \de d  \max\Bigl\{\frac{\pi_v}{\# \cV_v }:v\in \{1,\ldots,K\} \Bigr\}, \label{eq: Def_frakc1}                                                                                                                             \\
	\mathfrak{c}_2 & \de d^2 \max\Bigl\{\frac{\pi_v}{\# \cV_v }\frac{\mathbf{p}_{v,w}}{\# \cV_{w}}:v,w\in \{1,\ldots,K\} \Bigr\}, \label{eq: Def_frakc2}                                                                                       \\
	\mathfrak{c}_3 & \de d^3 \max\Bigl\{\frac{\pi_u}{\# \cV_u}\frac{\mathbf{p}_{u,v}}{\# \cV_{v}}\frac{\mathbf{p}_{v,w}}{\# \cV_{w}}: u,v,w \in \{1,\ldots,K \}\Bigr\}, \label{eq: Def_frakc3}                                                 \\
	\mathfrak{d}   & \de d^2 \max\Bigl\{ \Bigl\lvert \frac{1}{n}\sum_{t=1}^{n-3} (n-2-t) \frac{(\mathbf{p}^t)_{u,v} - \pi_v}{\# \cV_v}\Bigr\rvert \frac{\mathbf{p}_{v,w}}{\# \cV_w}  : u,v,w \in \{1,\ldots,K \} \Bigr\}.\label{eq: Def_frakd}
\end{align}
Here, we let $\mathfrak{d} = 0$ if $n-3 < 1$.
Simple upper bounds on the $\mathfrak{c}_i$ and $\mathfrak{d}$ are given in Lemma \ref{lem: FrakParamEstimates}.
\begin{lemma}
	For any $i,j,k,m \in \{1,\ldots,d \}$ it holds that
	\begin{align}
		 & \lvert \Cov(\bM)_{ij,km} \rvert \label{eq: Lem_Covijkm} \\
		 & \leq
		\begin{cases}
			\frac{1}{d}(\mathfrak{c}_2 + \frac{2}{d^2}\mathfrak{c}_2^2 + \frac{2}{d}\mathfrak{c}_3  + \frac{2}{d^2}\mathfrak{c}_2 \mathfrak{d}) & \quad \text{ if } (i,j) = (k,m),                                     \\
			\frac{1}{d^2}(\frac{3}{d}\mathfrak{c}_2^2 +  2\mathfrak{c}_3 + \frac{2}{d} \mathfrak{c}_2 \mathfrak{d})                             & \quad \text{ if } (i,j) \neq (k,m) \text{ and }(i=m \text{ or }j=k), \\
			\frac{1}{d^3}(3\mathfrak{c}_2^2 + 2\mathfrak{c}_2 \mathfrak{d})                                                                     & \quad \text{ else}.
		\end{cases} \nonumber
	\end{align}
	Furthermore, if $i \in \cV_a$ and $j \in \cV_b$ for $a,b \in \{1,\ldots,K \}$, then
	\begin{align}
		\Bigl\lvert \Cov(\bM)_{ij,ij} - d\frac{\pi_a\mathbf{p}_{a,b}}{\# \cV_{a} \#\cV_{b}} \Bigr\rvert \leq \frac{1}{d^2}\Bigl(\frac{d}{n} \mathfrak{c}_2 + \frac{3}{d}\mathfrak{c}_2^2 + 2\mathfrak{c}_3 + \frac{2}{d}\mathfrak{c}_2 \mathfrak{d}\Bigr).  \label{eq: Lem_Covijij}
	\end{align}
\end{lemma}
\begin{proof}
	Recall that the Markov chain of transitions $E = (E_t)_{t=1}^{n-1}$ is defined by $E_t = (Z_t, Z_{t+1})$.
	For any $i,j,k,m \in \{1,\ldots,n\}$ we can write
	\begin{align}
		 & \Cov(\bM)_{ij,km} = \frac{d}{n} \bbE\Bigl[ (\hat{\bN} - \bbE[\hat{\bN}])_{i,j} (\hat{\bN} - \bbE[\hat{\bN}])_{k,m} \Bigr]\label{eq: CovMijkm}                     \\
		 & \ = \frac{d}{n} \sum_{t_1,t_2 = 1}^{n-1}\bbE\Bigl[\bb1\{E_{t_1} = (i,j) \} \bb1\{E_{t_2} = (k,m) \} - \bbP(E_{t_1} = (i,j))\bbP(E_{t_2} = (k,m))  \Bigr]\nonumber \\
		 & \ \ed \frac{d}{n} \sum_{t_1, t_2=1}^{n-1} \cE_{t_1,t_2} (i,j,k,m).  \nonumber
	\end{align}

	Let us separately consider the case where $t_1 = t_2$, $\lvert t_1 - t_2 \rvert = 1$, and $\lvert t_1 - t_2 \rvert > 1$.
	If $t_1 = t_2$, then since the block Markov chain starts from stationarity,
	\begin{equation}
		\frac{d}{n}\sum_{t_1 = 1}^{n-1} \cE_{t_1,t_1}(i,j,k,m) \label{eq: t1t2=}
		=
		\begin{cases}
			-\frac{d(n-1)}{n}\bbP(E_{1} = (i,j))\bbP(E_{1} = (k,m))               & \text{ if }(i,j)\neq (k,m), \\
			\frac{d(n-1)}{n}\bigl(\bbP(E_{1} = (i,j)) - \bbP(E_1 = (i,j))^2\bigr) & \text{ if }(i,j)=(k,m),
		\end{cases}
	\end{equation}
	It here holds for $i \in \cV_{a}$ and $j\in \cV_{b}$ that $\bbP(E_1 = (i,j)) = \pi_a \mathbf{p}_{a,b}/  \# \cV_a \#\cV_b \leq \mathfrak{c}_2/d^2$.
	Hence, using that $\bbP(E_{1} = (i,j)) - \bbP(E_1 = (i,j))^2 \leq \bbP(E_{1} = (i,j))$ and $(n-1)/n \leq 1$,
	\begin{align}
		\Bigl\lvert \frac{d}{n} & \sum_{t_1 = 1}^{n-1} \cE_{t_1,t_1}(i,j,k,m) \Bigr\rvert \label{eq: Caset1=t2} \leq
		\begin{cases}
			\frac{1}{d^3} \mathfrak{c}_2^2, & \text{ if }(i,j)\neq (k,m), \\
			\frac{1}{d}\mathfrak{c}_2       & \text{ if }(i,j)=(k,m).
		\end{cases}
	\end{align}
	Furthermore, for $i\in \cV_a$ and $j \in \cV_b$,
	\begin{align}
		{} & {}\Bigl\lvert \frac{d}{n}\sum_{t_1 = 1}^{n-1} \cE_{t_1,t_1} (i,j,i,j) - d\frac{\pi_a \mathbf{p}_{a,b}}{\# \cV_a \#\cV_b} \Bigr\rvert\label{eq: Caset1=t2Precise}                  \\
		   & \leq \Bigl(\frac{dn}{n} - \frac{d(n-1)}{n}\Bigr)\bbP(E_1 = (i,j)) +  \frac{d(n-1)}{n}\bbP(E_1 = (i,j))^2  \leq \frac{\mathfrak{c}_2}{dn} + \frac{\mathfrak{c}_2^2}{d^3}.\nonumber
	\end{align}

	Now suppose that $\lvert t_1 - t_2  \rvert= 1$.
	By symmetry, it suffices to consider $t_2 = t_1 + 1$.
	Then,
	\begin{align}
		 & \Bigl\lvert\frac{d}{n}\sum_{t_1 = 1}^{n-2} \cE_{t_1,t_1+1}(i,j,k,m) \Bigr\rvert \\
		 & =
		\begin{cases}
			\frac{d(n-2)}{n}\bbP(E_{1} = (i,j))\bbP(E_{2} = (k,m))                                                      & \text{ if }j\neq k, \\
			\frac{d(n-2)}{n}\bigl\lvert \bbP(E_1 = (i,j), E_2 = (j,m)) - \bbP(E_1 = (i,j)) \bbP(E_2 = (j,m))\bigr\rvert & \text{ if }j = k.
		\end{cases}  \nonumber
	\end{align}
	A similar estimate applies when $t_1 = t_2 + 1$.
	The only difference is that the case distinction then depends on whether $i =m$ or $i \neq m$ because $\cE_{t_1,t_2}(i,j,k,m) = \cE_{t_2,t_1}(k,m,i,j)$.
	Hence,
	\begin{align}
		\Bigl\lvert \frac{d}{n}\sum_{\lvert t_1 - t_2 \rvert = 1} \cE_{t_1,t_2}(i,j,k,m) \Bigr\rvert
		\leq
		\begin{cases}
			\frac{2}{d^3} \mathfrak{c}_2^2,                             & \ \text{ if } j \neq k \text{ and }i\neq m, \\
			\frac{2}{d^2}(\mathfrak{c}_3 + \frac{1}{d}\mathfrak{c}_2^2) & \ \text{ if } j = k \text{ or }i=m.
		\end{cases}\label{eq: Caset1t2=1}
	\end{align}

	Finally, consider the case where $\lvert t_1 - t_2 \rvert >1$.
	Then, if $t_2 > t_1 + 1$,
	\begin{align}
		 & \Bigl\lvert \frac{d}{n}\sum_{t_1 = 1}^{n-3} \sum_{t_2 > t_1 + 1} \cE_{t_1,t_2}(i,j,k,m) \Bigr\rvert
		= \Bigl\lvert \frac{d}{n}\sum_{s = 2}^{n-2} (n-1-s) \cE_{1,1+s}(i,j,k,m) \Bigr\rvert                                                                            \\
		 & = d\bbP(E_1 = (i,j)) \bbP(Z_{2+s} = m \mid Z_{1+s} = k)\nonumber                                                                                             \\
		 & \qquad \qquad \times \Bigl\lvert\sum_{s = 2}^{n-2}  \frac{n-1-s}{n}\bigl( \bbP(Z_{1+s} = k \mid Z_2 = j) - \bbP(Z_{1+s} = k)  \bigr)  \Bigr\rvert. \nonumber
	\end{align}
	Recall \eqref{eq: Def_frakd} and note that $\bbP(Z_{1+s} = k \mid Z_2 = j) = (\mathbf{p}^{s-1})_{a,b}/\# \cV_b$ for any $j \in \cV_a$ and $k\in \cV_b$.
	A substitution $t = s-1$ hence yields that
	$
		\lvert \frac{d}{n}\sum_{t_1 = 1}^{n-3} \sum_{t_2 > t_1 + 1} \cE_{t_1,t_2}(i,j,k,m)\rvert
		\leq \frac{1}{d^3}\mathfrak{c}_2 \mathfrak{d}.
	$
	By using a similar estimate for the case $t_1 > t_2 +1$ we conclude that
	\begin{align}
		\Bigl\lvert \frac{d}{n} & \sum_{\lvert t_1 - t_2 \rvert >1} \cE_{t_1,t_2}(i,j,k,m) \Bigr\rvert \leq \frac{2}{d^3}\mathfrak{c}_2 \mathfrak{d}. \label{eq: Caset1t2>1}
	\end{align}

	We finally combine the foregoing estimates.
	Due to \eqref{eq: CovMijkm} it holds that
	\begin{align}
		\lvert \Cov(\bM)_{ij,km} \rvert\leq &
		\Bigl\lvert \frac{d}{n}\sum_{t_1 = 1}^{n-1} \cE_{t_1,t_1}(i,j,k,m) \Bigr\rvert   +  \Bigl\lvert \frac{d}{n}\sum_{\lvert t_1 - t_2 \rvert = 1} \cE_{t_1,t_2} (i,j,k,m) \Bigr\rvert \nonumber \\
		                                    & +  \Bigl\lvert \frac{d}{n}\sum_{\lvert t_1 - t_2 \rvert > 1} \cE_{t_1,t_1+1} (i,j,k,m) \Bigr\rvert. \label{eq: FineCovEntryEstimate}
	\end{align}
	Estimate the terms in the right-hand side of \eqref{eq: FineCovEntryEstimate} using \eqref{eq: Caset1=t2}, \eqref{eq: Caset1t2=1}, and \eqref{eq: Caset1t2>1} to complete the proof of \eqref{eq: Lem_Covijkm}.
	Similarly, the proof of \eqref{eq: Lem_Covijij} can be completed by combining \eqref{eq: Caset1=t2Precise}, \eqref{eq: Caset1t2=1}, and \eqref{eq: Caset1t2>1} together with the fact that
	\begin{align}
		\Bigl\lvert \Cov(\bM)_{ij,ij} -  d & \frac{\pi_a \mathbf{p}_{a,b}}{\# \cV_a \# \cV_b} \Bigr\rvert\leq \Bigl\lvert \frac{d}{n}\sum_{t_1 = 1}^{n-1} \cE_{t_1,t_1}(i,j,i,j) -  d\frac{\pi_a \mathbf{p}_{a,b}}{\# \cV_a \# \cV_b}\Bigr\rvert           \\
		                                   & +  \Bigl\lvert \frac{d}{n}\sum_{\lvert t_1 - t_2 \rvert = 1} \cE_{t_1,t_2} (i,j,i,j) \Bigr\rvert +  \Bigl\lvert \frac{d}{n}\sum_{\lvert t_1 - t_2 \rvert > 1} \cE_{t_1,t_1+1} (i,j,i,j)\Bigr\rvert .\nonumber
	\end{align}
	This concludes the proof.
\end{proof}

We next show that one can estimate $\Psi(Z)$ and $\Psi(E)$ in terms of $\Psi(\mathbf{p})$.
\begin{lemma}\label{lem: PsiZC}
	It holds that $\Psi(Z)  = \min\{n, \Psi(\mathbf{p})\}.$
\end{lemma}
\begin{proof}
	For any $i\in \cV_a$ and $j \in \cV_b$ it holds that $\bbP(Z_{1+t} = j \mid  Z_1 = i ) = (\mathbf{p}^t)_{a,b}/\# \cV_b$.
	The desired estimate is hence immediate from Lemma \ref{lem: Psi(Z)_finitestatespace}.
\end{proof}
\begin{lemma}\label{lem: PsiZE}
	It holds that $\Psi(E) \leq \Psi(Z) +1$.
\end{lemma}
\begin{proof}
	Since $Z$ starts in stationarity, the same holds for $E$.
	Further, for any $t>1$ and $i,j,k,m\in \{1,\ldots,d \}$ using the definition that $E_t = (Z_t,Z_{t+1})$,
	\begin{align}
		 & \max_{i,j,k,m \in \{1,\ldots,d\}}\Bigl\lvert\, \frac{\bbP(E_{1+t} = (i,j) \mid E_1 = (k,m)) - \bbP(E_{1+t} = (i,j))}{\bbP(E_{1+t} = (i,j))} \, \Bigr\rvert \leq \frac{1}{4} \Bigr\}    \\
		 & =  \max_{i,m \in \{1,\ldots,d\}}\Bigl\lvert\, \frac{\bbP(Z_{1+t-1} = i \mid Z_1 = m) -  \bbP(Z_{1 + t - 1} = i)}{\bbP(Z_{t+1} = i)} \, \Bigr\rvert \leq \frac{1}{4} \Bigr\}. \nonumber
	\end{align}
	The desired result now follows from Lemma \ref{lem: Psi(Z)_finitestatespace}.
\end{proof}

We bound the parameters \eqref{eq: Def_frakc1}--\eqref{eq: Def_frakd}.
Recall that
$
	\hat{\alpha}_{\min} = \min\{\# \cV_{k}/d: k \in  \{1,\ldots,K \} \}.
$
\begin{lemma}\label{lem: FrakParamEstimates}
	For any $i \in \{1,2,3 \}$ it holds that $\mathfrak{c}_i \leq \hat{\alpha}_{\min}^{-i}$ and $\mathfrak{d} \leq \frac{4}{3}\Psi(\mathbf{p})\hat{\alpha}_{\min}^{-2}$.
\end{lemma}
\begin{proof}
	The estimate $\mathfrak{c}_i \leq \hat{\alpha}_{\min}^{-i}$ is immediate from the definitions \eqref{eq: Def_frakc1}--\eqref{eq: Def_frakc3} and the fact that $\pi_v \leq 1$ and $\mathbf{p}_{v,w} \leq 1$ for all $v,w \in \{1,\ldots,K \}$.
	Further, for any $u,v,w\in \{1,\ldots,K \}$ the triangle inequality together the estimates $\#\cV_v/d \geq  \hat{\alpha}_{\min}$ and $\mathbf{p}_{v,w}\leq 1$ yields that
	\begin{align}
		\Bigl\lvert \frac{1}{n}\sum_{t=1}^{n-2} (n-2-t) & \frac{(\mathbf{p}^t)_{u,v} - \pi_v}{\# \cV_v}\Bigr\rvert \frac{\mathbf{p}_{v,w}}{\# \cV_w}\label{eq: frakdestimate}
		\leq \sum_{t=1}^{n-2}\Bigl\lvert  \frac{(\mathbf{p}^t)_{u,v} - \pi_v}{\# \cV_v}\Bigr\rvert \frac{\mathbf{p}_{v,w}}{\# \cV_w}                                                                                                  \\
		                                                & \leq d^{-2}\hat{\alpha}_{\min}^{-2}\Bigl((\Psi(\mathbf{p})-1) + \sum_{t = \Psi(\mathbf{p})}^{\infty} \bigl\lvert (\mathbf{p}^t)_{u,v} - \pi_v \bigr\rvert \Bigr). \nonumber
	\end{align}
	Since $\pi_v \leq 1$, one can relate the right-hand side of \eqref{eq: frakdestimate} to the $\psi$-dependence coefficient:
	\begin{align}
		\max\Bigl\{\lvert (\mathbf{p}^t)_{u,v} - \pi_v \rvert : u,v \in \{1,\ldots,K \}  \Bigr\} & \leq \max\Bigl\{\frac{\lvert (\mathbf{p}^t)_{u,v} - \pi_v \rvert }{\pi_v}: u,v \in \{1,\ldots,K \}  \Bigr\}.\nonumber
	\end{align}
	Due to Proposition \ref{prop: ExponentialPsi} it hence follows that
	\begin{equation}
		\mathfrak{d} \leq \hat{\alpha}_{\min}^{-2}\Bigl((\Psi(\mathbf{p})-1) + \negsp\sum_{t = \Psi(\mathbf{p})}^{\infty}\negsp\Bigl(\frac{1}{4}\Bigr)^{\lfloor t/ \Psi(\mathbf{p}) \rfloor}  \Bigr) = \hat{\alpha}_{\min}^{-2}\Bigl((\Psi(\mathbf{p})-1) + \Psi(\mathbf{p})\sum_{s = 1}^{\infty}\Bigl(\frac{1}{4}\Bigr)^{s}\Bigr).
	\end{equation}
	Use that $\sum_{t=1}^\infty (1/4)^t = 1/3$ to conclude the proof.
\end{proof}
Recall that $\Psi(\mathbf{p})$ refers to the $\psi$-mixing time of the Markov chain on $\{1,\ldots,K \}$ with transition matrix $\mathbf{p}$.
By Lemma \ref{lem: Psi(Z)_finitestatespace}, one can express $\Psi(\mathbf{p})$ more explicitly as
\begin{align}
	\Psi(\mathbf{p}) \de  \min\Bigl\{t \geq 1: \max_{i,j \in \{1,\ldots,K \}}\Bigl\lvert\, \frac{(\mathbf{p}^t)_{i,j} - \pi_j}{\pi_j} \, \Bigr\rvert \leq \frac{1}{4}\Bigr\}. \label{eq: PsiC}
\end{align}

\begin{lemma}\label{lem: RefinedParamEstimates}
	With $\bS$, $\bX$, and $E$ as in \eqref{eq: Def_S_Dilation}--\eqref{eq: Def_Xt} it holds that
	\begin{align}
		R(\bX) \leq 2\sqrt{d/n}, \ \ \Psi(E) \leq  \Psi(\mathbf{p}) +1, \ \ \varsigma(\bX)^2 \leq \mathfrak{c}_1, \ \
		\sigma(\bS)^2 \leq \mathfrak{g}, \ \ v(\bS)^2  \leq d^{-1}\mathfrak{v}, \nonumber
	\end{align}
	where $\mathfrak{g}$ and $\mathfrak{v}$ are explicit and satisfy $\mathfrak{g} \leq \mathfrak{c}_1 + Cd^{-1}\hat{\alpha}_{\min}^{-4} \Psi(\mathbf{p})$ and $\mathfrak{v}\leq  C' \hat{\alpha}_{\min}^{-4} \Psi(\mathbf{p})$ for certain absolute constants $C,C' >0$.
\end{lemma}
\begin{proof}
	The estimate $R(\bX) \leq 2 \sqrt{d/n}$ is immediate from the definition \eqref{eq: Def_Xt} of the matrices $\bX_t$.
	Similarly, the estimate on $\Psi(E)$ is immediate from  Lemmas \ref{lem: PsiZC} and \ref{lem: PsiZE}.

	We next consider  $\varsigma(\bX)$.
	The $\bX_t$ are identically distributed since the block Markov chain is assumed to start in stationarity.
	Hence, using that for any self-adjoint random matrix $\bA$ one has $\bbE[(\bA - \bbE[\bA])^2] = \bbE[\bA^2] - \bbE[\bA]^2 \preceq \bbE[\bA^2]$ with the positive semidefinite order,
	\begin{align}
		\sum_{t=1}^{n-1} \bbE[\bX_t^2] = (n-1) \bbE[\bX_1^2]
		 & \preceq d\sum_{i,j = 1}^d \bbP(E_1 = (i,j))
		\begin{pmatrix}
			0                     & e_{i}e_j^\transpose \\
			e_{j}e_{i}^\transpose & 0
		\end{pmatrix}^2 \nonumber              \\
		 & = d \sum_{i = 1}^d \bbP(Z_1 = i)\begin{pmatrix}
			                                   e_ie_i^\transpose & 0 \\
			                                   0                 & 0
		                                   \end{pmatrix} +
		d\sum_{j= 1}^d\bbP(Z_2 = j)\begin{pmatrix}
			                           0 & 0                 \\
			                           0 & e_je_j^\transpose
		                           \end{pmatrix}. \nonumber
	\end{align}
	Since the operator norm of a positive diagonal matrix is its maximal element, it follows that
	$
		\varsigma(\bX)^2
		\leq \max\{ d\pi_k/\# \cV_k:k=1,\ldots,K \} = \mathfrak{c}_1
	$,
	as desired.

	We next consider $v(\bS)^2$.
	Using \cite[Lemma 4.10]{bandeira2021matrix} and that $\operatorname{Cov}(\bM)$ being symmetric implies that its operator norm is bounded by the absolute row sums by \cite[Corollary 2.3.2]{golub2013matrix},
	\begin{align}
		v(\bS)^2 \leq 2 \Vert \Cov(\bM) \Vert_{\op}\leq \max_{i,j=1,\ldots,d} \sum_{k,m = 1}^d \lvert \Cov(\bM)_{ij,km} \rvert. \label{eq: CovSumEstimate}
	\end{align}
	For any fixed $i,j$ there is precisely one term in \eqref{eq: CovSumEstimate} with $(k,m) = (i,j)$; at most $2d$ terms with $(k,m) \neq (i,j)$ but $i = m$ or $j=k$; and at most $d^2$ remaining terms.
	The combination of \eqref{eq: CovSumEstimate} and \eqref{eq: Lem_Covijkm} hence yields that $v(\bS)^2 \leq d^{-1} \mathfrak{v}$ with
	\begin{equation}
		\mathfrak{v} \de 2\Bigl(\Bigl(\mathfrak{c}_2 + \frac{2}{d^2}\mathfrak{c}_2^2 + \frac{2}{d}\mathfrak{c}_3  + \frac{2}{d^2}\mathfrak{c}_2 \mathfrak{d}\Bigr) + 2\Bigl(\frac{3}{d}\mathfrak{c}_2^2 +  2\mathfrak{c}_3 +   \frac{2}{d} \mathfrak{c}_2 \mathfrak{d}\Bigr) + \Bigl(3\mathfrak{c}_2^2 + 2\mathfrak{c}_2 \mathfrak{d}\Bigr) \Bigr).\label{eq: Def_frakv}
	\end{equation}
	The claimed upper bound on $\mathfrak{v}$ now follows from Lemma \ref{lem: FrakParamEstimates}.

	We next consider the estimate on $\sigma(\bS)^2$.
	A direct computation shows that $\bS^2$ is a block diagonal matrix with diagonal blocks $\bM \bM^{\transpose}$ and $\bM^{\transpose}\bM$.
	Consequently, taking expectations and the operator norm on both sides, $\sigma(\bS)^2 = \max\{\Vert \bA \Vert_{\op} ,\allowbreak \Vert \bB\Vert_{\op} \}$ where $\bA$ and $\bB$ are the self-adjoint $d\times d$ matrices whose entries are given by
	\begin{align}
		\bA_{i,j} = \sum_{\ell =1}^d\Cov(\bM)_{i\ell, j\ell},\qquad
		\bB_{i,j} = \sum_{\ell =1}^d\Cov(\bM)_{\ell i, \ell j}.
	\end{align}
	We consider the diagonal and the off-diagonal terms of these matrices separately.
	First, consider the case where $i \neq j$.
	Then, $(i,\ell) \neq (j,\ell)$ and $(\ell,i) \neq (\ell, j )$ for all $\ell \in \{1,\ldots,d \}$.
	Hence, using \eqref{eq: Lem_Covijkm} with separate consideration of the case where $i = \ell$ or $j = \ell$
	\begin{align}
		\lvert \bA_{i,j} \rvert
		 & \leq \frac{2}{d^2}\Bigl(\frac{3}{d}\mathfrak{c}_2^2 + 2\mathfrak{c}_3 + \frac{2}{d}\mathfrak{c}_2 \mathfrak{d}\Bigr) + \frac{d}{d^3}(3\mathfrak{c}_2^2 + 2 \mathfrak{c}_2 \mathfrak{d}), \\
		\lvert \bB_{i,j} \rvert
		 & \leq \frac{2}{d^2}\Bigl(\frac{3}{d}\mathfrak{c}_2^2 + 2\mathfrak{c}_3 + \frac{2}{d}\mathfrak{c}_2 \mathfrak{d}\Bigr) + \frac{d}{d^3}(3\mathfrak{c}_2^2 + 2 \mathfrak{c}_2 \mathfrak{d}).
	\end{align}
	For the case with $i = j$ we get better leading-order term if we replace \eqref{eq: Caset1=t2} by the following:
	\begin{align}
		\Bigl\lvert \frac{d}{n}\sum_{\ell = 1}^d \sum_{t_1 = 1}^{n-1} \cE_{t_1,t_1}(i,\ell,i,\ell) \Bigr\rvert
		 & = \frac{d(n-1)}{n} \sum_{\ell = 1}^d \bigl(\bbP(E_{1} = (i,\ell)) - \bbP(E_1 = (i,\ell))^2\bigr)  \label{eq: Case_ellt1=t2} \\
		 & \leq d \bbP(Z_1 =i)  \leq \mathfrak{c}_1.\nonumber
	\end{align}
	It was here used that $\bbP(E_1 = (i,\ell)) \geq \bbP(E_1 = (i,\ell))^2$.
	Now observe that by \eqref{eq: CovMijkm},
	\begin{align}
		\Bigl\lvert \sum_{\ell =1}^d\Cov(\bM)_{i\ell, i\ell} \Bigr\rvert \leq & \Bigl\lvert \frac{d}{n}\sum_{\ell = 1}^d \sum_{t_1 = 1}^{n-1} \cE_{t_1,t_1}(i,\ell,i,\ell) \Bigr\rvert +  \Bigl\lvert \frac{d}{n}\sum_{\ell = 1}^d\sum_{\lvert t_1 - t_2 \rvert = 1} \cE_{t_1,t_2} (i,\ell,i,\ell) \Bigr\rvert \nonumber \\
		                                                                      & +  \Bigl\lvert \frac{d}{n}\sum_{\ell = 1}^d \sum_{\lvert t_1 - t_2 \rvert > 1} \cE_{t_1,t_1+1} (i,\ell,i,\ell) \Bigr\rvert. \label{eq: Covsumell}
	\end{align}
	Hence, by using \eqref{eq: Caset1t2=1}, \eqref{eq: Caset1t2>1}, and \eqref{eq: Case_ellt1=t2} in \eqref{eq: Covsumell}, $
		\lvert \bA_{i,i}  \rvert
		\leq \mathfrak{c}_1 + d^{-1}(2d^{-1} \mathfrak{c}_2^2 + 2\mathfrak{c}_3  + 2d^{-1} \mathfrak{c}_2 \mathfrak{d})$.
	Exactly the same estimate applies to $\lvert \bB_{i,i} \rvert$.
	Now using \cite[Corollary 2.3.2]{golub2013matrix} as in \eqref{eq: CovSumEstimate}, we conclude that $\sigma(\bS)^2 \leq \mathfrak{g}$ with
	\begin{equation}
		\mathfrak{g} \de \mathfrak{c}_1 + \frac{1}{d}\Bigl(\frac{2}{d} \mathfrak{c}_2^2 + 2\mathfrak{c}_3  + \frac{2}{d} \mathfrak{c}_2 \mathfrak{d}\Bigr) + \frac{1}{d} \Bigl(2\Bigl( \frac{3}{d}\mathfrak{c}_2^2 + 2\mathfrak{c}_3 + \frac{2}{d}\mathfrak{c}_2 \mathfrak{d} \Bigr) + (3\mathfrak{c}_2^2 + 2 \mathfrak{c}_2 \mathfrak{d}) \Bigr).  \label{eq: Def_frakg}
	\end{equation}
	The claimed upper bound on $\mathfrak{g}$ now follows from Lemma \ref{lem: FrakParamEstimates}.
\end{proof}

\subsection{Convergence of singular value distributions}\label{sec: SingvalDistribution}

Recall the definition of the empirical singular value \eqref{eq: Def_SingvalMeasure} and that $\bS$ was defined by a self-adjoint dilation in \eqref{eq: Def_S_Dilation}.
It follows that $\sym(\nu_{\bM}) = \mu_{\bS} $ where $\mu_{\bS}$ is the \emph{empirical eigenvalue distribution} of $\bS$, defined by
\begin{align}
	\mu_{\bS}([a,b]) \de \frac{1}{2d}\#\bigl\{i\in\{1,\ldots,2d \}: a \leq \lambda_i(\bS) \leq b   \bigr\}
\end{align}
where the $\lambda_i(\bS)$ are the eigenvalues of $\bS$.
We rely on the well known moment method to establish the limiting law of $\mu_{\bS}$ and state it as a lemma for the sake of clarity.
\begin{lemma}\label{lem: MomentMethod}
	Consider a sequence of self-adjoint random matrices $(\bS_{i})_{i=1}^\infty$ and let $\mu$ be a deterministic compactly supported probability measure on $\bbR$.
	If for every integer $p\geq 1$,
	\begin{align}
		\lim_{i \to \infty} \bbE[\tr \bS_i^{p}] = \int x^p \intd \mu(x)\qquad \text{ and } \qquad \lim_{i\to \infty}\Var[\tr \bS_i^{p}] = 0,  \nonumber
	\end{align}
	then $\mu_{\bS_i}$ converges weakly in probability to $\mu$ as $i$ tends to infinity.
\end{lemma}
\begin{proof}
	This is standard, for instance implicit in the proof of the Wigner semicircular law in \cite[Section 2.1.2]{anderson2010introduction}.
\end{proof}

\begin{lemma}\label{lem: BMC_tracialmoment}
	If $\lim_{d\to \infty} d/n = 0$, then for every positive integer $p\geq 1$
	\begin{align}
		\lim_{d\to \infty} \lvert \bbE[\tr \bS^{p}] -\bbE[\tr \bG^p] \rvert = 0 \quad \text{ and }\quad \lim_{d\to \infty} \lvert \Var[\tr \bS^{p}] -\Var[\tr \bG^p] \rvert = 0. \nonumber
	\end{align}
\end{lemma}
\begin{proof}
	Using that $\liminf_{d\to \infty}\min\{ \#\cV_k/d: k=1,\ldots,K\}  >0$ in the considered limiting regime and that $\lim_{d\to \infty} d/n = 0$ it follows from Lemmas \ref{lem: FrakParamEstimates} and \ref{lem: RefinedParamEstimates} that
	\begin{equation}
		\lim_{d\to \infty}R(\bX)  = 0,\qquad
		\limsup_{d\to \infty}\Psi(E) < \infty,\qquad
		\limsup_{d\to \infty}\varsigma(\bX)^2 <\infty.  \label{eq: ParamBMCLimitEstimates}
	\end{equation}
	Hence, it follows from Theorem \ref{thm: MainTracial} that $\lim_{d\to \infty} \lvert \bbE[\tr \bS^{2p}] - \bbE[\tr \bG^{2p}] \rvert = 0$.
	Further, it holds for every odd $p$ that $\tr \bS^{p} = 0 = \tr \bG^p$ by definition of a self-adjoint dilation \eqref{eq: Def_S_Dilation}.

	It now remains to show that the variance of even tracial moments is universal, for which purpose it suffices to show that $\lim_{d\to \infty} \lvert \bbE[(\tr \bS^{2p})^{2}]  - \bbE[(\tr \bG^{2p})^{2}]  \rvert = 0$.
	For this purpose, we use the trick outlined in Remark \ref{rem: TracialTricks}: note that for every fixed $t\in \bbR$,
	\begin{equation}
		\bbE\Bigl[\tr (\bS \otimes \b1 + t \b1 \otimes \bS)^{2p}\Bigr]  = \sum_{j=0}^{2p} \binom{2p}{j} t^j \bbE\Bigl[\tr[ \bS^{2p -j}]\tr[\bS^j]\Bigr].\label{eq:GladArm}
	\end{equation}
	Hence, since pointwise convergence of polynomials implies convergence of coefficients, it suffices to prove universality for $\bbE[\tr (\bS \otimes \b1 + t \b1 \otimes \bS)^{2p}]$.
	(Consider $j=p$.)

	To this end, note that the matrix $\bS \otimes \b1 + t \b1 \otimes \bS$ can be represented as a Markovian model.
	Moreover, direct computation shows that the parameters are again of the same asymptotic order.
	Indeed, $\Vert \bX_i \otimes \b1 + t\b1 \otimes \bX_i  \Vert  \leq (1+t)R(\bX)$ and $\bbE[(\bX_i \otimes \b1 + t \b1 \otimes \bX_i)^2] = \bbE[\bX_i^2] \otimes \b1 + t \b1 \otimes \bbE[\bX_i^2]$ since  $\bbE[\bX_i] = 0$ so that the variance quantity is bounded by $(1+t) \varsigma(\bX)$.
	Hence, using the parameter estimates from \eqref{eq: ParamBMCLimitEstimates} together with Theorem \ref{thm: MainTracial} gives pointwise convergence for the polynomial \eqref{eq:GladArm}, concluding the proof.
\end{proof}
In order to prove Proposition \ref{prop: SingValN} it is now sufficient to consider the empirical eigenvalue distribution of the Gaussian model.
We will do this by using  \cite[Theorem 4.2]{vanwerde2023singular}.

The proofs in \cite{vanwerde2023singular} are not directly applicable to $\bS$ in the sparse regime $n \ll d^2$.
And outside the sparse regime, \cite{vanwerde2023singular} relies on quite some computation to prove that \cite[Theorem 4.2]{vanwerde2023singular} is applicable to $\bS$, by directly determining the joint moments of the entries of the matrix.
Gaussian universality allows one to bypass this, since the higher joint moments of a Gaussian vector are determined by its covariance structure.

\begin{remark}
	Given Gaussian universality, there are also other arguments that can recover the following results.
	For instance, the Gaussian comparison results in \cite[Section 8.1]{brailovskaya2022universality} can show that the spectral distribution and norm of $\bG$ are well-approximated by those of a self-adjoint Gaussian matrix $\tilde{\bG}$ with \emph{independent} entries.
	Approximating $\tilde{\bG}$ next by a free probabilistic object using \cite{bandeira2021matrix}, and then using the matrix Dyson equation \cite[Equation (1.5)]{haagerup2005new} for the spectral law of the latter object, would do the trick.
\end{remark}

For any positive semidefinite $D\times D$ matrix $\Sigma \succcurlyeq 0$ we define the following parameter measuring the size of the off-diagonal terms:
\begin{align}
	\epsilon(\Sigma) \de \max\bigl\{\lvert \Sigma_{i,j} \rvert: i\neq j,\ i,j \in \{1,\ldots,D \} \bigr\}.\label{eq: Def_epsilon}
\end{align}
The following lemma is used to verify a condition in \cite{vanwerde2023singular}.
\begin{lemma}\label{lem: GaussianVecMoments}
	Fix a positive integer $D \geq 1$ and positive scalars $\ell,u>0$ and define a set of positive semidefinite matrices by
	\begin{align}
		\fS_D(\ell,u) \de  \Bigl\{\Sigma \in \bbR^{D\times D}:\Sigma \succcurlyeq 0 \text{ and }  \ell \leq \Sigma_{i,i} \leq u \text{ for all }i\in\{1,\ldots,D \}\Bigr\}.\nonumber
	\end{align}
	Then, for any nonnegative integers $0 \leq r \leq D$ and $p_1,\ldots,p_D \geq 0$ with $p_i = 1$ for $i \in \{1,\ldots,r \}$ there exists $C>0$ such that for any centered Gaussian vector $g$ with covariance matrix $\Sigma \in \fS_D(\ell,u)$,
	\begin{align}
		\bigl\lvert \bbE\bigl[g_{1}^{p_1}g_{2}^{p_2} \cdots g_{D}^{p_D} \bigr] \bigr\rvert \leq  C \epsilon(\Sigma)^{\frac{r}{2}} \quad \textnormal{ and }\quad \lvert \bbE[g_{1}^{2}\cdots g_{D}^{2}] -  \bbE[g_{1}^{2}] \cdots \bbE[ g_{D}^{2}]\rvert  \leq C\epsilon(\Sigma)^2. \label{eq: gmoment}
	\end{align}
\end{lemma}
\begin{proof}
	This follows readily from direct calculations with the properties of the Gaussian distribution.
	For instance, one can proceed by induction on $D$.
	The base case $D=1$ is immediate from the assumption that $\bbE[g_1] = 0$, and in the the inductive step one can to exploit that the conditional distribution of $g_1$ given $g_2,\ldots,g_D$ is explicit  \cite[Proposition 3.13]{eaton1983multivariate}.
	This computation is straightforward but tedious, so we omit the details.\footnote{If required, these details can be found in the first arXiv version of this paper; see arXiv:2307.11632v1.}
\end{proof}

\begin{proof}[Proof of Proposition \ref{prop: SingValN}]
	For any $d\geq 1$ let $\bA_d \de \sqrt{d} \bG$ denote a rescaled version of the corresponding $2d\times 2d$ Gaussian model $\bG$ for $\bS$.
	We claim that the sequence of random matrices $(\bA_d)_{d=K}^\infty$ is \emph{approximately uncorrelated with variance profile} as defined in \cite[Definition 4.1]{vanwerde2023singular}.
	This means that we have to show that for any fixed non-negative integers $0 \leq r \leq D$ and $p_1,\ldots,p_D \geq 0$ with $p_i = 1$ for $i=1,\ldots,r$,
	\begin{align}
		\limsup_{d\to \infty} \max_{\forall k \neq l: \{i_k,j_k\} \neq \{i_l,j_l\}} d^{\frac{r}{2}}\lvert \bbE[\bA_{d,i_1j_1}^{p_1}\bA_{d,i_2j_2}^{p_2} \cdots \bA_{d,i_Dj_D}^{p_D} ] \rvert < \infty \label{eq: ApproxUncorGaussianMoments}
	\end{align}
	and
	\begin{align}
		\limsup_{d\to \infty}\max_{\forall k \neq l: \{i_k,j_k\} \neq \{i_l,j_l\}} \lvert \bbE[\bA_{d,i_1j_1}^{2} \cdots \bA_{d,i_Dj_D}^{2} ] -  \bbE[\bA_{d,i_1j_1}^{2}] \cdots \bbE[\bA_{d,i_Dj_D}^{2} ]\rvert = 0 \label{eq: ApproxUncorGaussianVariance}
	\end{align}
	with the maxima running over all values of $(i_1,j_1),\ldots,(i_D,j_D) \in \{1,\ldots,2d \}^2$ with $\{i_k, j_k  \} \neq \{i_l,j_l \}$ for all $k \neq l$.

	If $\bA_{d,i_k,j_k} = 0$ almost surely, then there is nothing to prove so assume that $\Var[\bA_{i_k, j_k}] \neq 0$.
	Recall that $\bG$ is a Gaussian model of $\bS$ which is a self-adjoint dilation of $\bM$.
	The covariance of the entries of $\bA$ is hence a function of the covariance of the entries of $\bM$.
	Hence, applying Lemma \ref{lem: GaussianVecMoments} to the vector $g \de (\bA_{d,i_1j_1},\ldots ,\bA_{d,i_Dj_D})$ yields  \eqref{eq: ApproxUncorGaussianMoments} and \eqref{eq: ApproxUncorGaussianVariance} if there exist constants $\ell, u,c >0$ such that $\ell \leq d\Cov(\bM)_{ij,ij} \leq u$ and $\lvert d\Cov(\bM)_{ij,km}  \rvert\leq c/d$ for all $(i,j) \neq (k,m)$.
	Indeed, $\ell \leq d\Cov(\bM)_{ij,ij} \leq u$ shows that the assumption of Lemma \ref{lem: GaussianVecMoments} is satisfied while $\lvert d\Cov(\bM)_{ij,km}  \rvert\leq c/d$ yields that $\epsilon(\Cov(g)) \leq c/d$ with $\epsilon$ as in \eqref{eq: Def_epsilon} which ensures that Lemma \ref{lem: GaussianVecMoments} provides a sufficiently strong conclusion.

	The required upper bounds on $\lvert d\Cov(\bM)_{ij,km}  \rvert$ and  $d\Cov(\bM)_{ij,ij}$ can be found in \eqref{eq: Lem_Covijkm}.
	For the lower bound on $\Cov(\bM)_{ij,ij}$, note that \eqref{eq: Lem_Covijij} implies that for all $a,b \in \{1,\ldots,K \}$,
	\begin{align}
		\lim_{d\to \infty}
		\max_{i \in \cV_a, j \in \cV_b}
		\Bigl\lvert d\Cov(\bM)_{ij,ij} - \frac{\pi_a\mathbf{p}_{a,b}}{\alpha_a \alpha_b} \Bigr\rvert
		=
		0
		.
		\label{eq: LimitingVarProfile}
	\end{align}
	In particular, since $\frac{\pi_a\mathbf{p}_{a,b}}{\alpha_a \alpha_b}\neq 0$ due to the preliminary reduction to the case where $\Var[A_{d,i_k,j_k}] \neq 0$, there exists a constant $\ell >0$ such that $\ell \leq d\Cov(\bM)_{ij,ij}$ for all $d$.

	We may note that \eqref{eq: LimitingVarProfile} is exactly the same limiting variance profile as occurs in the dense regime with $n = d^2$ in \cite[Equation (23)]{vanwerde2023singular}.
	Consequently, since the universal limiting eigenvalue distribution in \cite[Theorem 4.2]{vanwerde2023singular} only depends on the limiting variance profile, the empirical eigenvalue distribution of $\bG = \bA_d/\sqrt{d}$ has the same universal limit as is predicted by \cite[Theorem 1.1]{vanwerde2023singular}.
	More precisely, it follows from \cite[Lemma 6.5 and Lemma 6.6]{vanwerde2023singular} that
	\begin{align}
		\lim_{d \to \infty} \bbE[\tr \bG^{p}] = \int x^p  \intd \sym(\nu_{\infty})(x)\quad \text{ and } \quad \lim_{i\to \infty}\Var[\tr \bG^{p}] = 0\label{eq: GMoments}
	\end{align}
	with $\nu_{\infty}$ as in Proposition \ref{prop: SingValN}.
	The result now follows from Lemmas \ref{lem: MomentMethod} and \ref{lem: BMC_tracialmoment}, and the fact that $\mu_\bS = \sym(\nu_{\bM})$.
\end{proof}

\begin{proof}[Proof of Corollary \ref{cor: SingValNN}]
	Note that $\operatorname{rank}\bbE[\hat{\bN}] \leq K$.
	The desired result hence follows from Proposition \ref{prop: SingValN} since low-rank perturbations do not affect limiting singular value distributions.
	More precisely, one can combine \cite[Theorem A.43]{bai2010spectral} with a self-adjoint dilation.
\end{proof}
\begin{proposition}[Abundance of singular values near \texorpdfstring{$\mathfrak{m}$}{m}]\label{prop: SingvalLowerBound}
	Adopt the assumptions of Proposition \ref{prop: SingValN}.
	Then, for any $\varepsilon >0$ there exists some constant $c>0$ such that
	\begin{align}
		\lim_{d\to \infty} \bbP\bigl(\# \{i \in \{1,\ldots,d \}:  s_i(\bM) \in  (\mathfrak{m} - \varepsilon, \mathfrak{m} + \varepsilon) \} < cd\bigr) = 0.\label{eq: SingFrac}
	\end{align}
	In particular, since $\Vert \bM \Vert_{\op}$ is the greatest singular value,  $
		\lim_{d\to \infty}\bbP(\Vert \bM \Vert_{\op} < \mathfrak{m} - \varepsilon) = 0.$
\end{proposition}
\begin{proof}
	We can interpret the measure $\sym(\nu_{\infty})$ in Proposition \ref{prop: SingValN} as the spectral law of a $2K\times 2K$ free-probabilistic object.
	Specifically, if $\bH$ is the symmetric $2K\times 2K$ Gaussian matrix with independent entries satisfying that for every $i,j \leq K$,
	\begin{align}
		\operatorname{Var}[\bH_{i,j}]=0 \ \textnormal{ and }\ \operatorname{Var}[\bH_{i,j+K}] = \alpha_i^{-1} \pi_i \mathbf{p}_{i,j}
	\end{align}
	then, it is classical (see \eg \cite[Chapter 9]{mingo2017free} or \cite[Equation (1.5)]{haagerup2005new}) that the system of equations in Proposition \ref{prop: SingValN} corresponds to the Stieltjes transform of the spectral law of the free-probabilistic object $\bH_{\free}$ defined in \cite[Section 2.1]{bandeira2021matrix}.

	To be more specific, it is part of the definition of the latter object that it is an element of a free probability space of the form $\cA^{2K\times 2K}$ where $\cA$ is a $C^*$ algebra equipped with a faithful state $\tau:\cA \to \bbC$, and the statement that $\operatorname{sym}(\nu_{\infty})$ corresponds to the spectral law means that $(\tr \otimes \tau) (f(\bH_{\free})) = \int f(x) \intd \sym(\nu_\infty) $ for every continuous function $f:\bbR \to  \bbR$ where $f(\bH_{\free})$ is defined by functional calculus.
	Thus, it follows from \cite[(3.26)]{nica2006lectures}  that $\Vert \bH_{\free} \Vert$ is the upper edge of the support of this measure $\operatorname{sym}(\nu_{\infty})$.
	In particular, the measure assigns nonzero mass to any oven interval containing $\Vert \bH_{\free}  \Vert$.
	Now, \eqref{eq: SingFrac} follows from Proposition \ref{prop: SingValN} since $\Vert \bH_{\free} \Vert = \mathfrak{m}$ with $\mathfrak{m}$ as in Theorem \ref{thm: BMC_Noise_Free} by \cite[Lemma 3.2]{bandeira2021matrix}.
\end{proof}

\subsection{Sharp upper bounds on $\Vert \bM \Vert_{\op}$}\label{sec: OpNormBMC}
The main technical result in this section concerns a nonasymptotic upper bound on $\Vert \bS_{\free} \Vert$; see Lemma \ref{lem: SfreeEstimate}.
We further prove Theorem \ref{thm: BMC_Noise_Free}, and establish a nonasymptotic concentration inequality in Proposition \ref{prop: BMC_Explicit_Concentration}.

Recall that $\hat{\alpha}_i \de \#\cV_{i}/d$ and define a scalar $\hat{\mathfrak{m}}$ analogously to Theorem \ref{thm: BMC_Noise_Free} by
\begin{align}
	\hat{\mathfrak{m}} & \de \inf_{x \in \bbR_{>0}^{2K}} \max_{i= 1,\ldots,2K} \Bigl\{ \frac{1}{x_i} + \sum_{j=1}^{2K} \hat{c}_{i,j} x_j \Bigr\} \label{eq: Def_nhat}
\end{align}
where the infimum runs over all vectors $x$ with strictly positive coordinates and the coefficients $(\hat{c}_{i,j})_{i,j=1}^{2K}$ are defined by
\begin{align}
	\hat{c}_{i,j} = \begin{cases}
		                0                                                     & \text{ if } i \leq K \text{ and }j\leq K, \\
		                \hat{\alpha}_{i}^{-1} \pi_i\mathbf{p}_{i,j - K}       & \text{ if } i\leq K \text{ and }j > K,    \\
		                0                                                     & \text{ if }  i > K \text{ and }j >K,      \\
		                \hat{\alpha}_{i-K}^{-1}  \pi_{j} \mathbf{p}_{j, i -K} & \text{ if } i >K \text{ and }j \leq K.
	                \end{cases}. \label{eq: Def_chat}
\end{align}
Let us further introduce the following parameter:
\begin{align}
	\mathfrak{u} \de \max_{i \in \{1,\ldots,2K \}} \min\Bigl\{ \frac{2\sqrt{\mathfrak{c}_1}}{\hat{c}_{i,j}} : j \in \{1,\ldots,2K\}  \Bigr\}.\label{eq: Def_frakz}
\end{align}
\begin{lemma}\label{lem: nhat_minimizer_estimate}
	There exists a vector $x^*\in \bbR^{2K}_{>0}$ which realizes the infimum in \eqref{eq: Def_nhat} and this vector satisfies that
	$
		\frac{1}{2\mathfrak{c}_1}\leq x_i^* \leq \mathfrak{u}
	$
	for every $i\in \{1,\ldots,2K \}$.
\end{lemma}
\begin{proof}
	As in the proof of Proposition \ref{prop: SingvalLowerBound}, it follows from \cite[Lemma 3.2]{bandeira2021matrix} that $\hat{\mathfrak{m}} = \Vert \hat{\bH}_{\free} \Vert$ where $\hat{\bH}$ is the symmetric $2K \times 2K$ Gaussian matrix
	with
	$\operatorname{Var}[\hat{\bH}_{i,j}]=0
	$
	and
	$\operatorname{Var}[\hat{\bH}_{i,j+K}] = \alpha_i^{-1} \pi_i \mathbf{p}_{i,j}$
	for every $i,j \leq K$.
	Due to \eqref{eq: Pisier} it consequently holds that
	\begin{align}
		\hat{\mathfrak{m}} \leq 2\sigma(\hat{\bH})= 2\Vert \bbE[\hat{\bH}^2] \Vert_{\op}^{\frac{1}{2}} = 2\sqrt{\max\Bigl\{ \frac{\pi_a}{\hat{\alpha}_a}: a \in \{1,\ldots,K \} \Bigr\}} = 2 \sqrt{\mathfrak{c}_1}  \label{eq: PisierBoundn}
	\end{align}
	where we used that the operator norm of the diagonal matrix $\bbE[\hat{\bH}^2]$ is its maximal element.

	Let us now define a subset of $\bbR^{2K}$ by
	\begin{align}
		\fX \de \Bigl\{x\in \bbR^{2K}_{>0}: \frac{1}{2\sqrt{\mathfrak{c}_1}} \leq  x_i \leq \mathfrak{u}\text{ for all }i\in \{1,\ldots,2K\} \Bigr\}.
	\end{align}
	Then, due to \eqref{eq: PisierBoundn}, the infimum in \eqref{eq: Def_nhat} may be restricted to $\fX$.
	Further, since $\mathbf{p}$ is assumed to define an ergodic Markov chain, it holds for every $j$ that there is some $i$ with $\hat{c}_{i,j} \neq 0$.
	Thus, $\mathfrak{u}$ is finite and the set $\mathfrak{X}$ is compact.
	Consequently, there exists $x^* \in \mathfrak{X}$ with
	\begin{align}
		\hat{\mathfrak{m}} & = \inf_{x\in \fX}  \max_{i= 1,\ldots,2K} \Bigl\{ \frac{1}{x_i} + \sum_{j=1}^{2K} \hat{c}_{i,j} x_j \Bigr\} = \max_{i= 1,\ldots,2K} \Bigl\{ \frac{1}{x_i^*} + \sum_{j=1}^{2K} \hat{c}_{i,j} x_j^* \Bigr\}.
	\end{align}
	This concludes the proof.
\end{proof}
\begin{lemma}\label{lem: SfreeEstimate}
	With $\bS$ as in \eqref{eq: Def_S_Dilation}, there exists an explicit $\mathfrak{E}>0$ with
	\begin{align}
		\Vert \bS_{\free} \Vert \leq \hat{\mathfrak{m}} + \frac{1}{d}\mathfrak{E}.\nonumber
	\end{align}
	Moreover, it holds that $\mathfrak{E} \leq n^{-1}d \mathfrak{u} \mathfrak{c}_2  + C \mathfrak{u}\Psi(\mathbf{p})\hat{\alpha}_{\min}^{-4}$ for some absolute constant $C>0$.
\end{lemma}
\begin{proof}
	Let $x^*\in \bbR^{2K}$ be the vector provided by Lemma \ref{lem:  nhat_minimizer_estimate} and let $\bW$ be the $2d\times 2d$ diagonal matrix with $\bW_{i,i} = x_a^*$ and $\bW_{i+d, i+d} = x_{a+K}^*$ for any $i\in \cV_a$ and $a\in \{1,\ldots,K \}$.
	Then, due to \eqref{eq: Lehner} and the fact that eigenvalues are dominated by the operator norm
	\begin{align}
		\Vert \bS_{\free} \Vert \leq \lambda_{\max}\Bigl(\bW^{-1} + \bbE[\bS \bW \bS]  \Bigr) \leq \Vert \bW^{-1} + \bbE[\bS \bW \bS] \Vert_{\op}.\label{eq: SfreeBMC}
	\end{align}
	For any $i,j \in \{1,\ldots,2d \}$ the $(i,j)$th entry of $\bbE[\bS \bW \bS]$ is given by
	\begin{align}
		 & \bbE[\bS \bW \bS]_{i,j} = \sum_{k= 1}^{2d} \bW_k\bbE[\bS_{i,k} \bS_{k,j}]\label{eq: bSbZbS_ij_exact} \\
		 & =
		\begin{cases}
			0                                               & \text{ if }(i\leq d \text{ and }j>d)\text{ or }(i>d \text{ and }j\leq d), \\
			\sum_{k=1}^d \bW_{k+d,k+d}\Cov(\bM)_{ik,jk}     & \text{ if } i\leq d \text{ and }j\leq d,                                  \\
			\sum_{k=1}^d \bW_{k,k}\Cov(\bM)_{k(i-d),k(j-d)} & \text{ if }i>d \text{ and }j>d.
		\end{cases} \nonumber
	\end{align}
	We next study the diagonal and off-diagonal entries separately, starting with the diagonal.

	Recall the estimate \eqref{eq: Lem_Covijij} on $\Cov(\bM)_{ik,ik}$ and note that, since $\hat{\alpha}_a = \# \cV_a / d$,
	\begin{align}
		\Bigl\lvert \Cov(\bM)_{ik,ik} - \frac{1}{d} \frac{\pi_a \mathbf{p}_{a,b}}{\hat{\alpha}_a \hat{\alpha_b}} \Bigr\rvert \leq \frac{1}{d^2}\bigl(\frac{d}{n} \mathfrak{c}_2 + \frac{3}{d}\mathfrak{c}_2^2 + 2\mathfrak{c}_3 + \frac{2}{d}\mathfrak{c}_2 \mathfrak{d}\bigr) \label{eq: Covik}
	\end{align}
	for every $i\in \cV_a$ and $k\in \cV_b$.
	Recall from \eqref{eq: Def_chat} and that $\hat{c}_{a,b}= 0$ if $\max\{a,b \}\leq K$ or if $\min\{a,b \}>K$.
	Further, note that $\max_{i=1,\ldots,d} \bW_{i,i} = \max_{b=1,\ldots,2K} x_b^*\leq \mathfrak{u}$ by Lemma \ref{lem: nhat_minimizer_estimate}.
	Hence, grouping terms along the clusters in \eqref{eq: bSbZbS_ij_exact} and substituting \eqref{eq: Covik} yields that
	\begin{align}
		\Bigl\lvert \bbE[\bS\bW\bS]_{i,i} - \sum_{b=1}^{2K} \hat{c}_{a,b}x_b^* \Bigr\rvert \leq \frac{\mathfrak{u}}{d}\bigl(\frac{d}{n} \mathfrak{c}_2 + \frac{3}{d}\mathfrak{c}_2^2 + 2\mathfrak{c}_3 + \frac{2}{d}\mathfrak{c}_2 \mathfrak{d}\bigr)\label{eq: DiagSZS}
	\end{align}
	for every $i,a$ with $i\in \cV_a$ and $a \leq K$ or $i\in\cV_{a-K}$ and $a > K$.

	Now suppose that $i\neq j$.
	Then, using the estimate \eqref{eq: Lem_Covijkm} in \eqref{eq: bSbZbS_ij_exact} with the fact that there are precisely $2$ values of $k$ with $k\in \{i,j \}$ and $d-2\leq d$ remaining values,
	\begin{align}
		\lvert \bbE[\bS \bW \bS]_{i,j} \rvert & \leq \frac{2\mathfrak{u}}{d^2}(\frac{3}{d}\mathfrak{c}_2^2 + 2 \mathfrak{c}_3 + \frac{2}{d}\mathfrak{c}_2\mathfrak{d}) + \frac{\mathfrak{u}}{d^2}(3\mathfrak{c}_2^2 + 2 \mathfrak{c}_2 \mathfrak{d})\label{eq: OffDiagSZS} \\
		                                      & = \frac{\mathfrak{u}}{d^2}\Bigl((3\mathfrak{c}_2 + 4 \mathfrak{c}_3 + 2 \mathfrak{c}_2^2 \mathfrak{d}) + \frac{1}{d}(6\mathfrak{c}_2^2 + + 4\mathfrak{c}_2\mathfrak{d}) \Bigr). \nonumber
	\end{align}

	Let us now split $\bW^{-1} + \bbE[\bS\bW\bS] = \bD + \bR$ where $\bD$ is the matrix containing all diagonal entries and $\bR$ is the matrix containing all off-diagonal entries.
	Then, since the operator norm of a diagonal matrix is the greatest absolute value of its elements it follows from \eqref{eq: DiagSZS} that
	\begin{align}
		\Vert \bD \Vert_{\op} \leq \hat{\mathfrak{m}} +\frac{\mathfrak{u}}{d}\bigl(\frac{d}{n} \mathfrak{c}_2 + \frac{3}{d}\mathfrak{c}_2^2 + 2\mathfrak{c}_3 + \frac{2}{d}\mathfrak{c}_2 \mathfrak{d}\bigr).\label{eq: FreeDEstimate}
	\end{align}
	Further, using that the operator norm of a block diagonal matrix is the maximum of the operator norms of the blocks as well as the fact that the operator norm is dominated by the Frobenius norm, it follows from \eqref{eq: OffDiagSZS} that
	\begin{align}
		\Vert \bR \Vert_{\op}  = \Bigl\Vert \begin{pmatrix}
			                                    \bR_1 & 0     \\
			                                    0     & \bR_2
		                                    \end{pmatrix} \Bigr\Vert_{\op}\label{eq: FreeREstimate}
		\leq \frac{\mathfrak{u}}{d}(3\mathfrak{c}_2^2 + 4\mathfrak{c}_3 + 2\mathfrak{c}_2 \mathfrak{d} ) + \frac{\mathfrak{u}}{d^2}(6\mathfrak{c}_2^2 + 4\mathfrak{c}_2\mathfrak{d}).
	\end{align}
	Hence, since $\Vert \bW^{-1} + \bbE[\bS\bW\bS]\Vert_{\op}  \leq \Vert  \bD \Vert_{\op} +   \Vert \bR \Vert_{\op}$, the combination of \eqref{eq: SfreeBMC}, \eqref{eq: FreeDEstimate}, and \eqref{eq: FreeREstimate} yields the desired result with
	\begin{align}
		\mathfrak{E} \de \mathfrak{u}\bigl(\frac{d}{n}\mathfrak{c}_2  +3\mathfrak{c}_2^2 + 6\mathfrak{c}_3 + 2\mathfrak{c}_2\mathfrak{d}  \bigr) + \frac{\mathfrak{u}}{d}\bigl(9\mathfrak{c}_2^2 + 6\mathfrak{c}_2\mathfrak{d}\bigr). \label{eq: Def_Efrak}
	\end{align}
	The claimed upper bound on $\mathfrak{E}$ further follows from Lemma \ref{lem: FrakParamEstimates}.
\end{proof}
The combination of Lemmas \ref{lem: RefinedParamEstimates} and \ref{lem: SfreeEstimate} provides close--to--optimal estimates on the parameters of $\bS$.
This yields the following concentration inequality:
\begin{proposition}\label{prop: BMC_Explicit_Concentration}
	There exists an absolute constant $c>0$ such that, for every $0< \delta \leq 1$ and $x>0$, the matrix $\bM = \sqrt{d/n}(\hat{\bN} - \bbE[\hat{\bN}])$ satisfies
	\begin{align}
		\bbP\bigl(\Vert \bM \Vert_{\op} \geq (1+\delta)(\hat{\mathfrak{m}} + d^{-1}\mathfrak{C})   + c\cE(x) \bigr) \leq (d+1) (1+\delta)^{-x}.  \nonumber
	\end{align}
	where $\cE(x) \de (dn^{-1}\Psi(E)^4\mathfrak{c}_1^2)^{1/6}x^{2/3} + d^{1/2}
		n^{-1/2}\Psi(E) x + (d^{-1}\mathfrak{v}\mathfrak{g})^{1/4} (x^{1/2} + \ln(d+1)^{3/4})$ with $\hat{\mathfrak{m}}, \mathfrak{E},\mathfrak{g}, \mathfrak{v}$, and $\mathfrak{c}_1$ as in \eqref{eq: Def_nhat}, \eqref{eq: Def_Efrak}, \eqref{eq: Def_frakg}, \eqref{eq: Def_frakv}, and \eqref{eq: Def_frakc1}, respectively.
\end{proposition}
\begin{proof}
	Apply Proposition \ref{prop: MarkovMarkovTail} with Lemmas \ref{lem: RefinedParamEstimates} and \ref{lem: SfreeEstimate} and use that $\Vert \bM \Vert = \Vert \bS \Vert$.
\end{proof}
\begin{lemma}\label{lem: UpperMAssymptotic}
	Adopt the notation and assumptions of Theorem \ref{thm: BMC_Noise_Free}.
	Then, for any $\varepsilon >0$
	\begin{align}
		\lim_{d\to \infty}\bbP(\Vert \bM \Vert_{\op} > \mathfrak{m} + \varepsilon ) = 0.
	\end{align}
\end{lemma}
\begin{proof}
	A comparison of \eqref{eq: Def_m} and \eqref{eq: Def_nhat} shows that $\lim_{d\to \infty} \hat{\mathfrak{m}} = \mathfrak{m}$; recall that we assume that $\lim_{d\to \infty}\#\cV_a/d = \alpha_a$ and that $\mathbf{p}$ is kept fixed.
	Further, the parameter $\mathfrak{E}$ defined in \eqref{eq: Def_Efrak} remains bounded as $d$ tends to infinity, so $\mathfrak{E}/d \to 0$.
	Now, using Proposition \ref{prop: BMC_Explicit_Concentration} with $\delta$ fixed at a sufficiently small value and subsequently taking $x$ a large multiple of $\ln(d)$, a direct calculation using the assumption $\lim_{d \to \infty} d\ln(d)^4/n = 0$ yields the result.
\end{proof}
\begin{proof}[Proof of Theorem \ref{thm: BMC_Noise_Free}]
	Combine Proposition \ref{prop: SingvalLowerBound} and Lemma \ref{lem: UpperMAssymptotic}.
\end{proof}

\end{document}